\documentclass{conm-p-l}
\usepackage{mathrsfs}
\usepackage{amsmath,amssymb}

\def\bbR{\mathrm{I\!R}}
\def\bbC{{\mathchoice {\setbox0=\hbox{$\displaystyle\mathrm{C}$}
\hbox{\hbox to0pt{\kern0.4\wd0\vrule height0.9\ht0\hss}\box0}} 
{\setbox0=\hbox{$\textstyle\mathrm{C}$}\hbox{\hbox 
to0pt{\kern0.4\wd0\vrule height0.9\ht0\hss}\box0}} 
{\setbox0=\hbox{$\scriptstyle\mathrm{C}$}\hbox{\hbox 
to0pt{\kern0.4\wd0\vrule height0.9\ht0\hss}\box0}} 
{\setbox0=\hbox{$\scriptscriptstyle\mathrm{C}$}\hbox{\hbox 
to0pt{\kern0.4\wd0\vrule height0.9\ht0\hss}\box0}}}}

\newcommand{\ein}{\mathrm{e}}
\newcommand{\sca}{\mathrm{s}}
\newcommand{\w}{^{\phantom i}}
\newcommand{\nnh}{\hskip-1pt}
\usepackage{color}

\def\blue{\color{blue}}
\newcommand{\ric}{\mathrm{r}}
\newcommand{\sym}{\mathrm{b}}

\def\cwedge{\bigcirc\kern-1.07em\wedge\ }
\newcommand{\hyp}{\hskip.5pt\vbox
{\hbox{\vrule width2.5ptheight0.5ptdepth0pt}\vskip2pt}\hskip.5pt}

\def\vt{{\tau\hskip-4.55pt\iota\hskip.6pt}} 
\def\evt{{\tau\hskip-3.55pt\iota\hskip.6pt}} 

\def\hi{\varphi}
\def\ay{\varphi}
\def\sg{\kappa}

\def\khi{\chi}
\def\wf{f}
\def\hs{\hskip.7pt}
\def\hh{\hskip.4pt}
\def\nh{\hskip-.7pt}
\def\hn{\hskip-.4pt}
\def\w{^{\phantom i}}

\def\tM{\hskip3pt\widetilde{\hskip-3ptM\nh}\hs}
\def\hM{\hskip3pt\widehat{\hskip-3ptM\nh}\hs}
\def\tY{\hskip1.7pt\widetilde{\hskip-1.7ptY\hskip-1.3pt}\hskip1.3pt}
\def\hY{\hskip1.7pt\widehat{\hskip-1.7ptY\hskip-1.9pt}\hskip1.9pt}
\def\bY{\hskip.2pt\overline{\hskip-.2ptY\hskip-1.3pt}\hskip1.3pt}
\def\tQ{\hskip2.2pt\widetilde{\hskip-2.2ptQ\nh}\hs}
\def\hQ{\hskip3pt\widehat{\hskip-3ptQ\nh}\hs}
\def\bQ{\hskip3pt\overline{\hskip-3ptQ\nh}\hs}
\def\bna{\hs\overline{\nh\nabla\nh}\hs}
\def\tna{\hs\widetilde{\nh\nabla\nh}\hs}
\def\hna{\hs\widehat{\nh\nabla\nh}\hs}
\def\td{\hskip2pt\widetilde{\hskip-2ptd\hskip-1pt}\hskip1pt}
\def\hd{\hskip2pt\widehat{\hskip-2ptd\hskip-.3pt}\hskip.3pt}
\def\bg{\hskip1.2pt\overline{\hskip-1.2ptg\hskip-.3pt}\hskip.3pt}
\def\tg{\hskip2pt\widetilde{\hskip-2ptg\hskip-1pt}\hskip1pt}
\def\hg{\hskip2pt\widehat{\hskip-2ptg\hskip-.3pt}\hskip.3pt}
\def\brc{\hskip.2pt\overline{\hskip-.2pt\mathrm{r}\nh}\hs}

\def\bei{\hskip.2pt\overline{\hskip-.2pt\mathrm{e}\nh}\hs}
\def\tei{\tilde{\mathrm{e}}}
\def\hei{\hat{\mathrm{e}}}
\def\bsc{\hskip.2pt\overline{\hskip-.2pt\mathrm{s}\nh}\hs}
\def\tsc{\tilde{\mathrm{s}}}
\def\hsc{\hat{\mathrm{s}}}
\def\bde{\overline{\nh\Delta\nh}}
\def\tde{\widetilde{\nh\Delta\nh}}
\def\hde{\widehat{\nh\Delta\nh}}

\def\ta{\hskip1.2pt\widetilde{\hskip-1.2pt\alpha\hskip-.3pt}\hskip.3pt}
\def\ha{\hskip1.2pt\widehat{\hskip-1.2pt\alpha\hskip-.3pt}\hskip.3pt}
\def\tb{\hskip1.2pt\widetilde{\hskip-1.2pt\beta\hskip-.3pt}\hskip.3pt}
\def\hb{\hskip1.2pt\widehat{\hskip-1.2pt\beta\hskip-.3pt}\hskip.3pt}

\def\gm{\gamma}
\def\vg{\varGamma}
\def\ve{\varepsilon}
\def\h{\delta}
\def\kp{\theta}

\newtheorem{thm}{Theorem}[section]
\newtheorem{lem}[thm]{Lemma}

\theoremstyle{definition}

\theoremstyle{remark}
\newtheorem{rem}[thm]{Remark}

\numberwithin{equation}{section}

\begin{document}

\title[Weakly Ein\-stein con\-for\-mal products]{Weakly Ein\-stein 
con\-for\-mal products}

\author[A.\,Derdzinski]{ Andrzej Derdzinski}
\address[Andrzej Derdzinski]{Department of Mathematics, 
The Ohio State University\hskip-1pt\\
Columbus, OH 43210, USA}
\email{andrzej@math.ohio-state.edu}

\author[J.\,H.\,Park]{JeongHyeong Park}
\address[JeongHyeong Park]{Department of Mathematics, 
Sungkyunkwan University, Suwon, 16419, Korea}
\email{parkj@skku.edu}

\author[W\nnh.\,Shin]{Wooseok Shin}
\address[Wooseok Shin]{Department of Mathematics,\nnh\ Sungkyunkwan
University,\nnh\ Suwon,\nnh\ 16419, Korea}
\email{tlsdntjr@skku.edu}


\begin{abstract}\hskip-2.7ptOne says that a Riemannian
four-man\-i\-fold is \emph{weakly 
Ein\-stein} if the three-in\-dex  contraction of its curvature tensor against
itself equals a function times the metric. Since this includes all
four-man\-i\-folds that are Ein\-stein, or con\-for\-mal\-ly flat
and scalar-flat, the term \emph{proper} may be used for weakly Ein\-stein
manifolds (or metrics) not belonging to the latter two classes.
We establish two clas\-si\-fi\-ca\-tion-type results about proper weakly
Einstein metrics con\-for\-mal to Riemannian
products. This includes constructions of new examples, among them -- some of
(local) co\-ho\-mo\-ge\-ne\-i\-ty two, in contrast with the two previously 
known narrow classes of examples, having co\-ho\-mo\-ge\-ne\-i\-ty zero and 
one. We also exhibit a simple coordinate description of one of the known 
examples, the EPS space, which shows that it is a con\-for\-mal product
and constitutes a single lo\-cal-homo\-thety type. Finally, we prove that
there exist no proper weakly Einstein manifolds with harmonic curvature.
\end{abstract}

\maketitle

\section{Introduction}\label{in}
\setcounter{equation}{0}
One calls a Riemannian four-man\-i\-fold \emph{weakly Ein\-stein\/}
\cite[p.\,112]{EPS-RM} 
when the triple contraction of its curvature tensor against itself is a 
functional multiple of the metric. 
According to formula (\ref{iff}) below, this follows if the 
four-man\-i\-fold in question is Ein\-stein, or con\-for\-mal\-ly flat 
and scalar-flat. Weakly Ein\-stein manifolds (or metrics) not belonging to
these two classes will 
be referred to as \emph{proper}. 

Known examples of proper weakly Einstein manifolds consist of the EPS space
\cite[p.\,602]{EPS-MS}, and a 
very narrow class of 
K\"ah\-ler surfaces \cite[Sect.\,12]{derdzinski-euh-kim-park}. In
Sect.\,\ref{ep} we provide a simple coordinate description of the EPS space,
which shows that it represents \emph{a single 
lo\-cal-homo\-thety type}. This is a surprising conclusion, since the
original description \cite{EPS-MS} of the EPS space realizes it in the form
of left-in\-var\-i\-ant metrics on Lie groups corresponding to
infinitely many non-iso\-mor\-phic Lie algebras.

As another consequence of its coordinate description, the EPS space 
admits two kinds of 
warp\-ed-prod\-uct decompositions, with the base/fibre dimensions $\,2+2$
(a unique one) and $\,3+\nh1\,$ (infinitely many), which raises the
question whether there exist other
proper weakly Einstein warped products and -- more generally -- 
\emph{con\-for\-mal products\/} (that is, man\-i\-folds/me\-trics 
con\-for\-mal to Riemannian products).

The present paper answers this question in the affirmative, which
at the same time leads to new examples of proper weakly Ein\-stein
four-man\-i\-folds. Those constructed in Sect.\,\ref{gc} have,
as a consequence of Remark~\ref{sglbd}, local co\-ho\-mo\-ge\-ne\-i\-ty two,
which makes them fundamentally different (see Sect.\,\ref{pn}) from all the
currently known examples, mentioned above.

Our main result consists of two clas\-si\-fi\-ca\-tion-type
theorems for proper weakly
Ein\-stein con\-for\-mal products. We start from some constructions of
examples, phras\-ed as Theorems~\ref{exatt} and~\ref{exaot}, and
followed by four sections in which we show that the
resulting metrics are in fact weakly Ein\-stein, and prove the existence of 
the geometric structures used in the constructions. The classification 
itself is provided by Theorems~\ref{cpsfm} and~\ref{cptho}, stating that,
locally, at generic points, the constructions just mentioned yield
all proper weakly Ein\-stein con\-for\-mal products, except 
possibly some nongeneric ones of type $\,3+1$.

Sect.\,\ref{pn} clarifies how 
the new examples of proper weakly Ein\-stein four-man\-i\-folds, 
arising from Theorems~\ref{exatt} and~\ref{exaot}, 
differ from the previously known ones, and where they stand 
within the classification, established in 
\cite{derdzinski-park-shin}, of algebraic curvature tensors having the weakly 
Ein\-stein property. 
Finally, in Sect.\,\ref{hc}, we combine our classification theorems with
some results of \cite{derdzinski-25} to show that a proper weakly
Ein\-stein Riemannian four-man\-i\-fold cannot have harmonic curvature.

\section{Preliminaries}\label{pr}
\setcounter{equation}{0}
Manifolds are by definition connected, all manifolds and tensor fields
smooth; 
$\,R,W\nnh\nh,\ric,\ein\,$ stand for the curvature, Weyl, 
Ric\-ci and Ein\-stein tensors of a Riemannian metric $\,g\,$ in dimension 
$\,n$, and $\,\sca\,$ for its scalar curvature, $\,R\,$ having the sign such
that $\,\ric_{ij}\w=g^{pq}\nh R_{ipjq}\w$. Thus, $\,\ein=\ric-\sca g/n$,
while, if $\,n\ge3$, 
\[
\begin{array}{l}
W_{\!ijpq}\w=R_{ijpq}\w
-\displaystyle{\frac1{n-2}}\,(g_{ip}\w\ric_{jq}\w
+g_{jq}\w\ric_{ip}\w-g_{jp}\w\ric_{iq}\w
-g_{iq}\w\ric_{jp}\w)\\
\hskip64.5pt
+\,\displaystyle{\frac{\sca}{(n-1)(n-2)}}
\hs(g_{ip}\w g_{jq}\w-g_{jp}\w g_{iq}\w)\hh.
\end{array}
\]
As already mentioned in the Introduction, a Riemannian four-man\-i\-fold
$\,(M\nh,g)\,$ is said to be weakly Ein\-stein when the three-in\-dex
contraction of its curvature tensor $\,R\,$ equals some function $\,\phi\,$ 
times the metric (in coordinates: $\,R_{ikpq}\w R_j\w{}^{kpq}\nnh
=\phi g_{ij}\w$). According to \cite[formula (4.7)]{derdzinski-euh-kim-park},
\begin{equation}\label{iff}
(M\nh,g)\,\mathrm{\ is\ weakly\ Ein\-stein\ if\ and\ only\ if\ 
}\,\,6\hh W\nnh\nh\ein=-\sca\hh\ein\hh,
\end{equation}
with $\,W\nnh\nh\ein\,$ denoting here the usual action of 
algebraic curvature tensors on symmetric $\,(0,2)\,$ tensors: 
$\,[W\nnh\nh\ein]_{ij}\w=\hs W_{\!ipjq}\w\ein\hs^{pq}\nh$.

For a torsion-free connection $\nabla$ with the Ricci tensor $\ric$ on a
manifold $M$, every vector field $v$ satisfies the \emph{Boch\-ner identity} 
\begin{equation}\label{bch}
\delta\nabla\nh v\,=\,\ric(\cdot,v)\,+\,d\delta v ,
\end{equation}
$\delta$ being the divergence. In fact, the coordinate form 
$\,v^{\hh k}{}\hskip-2.7pt_{,ik}\w
=\ric_{ik}\w v^{\hh k}+v^{\hh k}{}\hskip-2.7pt_{,ki}\w$ of
(\ref{bch}) arises via contraction from
the Ricci identity
$\,v^{\hh p}{}\hskip-2.7pt_{,ij}\w-v^{\hh p}{}\hskip-2.7pt_{,ji}\w
=R_{ijk}\w{}^pv^{\hh k}\nh$. 
For functions $\,\alpha:M\to\bbR\,$ and $\,\hi:M\to\bbR\smallsetminus\{0\}\,$ 
on a Riemannian manifold $\,(M\nh,g)$,
\begin{equation}\label{dqe}
\begin{array}{rl}
\mathrm{a)}&
dQ=2[\nabla\nh d\alpha](\nabla\nh\alpha,\cdot)\mathrm{,\ \ where\ \ }\,Q
=g(\nabla\nh\alpha,\nnh\nabla\nh\alpha),\\
\mathrm{b)}&\hi^3\nnh\Delta\hn\hi^{-1}\hn
=\,2g(\nabla\nnh\hi,\nnh\nabla\nnh\hi)\,-\,\hi\Delta\hn\hi\hh,
\end{array}
\end{equation}
(\ref{dqe}-a) obvious since, in local coordinates,
$\,Q\nh_{,i}\w=(\alpha^{,k}\alpha_{,k}\w)_{,i}\w
=2\hh\alpha\nh_{,ik}\w\alpha^{,k}\nh$. 
In an oriented Riemannian four-man\-i\-fold $\,(M\nh,g)$, we denote by 
$\,W^\pm\nnh\nh:\Lambda\nnh^\pm\nnh\hn M\to\Lambda\nnh^\pm\nnh\hn M\,$
the vec\-tor-bun\-dle mor\-phisms arising when one restricts the Weyl
tensor $\,W\nnh$ acting on bi\-vec\-tors to the sub\-bun\-dles
$\,\Lambda\nnh^\pm\nnh\hn M\,$ of self-dual and anti-self-dual bi\-vec\-tors.
\begin{rem}\label{denot}We always denote by $\,\nabla\nh,\bna\nh,\hna\nh,\tna$ 
the gradient operators and Le\-vi-Ci\-vi\-ta connections of the metrics 
$\,g,\bg,\hg,\tg$, and 
similarly for the scalar curvatures $\,\sca$, Ric\-ci/Ein\-stein tensors
$\,\ric,\,\ein\,$ and 
La\-plac\-i\-ans $\,\Delta$. We will also write
\begin{equation}\label{cnv}
\begin{array}{l}
\bY\nnh\nh=\hs\bde\hs\hi\hh,\qquad
\hY\nnh\nh=\hs\hde\hs\hi\hh,\qquad
\tY\nnh\nh=\hs\tde\hs\hi\hh,\\
\bQ=\bg(\bna\nh\hi,\nnh\bna\nh\hi)\hh,\hskip7pt
\hQ=\hg(\hna\nh\hi,\nnh\hna\nh\hi)\hh,\hskip7pt
\tQ=\tg(\tna\nh\hi,\nnh\tna\nh\hi)\hh.
\end{array}
\end{equation}
Given a product metric $\,\bg=\hg+\tg\,$ on a product manifold 
$\,\hM\times\tM\nh$, we interpret
$\,\hY\nnh\nh,\tY\nnh\nh,\hQ,\tQ\,$ and
$\,\hna\nh d\hi,\tna\nh d\hi\,$ as the results of ``partial'' 
operations, involving the restrictions of
$\,\hi:\hM\times\tM\to\bbR\,$ to sub\-man\-i\-folds of the
form $\,\hM\times\{z\}\,$ and $\,\{y\}\times\tM\nh$, where
$\,(y,z)\in\hM\times\tM\nh$, so that
$\,\bde\hi=\hY\nnh+\tY\,$ and 
$\,\bg(\bna\nh\hi,\nnh\bna\nh\hi)=\hs\hQ+\tQ$.
\end{rem}
\begin{rem}\label{kahsf}
On a K\"ah\-ler surface,
$\,W^+\nnh\nh:\Lambda\nnh^+\nnh\nh M\to\Lambda\nnh^+\nnh\nh M\nh$, for
the standard o\-ri\-en\-ta\-tion, has the 
spectrum $\,(\sca/6,-\sca/\nh12,-\sca/\nh12)$, the eigen\-val\-ue function 
$\,\sca/6\,$ being realized by the K\"ah\-ler form treated as a 
bi\-vec\-tor \cite[p.\,459]{derdzinski-00}.
\end{rem}
\begin{rem}\label{prsrf}
In a Riemannian product $\,(M\nh,g)\,$ of two surfaces, which is locally
K\"ah\-ler for two complex structures corresponding to opposite orientations,
Remark~\ref{kahsf} implies that 
both $\,W^\pm\nnh\nh:\Lambda\nnh^\pm\nnh M\to\Lambda\nnh^\pm\nnh M\,$ have the 
same spectrum $\,(\sca/6,-\sca/\nh12,-\sca/\nh12)$. Thus, con\-for\-mal flatness
of $\,(M\nh,g)\,$ is equivalent to the vanishing of its scalar curvature.
\end{rem}
\begin{rem}\label{noprd}
A proper weakly Einstein manifold cannot be a Riemannian product. Namely, 
the triple contraction of the curvature tensor against itself 
behaves ``multiplicatively'' under Riemannian products, is equal to $\,0\,$ in 
dimension one, and to $\,2K^2\nh g\,$ for a surface metric $\,g\,$ with
Gauss\-i\-an curvature $\,K$. Thus, for a $\,3+1\,$ (or, $\,2+2$) product,
being weakly Einstein is equivalent to flatness or, respectively, to being
Ein\-stein or con\-for\-mally flat. (The last two options correspond to
equal/op\-po\-site Gauss\-i\-an curvatures of the factor metrics, cf.\
Remark~\ref{prsrf}.)
\end{rem}
\begin{rem}\label{sezro}
Vanishing of $\,\sca\hh\ein\,$ implies, even without assuming 
real-an\-a\-lyt\-ic\-i\-ty, that one of $\,\sca,\ein\,$ is identically zero
\cite[Remark~2.1]{derdzinski-euh-kim-park}. Thus, any conformally flat weakly
Einstein manifold has, by (\ref{iff}), either $\,\ein=0$, or
$\,W\nnh\nh=\hs0\,$ and $\,\sca=0$.
\end{rem}
\begin{rem}\label{genrc}Given several self-ad\-joint en\-do\-mor\-phisms 
$\,A_j\w$ of vector bundles over a manifold $\,M\,$ endowed with Riemannian
fibre metrics, we say that a point $\,x\in M\,$ is
\emph{generic relative to the spectra of\/} (all) $\,A_j\w$ if, for some 
neighborhood $\,\,U\hh$ of $\,x$, 
every $\,A_j\w$ restricted to $\,\,U\hh$ has a constant number of distinct 
eigen\-val\-ues. It is clear that such generic points form a dense open subset
$\,M'$ of $\,M\nh$, and on each connected component of $\,M'$ the 
eigen\-val\-ues of each $\,A_j\w$ constitute smooth \emph{eigen\-val\-ue
functions}, while the corresponding eigen\-spaces form smooth {\it 
eigen\-space sub\-bundles\/} of the vector bundles in question.
\end{rem}
\begin{rem}\label{bideg}Eigen\-vec\-tors of a linear en\-do\-mor\-phism 
$\,D\hs$ of a real vector space $\,\mathcal{E}\nh$, corresponding to mutually 
different eigen\-val\-ues, are linearly independent (or else in 
a $\,D$-in\-var\-i\-ant subspace $\,\mathcal{E}\hh'\hskip-2pt$ spanned by 
such eigen\-vec\-tors, $\,D\hs$ would have 
more than $\,\dim\hs\mathcal{E}\hh'\hskip-2pt$ eigen\-val\-ues). 
Consequently, the space of polynomial functions 
$\,\mathcal{E}\nh\times\mathcal{E}\nh\to\bbR\,$ is the direct sum of 
the sub\-spaces $\,\mathcal{P}\hskip-3pt_{i,j}\w$, each formed by
polynomials of bi\-de\-gree $\,(i,j)$. Namely, each 
$\,\mathcal{P}\hskip-3pt_{i,j}\w$ is the $\,(i+qj)$-eigen\-space of the
directional derivative operator $\,d_w\w$, for the linear vector field
$\,w\,$ with $\,w_{(y,z)}\w=(y,qz)$, where $\,q\,$ is any fixed irrational
number. 
\end{rem}
\begin{rem}\label{divis}If $\,P\nnh,S\,$ are $\,C^\infty\nnh$
functions on a manifold 
$\,M\,$ and $\,0\,$ is a regular value of $\,S$, while $\,P\nh=0\,$ 
along $\,\varSigma=S^{-\nnh1}\nh(0)$, then 
$\,P\nh/S:M\smallsetminus\varSigma\to\bbR$ has a $\,C^\infty\nnh$
extension to
$\,M\nh$. In fact, identifying a neighborhood of $\,z\in\varSigma\,$ with
a convex neighborhood $\,\,U\,$ of $\,z=0\,$ in $\,\bbR\hn^n\nh$,
$\,n=\dim M\nh$, so that 
$\,S\,$ is the coordinate function $\,x^1\nh$, one has
$\,P(x^1\nh,\dots,x^n)=P(x^1\nh,\dots,x^n)-P(0,x^2\nh,\dots,x^n)=
x^1\hskip-3pt\int_0^1[\partial\hn_1\w P](tx^1\nh,x^2\nh,\dots,x^n)\,dt$.  
\end{rem}

\section{Warped products}\label{wp}
\setcounter{equation}{0}
The \emph{warped product\/} of Riemannian manifolds $\,(\varSigma,\gamma)\,$
and $\,(\varPi\nh,\h)\,$ with the \emph{warping function\/} 
$\,\wf:\varSigma\to(0,\infty)\,$ is the Riemannian manifold
\begin{equation}\label{war}
(M\nh,g)\,=\,(\varSigma\times\varPi,\,\gamma+\nh\wf^2\nh\h),
\end{equation}
$\gamma,\h,\wf\,$ standing here for also the pull\-backs of 
$\,\gamma,\h,\wf\,$ to the product $\,M\nh=\varSigma\times\varPi\nh$. One
calls $\,(\varSigma,\gamma)\,$ the \emph{base\/} and $\,(\varPi\nh,\h)\,$ the 
\emph{fibre\/} of (\ref{war}).
As $\,\gamma+\nh\wf^2\h=\wf^2\hh[\hs\wf^{-\nh2}\hh\gamma+\h\hs]$,
\begin{equation}\label{wrp}
\begin{array}{l}
\mathrm{a\ warped\ product\ is\ just\ a\ Riemannian\
manifold\ con\-for\-mal}\\
\mathrm{to\hs\hh\ a\hs\hh\ Riemannian\hs\hh\ product\hs\hh\ via\hs\hh\
multiplication\hs\hh\ by\hs\hh\ a\hs\hh\ positive}\\
\mathrm{function\hn\ which\hn\ is\hn\ constant\hn\ along\hn\ one\hn\ of\hn\
the\hn\ factor\hn\ manifolds.}
\end{array}
\end{equation}
We will need the easy and well-known observation -- see, e.g., 
\cite[formula (A.3)]{derdzinski-piccione-20} -- is that, in any warped product
with the Ric\-ci tensor $\,\ric$,
\begin{equation}\label{rog}
\begin{array}{rl}
\mathrm{a)}&\mathrm{the\,\ base\,\ and\,\ fibre\,\ factor\,\ distributions\,\
are\hs\ }\,\ric\hyp\mathrm{orthogonal,}\\
\mathrm{b)}&\mathrm{if\hs\ the\hs\ fibre\hs\ dimension\hs\ is\hs\ less\hs\
than\hs\ three,\hs\ all\hs\ nonzero\hs\ vec}\hyp\\
&\mathrm{tors\ tangent\ to\hh\ the\hh\ fibre\hh\ distribution\hn\ are\hn\ eigenvectors\hn\
of\nh\ }\,\ric.
\end{array}
\end{equation}
\begin{rem}\label{wpone}
It is well known \cite[Lemma\,19.2]{derdzinski-00}, 
\cite[Remark 3.1]{derdzinski-piccione-20} that
a one-di\-men\-sion\-al distribution on a Riemannian manifold is, locally, the 
fibre distribution of a warp\-ed-prod\-uct decomposition if and only if it
is spanned, locally, by a Kil\-ling field $\,v\,$ without zeros having 
an in\-te\-gra\-ble orthogonal complement $\,v^\perp\nnh$.
\end{rem}
\begin{rem}\label{nnunq}If the fibre $\,(\varPi\nh,\h)\,$ of (\ref{war}) 
a Riemannian product, for instance, 
$\,(\varPi\nh,\h)=(\varPi'\nh\times\varPi''\nnh,\h'\nnh+\h'')$, then
(\ref{war}) $\,(\varPi\nh,\h)\,$ is, obviously, also a warped product with the
base $\,(\varSigma\times\varPi',\,\gamma+\nh\wf^2\nh\h')\,$ 
and fibre $\,(\varPi''\nnh,\h'')$.

Thus, any warped product having the base and fibre dimensions $\,m,2\,$ and
a flat fibre is also, locally, a warped product with the dimensions $\,m+1\,$
and $\,1$.
\end{rem}
\begin{rem}\label{isowp}By (\ref{wrp}), isometries of 
the fibre act isometrically on the warp\-ed-prod\-uct manifold, 
and so do isometries of the base preserving the warping function.
\end{rem}

\section{The EPS space}\label{ep} 
\setcounter{equation}{0}
The \emph{EPS space\/} \cite[p.\,112]{EPS-RM} is an example of
a proper weakly
Einstein manifold $\,(M\nh,g)\,$ arising as follows: 
$\,M\,$ is a Lie group which carries  
left-in\-var\-i\-ant $\,g$-or\-tho\-nor\-mal vector fields
$\,u_1\w,\dots,u_4\w$, trivializing $\,T\nh M\nh$, and having the  
Lie brackets
\begin{equation}\label{lbr}
\begin{array}{l}
[u_1\w,u_2\w]\hs=\hs au_2\w,\quad[u_1\w,u_3\w]\hs
=\hs-au_3\w-bu_4\w,\quad[u_1\w,u_4\w]\hs
=\hs bu_3\w-\hh au_4\w,\\
{}[u_2\w,u_3\w]\,=\,[u_2\w,u_4\w]\,=\,[u_3\w,u_4]\,=\,0\mathrm{,\nnh\ with\nnh\
constants\nnh\ }\,a\ne0\,\mathrm{\ and\ }\,b.
\end{array}
\end{equation}
The acronym `EPS' was introduced in \cite{arias-marco-kowalski}. 
The Lie algebras $\,\mathfrak{h}\,$ defined by (\ref{lbr})
\begin{equation}\label{inf}
\mathrm{form\ infinitely\ many\ Lie}\hyp\mathrm{algebra\ isomorphism\ types,}
\end{equation}
as pointed out in \cite[Theorem 8.1]{arias-marco-kowalski}. In fact,
the $\hs\mathrm{ad}\hs$ action on
$\hs[\mathfrak{h},\mathfrak{h}]\hn=\hn\mathrm{span}\hn\,(u_2\w,u_3\w,u_4\w)$ by
the co\-set of $\,u_1\w$ spanning
$\,\mathfrak{h}/[\mathfrak{h},\mathfrak{h}]\,$ has the complex
characteristic roots $\,a\,$ and $\,-a\pm b\hh i$, which makes
$\,|\hh b/a|$ an algebraic invariant of $\,\mathfrak{h}$, ranging over 
$\,[\hs0,\infty)$.

From \cite[formula (3.14)]{EPS-MS} it is immediate that the only nonzero
components 
of the Ric\-ci tensor $\,\ric\,$ in the frame $\,u_1\w,\dots,u_4\w$ are equal 
to
\begin{equation}\label{rei}
\mathrm{the\ eigen\-val\-ues\ \ }\,\,\ric\nh_{11}\w=-3a^2\nh,\quad\ric_{22}\w
=a^2\nh,\quad\ric_{33}\w=\ric_{44}\w=-a^2\nh.
\end{equation}
The $\,1$-form $\,\zeta\,$ on $\,M\,$ with
$\,(\zeta(u_1\w),\zeta(u_2\w),\zeta(u_3\w),\zeta(u_4\w))=(2a,0,0,0)\,$
is closed due to (\ref{lbr}), so that, locally,
$\,\nabla\hn\vt=2a\hh u_1\w$ for some function $\,\vt$. It is also
immediate from (\ref{lbr}) that the vector fields
$\,\partial\nh_1\w,\dots,\partial\nh_4\w$ given by
\begin{equation}\label{par}
\begin{array}{l}
2a\hs\partial\nh_1\w=u_1\w,\qquad
2a\hs\partial\hn_2\w=\hs e^{-\evt/2}u_2\w,\\
2a\hs\partial\hn_3\w=\hs e\hs^{\evt/2}[\hs\cos\hs(b\hh\vt/2a)\hs u_3\w\,
+\,\sin\hs(b\hh\vt/2a)\hs u_4\w],\\
2a\hs\partial\nh_4\w
=\hs e\hs^{\evt/2}[\hs-\nnh\hn\sin\hs(b\hh\vt/2a)\hs u_3\w\,
+\,\cos\hs(b\hh\vt/2a)\hs u_4\w]
\end{array}
\end{equation}
all commute. They are thus, locally, the coordinate vector fields for some
local coordinates $\,\vt,\xi,\eta,\zeta$, where the first coordinate may be chosen
equal to our function $\,\vt$, as $\,d\vt\,$ sends 
$\,\partial\nh_1\w,\dots,\partial\nh_4\w$ to $\,1,0,0,0$.
Clearly, in the coordinates $\,\vt,\xi,\eta,\zeta$,
\begin{equation}\label{eps}
\begin{array}{l}
4a^2\nh g\,=\hs\,d\vt^2\hs+\,e^{-\nh\evt}\hn d\xi^2\hs
+\,e\hs^\evt(d\eta^2\hs+\,d\zeta^2)\mathrm{,\hh\ that\hs\
is,\hh\ the\hs\ component}\\
\mathrm{functions\ of\ the\ metric\ }4a^2\nh g\mathrm{\ form\ the\
matrix\ }\mathrm{diag}\hs(1,\,e^{-\nh\evt}\nh,\,e^\evt\nh,\,e^\evt).
\end{array}
\end{equation}
This has two immediate consequences. First,
\begin{equation}\label{htt}
\mathrm{the\ EPS\ space\ represents\ a\ single\
lo\-cal}\hyp\mathrm{homo\-thety\ type.}
\end{equation}
Secondly, the EPS space admits both $\,2+2\,$ and $\,3+1\hs$
warp\-ed-prod\-uct
decompositions (cf.\ the final clause of Remark~\ref{nnunq}), with
$\,\wf\,$ and $\,\h\,$ in (\ref{war}) given by
\begin{equation}\label{fet}
\begin{array}{rl}
\mathrm{a)}&\wf=\hs e\hs^{\evt/2}\mathrm{\ and\
}\,\h=d\eta^2\nh+d\zeta^2\nh\mathrm{,\ or}\\
\mathrm{b)}&\wf=\hs e\hs^{-\evt/2}\mathrm{\ and\ }\,\h=d\xi^2\nh\mathrm{,\
or}\\
\mathrm{c)}&\wf=\hs e\hs^{\evt/2}\mathrm{\ and\ }\,\h=d\hat\zeta^2\nh,
\end{array}
\end{equation}
where $\,\hat\zeta\,$ in (\ref{fet}-c) belongs to a pair $\,\hat \eta,\hat\zeta\,$
arising from $\,\eta,\zeta\,$ via any fixed rotation (so that 
$\,d\hat \eta^2\nh+d\hat\zeta^2\nh=d\eta^2\nh+d\zeta^2$). 
Thus,  even locally, among the warp\-ed-prod\-uct decompositions of
the EPS space,
\begin{equation}\label{ttu}
\begin{array}{rl}
\mathrm{i)}&\mathrm{the\ }\,\,2+2\,\,\mathrm{\ decomposition\ is\ unique,\
and\ given\ by\ \,(\ref{fet}}\hyp\mathrm{a),}\\
\mathrm{ii)}&\mathrm{there\ are\ infinitely\ many\ decompositions\ of\
the\ type\ }\,3+\nh1,\\
\end{array}
\end{equation}
`unique' meaning \emph{up to multiplying\/ $\,\wf\,$ and $\,\h\,$ by
mutually inverse constants}. In fact, (\ref{ttu}-i) is obvious from
(\ref{rog}-b) and (\ref{rei}), while (\ref{fet}-c) 
yields (\ref{ttu}-ii).
\begin{thm}\label{isogp}The simply connected complete model of an EPS space
consists of\/ $\,\bbR\hn^4$ with the metric\/ 
$\,g=d\vt^2\nh+e\nh^{-\nh\evt}\hn d\xi^2\nh+e\hs^\evt(d\eta^2\nh+d\zeta^2)\,$ in the
coordinates\/ $\,\vt,\xi,\eta,\zeta$. Its full isom\-e\-try group is 
five-di\-men\-sion\-al, acts on\/ $\,\bbR\hn^4$ transitively, and has the
identity component\/ $\,G\,$ formed by the mappings
\begin{equation}\label{dfb}
\bbR\hn^4\nh=\bbR\hn^2\nnh\nh\times\bbC\ni(\vt,\xi,\eta+i\zeta)
\mapsto(\vt+2c,\,e\hs^c\nh \xi+p,\,e\nh^{-\hn c}\nh w(\eta+i\zeta)+q)
\end{equation}
depending on the 
parameters\/ $\,c,p\in\bbR\,$ and\/ $\,w,q\in\bbC\,$ with\/ $\,|w|=1$.
\end{thm}
\begin{proof}The mappings (\ref{dfb}) are easily seen to form a group acting
on $\,\bbR\hn^4$ effectively and transitively, via isom\-e\-tries of $\,g$, 
as they separately preserve $\,d\vt^2\nh$, $\,e^{-\nh\evt}\hn d\xi^2$ and
$\,e\hs^\evt(d\eta^2\nh+d\zeta^2)$. Our assertion now follows: the dimension
of the full isom\-e\-try group equals $\,5\,$ due to the fact 
that, by 
(\ref{rei}), the Lie algebra of the iso\-tropy
group acting in any given tangent space is of dimension less than $\,2$.
\end{proof}
The conclusion (\ref{htt}) above may seem surprising in the light of
(\ref{inf}). The puzzle is resolved by Theorem~\ref{isogp}: 
$\,G\,$
has infinitely many four-di\-men\-sion\-al sub\-groups, acting on
$\,\bbR\hn^4$ simply transitively, with mutually non\-iso\-mor\-phic Lie
algebras.

For (\ref{fet}-b) and (\ref{fet}-c), the $\,\xi\,$ or $\,\hat\zeta\,$
coordinate vector field is, obviously, a 
Kil\-ling field with an in\-te\-gra\-ble orthogonal complement, as 
required in Remark~\ref{wpone}.

The\hn\ EPS\hn\ space\hn\ has\hn\ natural\hn\ high\-er-di\-men\-sion\-al\hn\
generalizations\hn\ \cite[Remark\hn\ 3.1]{euh-kim-nikolayevsky-park}.

\section{Weakly Ein\-stein curvature tensors}\label{wt}
\setcounter{equation}{0}
In the following theorem $\,\mathcal{T}\hs$ is an oriented Euclidean
$\,4$-space, $\,\sca,\ein,W\nnh$ denote the scalar curvature, Ein\-stein
(trace\-less Ric\-ci) and Weyl tensors of the algebraic curvature tensor in
question, $\,W^\pm\nnh\nh:\Lambda\nnh^\pm\to\Lambda\nnh^\pm$ being the
restrictions of $\,W\nnh$ to the spaces $\,\Lambda\nnh^\pm$ of self-dual and
anti-self-dual bi\-vec\-tors in $\,\mathcal{T}\nnh$. 
Any positive or\-tho\-nor\-mal basis $\,u_1\w,\dots,u_4\w$ of
$\,\mathcal{T}\hs$ leads to the length $\,\sqrt{2\,}$ orthogonal 
bases of $\,\Lambda\nnh^\pm$ given by
\begin{equation}\label{bas}
u_1\w\wedge u_2\w\pm u_3\w\wedge u_4\w\hh,\quad 
u_1\w\wedge u_3\w\pm u_4\w\wedge u_2\w\hh,\quad 
u_1\w\wedge u_4\w\pm u_2\w\wedge u_3\w\hh,
\end{equation}
so that $\,\Lambda\nnh^\pm$ are both canonically oriented. Cf.\ 
\cite[the lines following (7.4)]{derdzinski-park-shin}.
\begin{thm}\label{weact}
Given a non-Ein\-stein, weakly Ein\-stein algebraic curvature tensor in\/
$\,\mathcal{T}\hs$ such that, respectively,
\begin{equation}\label{sth}
\begin{array}{rl}
\mathrm{a)}&\sca=0\,\mathrm{\ and\ }\,\ein\,\mathrm{\ does\ not\ have\ two\
distinct\ double\ eigen\-val\-ues,\ or}\\
\mathrm{b)}&\sca\ne0\,\mathrm{\ and\ }\,\ein\,\mathrm{\ does\ not\ have\ two\
distinct\ double\ eigen\-val\-ues,\ or}\\
\mathrm{c)}&\ein\,\mathrm{\ has\ two\ distinct\
double\ eigen\-val\-ues,}
\end{array}
\end{equation}
there exists a positive or\-tho\-nor\-mal basis\/ 
$\,u_1\w,\dots,u_4\w$ of\/ $\,\mathcal{T}\nnh$, consisting of eigen\-vec\-tors 
of\/ $\,\ein$, for which the bases\/ {\rm(\ref{bas})} of\/ 
$\,\Lambda\nnh^\pm\nnh\hn$ di\-ag\-o\-nal\-ize\/ $\,W^\pm\nh$, with the
corresponding ordered quadruple and two triples of eigen\-values having 
the respective form
\begin{equation}\label{ari}
\begin{array}{rl}
\mathrm{a)}&(\mu_1\w,\mu_2\w,\mu_3\w,\mu_4\w)\,\mathrm{\ and\
}\,(\pm c_2\w,\pm c_3\w,\pm c_4\w)\hh,\\
\mathrm{b)}&(-\nnh\lambda,-\nh\mu,\mu,\lambda)\,\mathrm{\ and\
}\,(\pm c_2\w-\sca/\nh12,\,\pm c_3\w-\sca/\nh12,\,\pm c_4\w+\sca/6)\hh,\\ 
\mathrm{c)}&(-\nnh\lambda,-\nnh\lambda,\lambda,\lambda)\,\mathrm{\ and\
}\,(\pm c_2\w-\sca/\nh12,\,\pm c_3\w+\xi-\sca/\nh12,\,\pm c_4\w-\xi+\sca/6)\hh,\\
\end{array}
\end{equation}
the parameters\/
$\,\mu_1\w,\mu_2\w,\mu_3\w,\mu_4\w,c_2\w,c_3\w,c_4\w,\sca,\lambda,\mu,\xi
\in\bbR\,$ having
\begin{equation}\label{hvg}
\begin{array}{rl}
\mathrm{i)}&c_2\w\hs+\,c_3\w\hs+\,c_4\w\hs=\,0\,\,\mathrm{\,\ in\hh\ all\hh\ 
cases,\hs\ as\hh\ well\hh\ as}\\
\mathrm{ii)}&\mu_1\w\hn+\hh\mu_2\w\hn+\hh\mu_3\w\hn+\hh\mu_4\w\hn=\hh\sca\hh
=\hh0\,\,\mathrm{\ in\ case\ (\ref{ari}}\hyp\mathrm{a)},\\
\mathrm{iii)}&\lambda>\mu\ge0\ne\sca\,\mathrm{\ in\
(\ref{ari}}\hyp\mathrm{b),\ 
}\,\lambda>0\,\mathrm{\ in\ (\ref{ari}}\hyp\mathrm{c).}
\end{array}
\end{equation}
\end{thm}
\begin{proof}See \cite[Theorems 1.3\hs--\hs1.5]{derdzinski-park-shin}.
\end{proof}
\begin{rem}\label{cnvrs}Conversely, according to 
\cite[Theorems 1.3\hs--\hs1.5]{derdzinski-park-shin}, each of the choices 
(\ref{ari}), 
for $\,\mu_1\w,\mu_2\w,\mu_3\w,\mu_4\w,c_2\w,c_3\w,c_4,\sca,\lambda,\mu,\xi\,$ 
with (\ref{hvg}-i) and (\ref{hvg}-ii), defines, via the
Sing\-er\hh-Thorpe theorem 
\cite[Sect.\,1.128]{besse}, an algebraic curvature tensor $\,R\,$
which is weakly
Ein\-stein. Even though \cite{derdzinski-park-shin} assumes (\ref{hvg}-iii), 
this assumption can be relaxed for the following reasons. First, we may
assume that $\,\lambda\ne0\,$ in (\ref{ari}-c), for otherwise $\,R\,$ is
Ein\-stein and hence, by (\ref{iff}), weakly Ein\-stein. Similarly,
the cases $\,\lambda<0\,$ in (\ref{ari}-c) and $\,\mu=\lambda\,$ (or,
$\,\mu=-\lambda$) in (\ref{ari}-b) can be excluded, as the former amounts to
the second part of (\ref{hvg}-iii), with
$\,(u_3\w,u_4\w,u_1\w,u_2\w,-\nh\lambda)\,$ used instead
of $\,(u_1\w,\dots,u_4\w,\lambda)$, and the latter reduces 
(\ref{ari}-b) to (\ref{ari}-c) for $\,\xi=0\,$ (or, respectively,
to (\ref{ari}-c) with $\,(u_1\w,u_3\w,u_4\w,u_2\w)\,$ instead 
of $\,(u_1\w,\dots,u_4\w)$, and $\,\xi=\sca/4$). Also, (\ref{ari}-b) with
$\,\sca=0\,$ is (\ref{ari}-a) for
$\,(\mu_1\w,\mu_2\w,\mu_3\w,\mu_4\w)
=(-\nnh\lambda,-\nh\mu,\mu,\lambda)$.

The next four possibilities amount to four lines, each being a
condition imposed on $\,\lambda,\mu\,$ in (\ref{ari}-b), followed by
a nonuple replacing $\,(u_1\w,\dots,u_4\w,c_2\w,c_3\w,c_4\w,\lambda,\mu)$,
so that the first part of (\ref{hvg}-iii) is satisfied by the new data:
\begin{equation}\label{rpl}
\begin{array}{ll}
-\nnh\lambda<\nnh-\mu\le0\hh,
&(u_1\w,u_2\w,u_3\w,u_4\w,c_2\w,c_3\w,c_4\w,\lambda,\mu)\hh,\\
-\nnh\lambda<\mu\le0\hh,
&(u_1\w,u_3\w,u_2\w,-u_4\w,c_3\w,c_2\w,c_4\w,\lambda,-\mu)\hh,\\
\lambda<\nnh-\mu\le0\hh,
&(u_4\w,u_2\w,u_3\w,-u_1\w,c_3\w,c_2\w,c_4\w,-\nnh\lambda,\mu)\hh,\\
\lambda<\mu\le0\hh,
&(u_4\w,u_3\w,u_2\w,u_1\w,c_2\w,c_3\w,c_4\w,-\nnh\lambda,-\mu)\hh.\\
\end{array}
\end{equation}
There are four more options left: 
$\,-\mu<\nnh-\lambda\le0$, $\,\,-\mu<\lambda\le0$,
$\,\,\mu<\nnh-\lambda\le0\,$ and 
$\,\mu<\lambda\le0$. We reduce them to (\ref{rpl}) by replacing
$\,(u_1\w,u_2\w,u_3\w,u_4\w,c_2\w,c_3\w,c_4\w,\lambda,\mu)$ with
$\,(u_2\w,u_1\w,u_4\w,u_3\w,c_2\w,c_3\w,c_4\w,\mu,\lambda)$.
\end{rem}
\begin{rem}\label{wknwn}
It is well known -- see, e.g.,\ \cite[Lemma 7.1]{derdzinski-park-shin} --
that every pair of length $\,\sqrt{2\,}$ positive orthogonal 
bases of $\,\Lambda\nnh^\pm\nnh$ has the form (\ref{bas}) for some 
positive or\-tho\-nor\-mal basis $\,u_1\w,\dots,u_4\w$ of $\,\mathcal{T}\nnh$,
which is unique up to an overall sign change.
\end{rem}
\begin{rem}\label{nincs}
Each of the following nine choices of the data
$\,c_2\w,c_3\w,c_4\w,\xi,\sca\,$ clearly leads, via (\ref{ari}-c) 
(as well as (\ref{ari}-b), for the three quadruples with $\,\xi=0$), 
to the ordered spectra of $\,24\hh W^+$ and 
$\,24\hh W^-$ appearing at the end of each line below:
\begin{equation}\label{chb}
\begin{array}{rlll}
\mathrm{i)}&(c_2\w,c_3\w,c_4\w,\xi)\,=\,(0,0,0,\sca/8)\hh,
&\hskip-3pt(-\nh2\sca,\sca,\sca)\hh,&\hskip-3pt(-\nh2\sca,\sca,\sca)\hh,\\
\mathrm{ii)}&(c_2\w,c_3\w,c_4\w,\xi)\,=\,(-\sca/4,\sca/4,0,0)\hh,
&\hskip-3pt(-\nh8\sca,4\sca,4\sca)\hh,&\hskip-3pt(4\sca,-\nh8\sca,4\sca)\hh,\\
\mathrm{iii)}&(c_2\w,c_3\w,c_4\w,\xi)\,=\,(-\sca/4,0,\sca/4,\sca/4)\hh,
&\hskip-3pt(-\nh8\sca,4\sca,4\sca)\hh,&\hskip-3pt(4\sca,4\sca,-\nh8\sca)\hh,\\
\mathrm{iv)}&(c_2\w,c_3\w,c_4\w,\xi)\,=\,(\sca/4,-\sca/4,0,0)\hh,
&\hskip-3pt(4\sca,-\nh8\sca,4\sca)\hh,&\hskip-3pt(-\nh8\sca,4\sca,4\sca)\hh,\\
\mathrm{v)}&(c_2\w,c_3\w,c_4\w,\xi)\,=\,(0,0,0,\sca/4)\hh,
&\hskip-3pt(-\nh2\sca,4\sca,-\nh2\sca)\hh,&\hskip-3pt(-\nh2\sca,4\sca,-\nh2\sca)\hh,\\
\mathrm{vi)}&(c_2\w,c_3\w,c_4\w,\xi)\,=\,(0,\sca/8,-\sca/8,\sca/8)\hh,
&\hskip-3pt(-\nh2\sca,4\sca,-\nh2\sca)\hh,&\hskip-3pt(-\nh2\sca,-\nh2\sca,4\sca)\hh,\\
\mathrm{vii)}&(c_2\w,c_3\w,c_4\w,\xi)\,=\,(\sca/4,0,-\sca/4,\sca/4)\hh,
&\hskip-3pt(4\sca,4\sca,-\nh8\sca)\hh,&\hskip-3pt(-\nh8\sca,4\sca,4\sca)\hh,\\
\mathrm{viii)}&(c_2\w,c_3\w,c_4\w,\xi)\,=\,(0,-\sca/8,\sca/8,\sca/8)\hh,
&\hskip-3pt(-\nh2\sca,-\nh2\sca,4\sca)\hh,&\hskip-3pt(-\nh2\sca,4\sca,-\nh2\sca)\hh,\\
\mathrm{ix)}&(c_2\w,c_3\w,c_4\w,\xi)\,=\,(0,0,0,0)\hh,
&\hskip-3pt(-\nh2\sca,-\nh2\sca,4\sca)\hh,&\hskip-3pt(-\nh2\sca,-\nh2\sca,4\sca)\hh.
\end{array}
\end{equation}
In each of the above nine cases it immediately follows that
\begin{equation}\label{bts}
\begin{array}{rl}
\mathrm{a)}&\mathrm{both\ }\,W^\pm\mathrm{\hn\ have\ the\ 
same\ unordered\ spectrum\ 
}\hs\{\sigma,-\hn\sigma\nnh/\hn2,-\hn\sigma\nnh/\hn2\}\mathrm{,\nnh\hn\ 
and}\\
\mathrm{b)}&\mathrm{the\ unique\ simple\ eigen\-val\-ue\
}\,\sigma\hh\mathrm{\ is\ one\
of\hs}-\nnh\nh\sca/3,-\sca/\nh12,\hs\sca/6\mathrm{,\nh\ if\ }\,\sca\ne0\hh.
\end{array}
\end{equation}
\end{rem}
\begin{rem}\label{cvrsl}
Let there be 
given a non-Ein\-stein, weakly Ein\-stein algebraic curvature tensor in 
$\,\mathcal{T}\hs$ such that 
$\,W^\pm\nnh$, both nonzero, have the same unordered spectrum with a repeated
eigen\-valu\-e. 
Using Theorem~\ref{weact}, we then get one of the nine cases of (\ref{chb}). 
In fact, (\ref{ari}-a) is excluded:
$\,\{\sigma,-\hn\sigma\nnh/\hn2,-\hn\sigma\nnh/\hn2\}
=\,\{-\hn\sigma,\sigma\nnh/\hn2,\sigma\nnh/\hn2\}$ only if $\,\sigma=0$. Thus,
(\ref{ari}-b) or (\ref{ari}-c) follows and, focusing just on 
the spectra of $\,W^\pm\nnh$, we may treat the former 
as a subcase of the latter, with $\,\xi=0$. There are now 
nine possibilities, based on choosing which eigen\-value in 
$\,\Lambda\nnh^+$ and which in $\,\Lambda\nnh^-$ is simple. Listed in the
order first-first, first-sec\-ond, ..., third-sec\-ond, third-third, the nine
cases easily lead to (\ref{chb}).
\end{rem}
\begin{rem}\label{htinv}Given a proper weakly Einstein oriented four-man\-ifold
$\,(M\nh,g)$ and a point $\,x\in M\nh$, Theorem~\ref{weact} allows us to
form an ordered string of eleven scalars, consisting of the scalar
curvature $\,\sca(x)\,$ followed by the ordered quadruple and two triples
in the respective line of (\ref{ari}), representing the spectra of the 
Ein\-stein tensor $\,\ein\,$ and $\,W^\pm\nh$ at $\,x\,$
in a suitable positive or\-tho\-nor\-mal basis $\,u_1\w,\dots,u_4\w$ of 
$\,T\hskip-3pt_x\w\hn M\,$ and the corresponding bases (\ref{bas}) 
of self-dual and anti-self-dual bi\-vec\-tors. When treated as defined only
up to rescaling and certain specific permutations, this  e\-lev\-en-sca\-lar
string is clearly \emph{a lo\-cal-homo\-thety invariant\/} of the metric
$\,g\,$ at $\,x$.
\end{rem}
 
\section{Permutation groups}\label{pg}
\setcounter{equation}{0}
Any positive basis of an oriented real vector space can be 
subjected to
\begin{equation}\label{mod}
\begin{array}{l}
\mathrm{permutations,\hs\hs\ possibly\hs\ combined\hs\ with\hs\ some}\\
\mathrm{sign\nh\ changes,\nh\nnh\ so\nh\ as\nh\ to\nh\ preserve\nh\ the\nh\
orientation.}
\end{array}
\end{equation}
Given an oriented Euclidean $\,4$-space $\,\mathcal{T}\nh$, let $\,G\,$ be
the finite matrix group transforming positive or\-tho\-nor\-mal bases 
$\,v\nh_1\w,\dots,v\nh_4\w$ of $\,\mathcal{T}\hs$ by (\ref{mod}). 
Each element of $\,G\,$ acts on the corresponding pair, analogous to
(\ref{bas}), of ordered bases of 
$\,\Lambda\nnh^\pm\nh$, again by (\ref{mod}), in both
$\,\Lambda\nnh^+$ and $\,\Lambda\nnh^-\nh$, in such a way that
\begin{equation}\label{smp}
\mathrm{both\ bases\ in\ the\ pair\ undergo\ the\ same\
permutation,}
\end{equation}
along with some possible sign changes. The reason is that the permutation 
group of $\,\{1,2,3,4\}\,$ naturally acts 
on $\,2+2\,$ \emph{partitions\/} of $\,\{1,2,3,4\}$, by which we mean
\begin{equation}\label{prt}
\mathrm{the\ three\ sets\ }\,\{\{i,j\},\{k,l\}\}\mathrm{,\ where\
}\,\{i,j,k,l\}=\{1,2,3,4\}\hh,
\end{equation}
and each $\,v\nh_i\w\wedge v\nh_j\w\pm v\nh_k\w\wedge v\nh_l\w$ is associated with a unique
partition $\,\{\{i,j\},\{k,l\}\}$.
\begin{rem}\label{dfeig}In Theorem~\ref{weact}, suppose that $\,\ein\,$ {\it
has four 
distinct eigen\-val\-ues}, 
or \emph{each of\/}  
$\,W^\pm$ \emph{has three distinct eigen\-val\-ues}. Let
$\,v\nh_1\w,\dots,v\nh_4\w$
be a positive or\-tho\-nor\-mal basis of $\,\mathcal{T}\nnh$. If, in the former
case, $\,v\nh_1\w,\dots,v\nh_4\w$ consists of eigen\-vec\-tors of $\,\ein$, or,
in the latter, the corresponding bases of type (\ref{bas}) in
$\,\Lambda\nnh^\pm\nnh$ di\-ag\-o\-nal\-ize $\,W^\pm\nh$, then a basis
$\,u_1\w,\dots,u_4\w$ realizing the respective
conclusions (\ref{ari}-a) -- (\ref{ari}-c) arises from
$\,v\nh_1\w,\dots,v\nh_4\w$ 
via (\ref{mod}). This is clear from Remark~\ref{wknwn}, as distinctness of the
eigen\-values makes the basis/bases in question unique up to (\ref{mod}).

As a consequence of (\ref{smp}), the ordered spectra of $\,W^\pm$ 
realized in (\ref{bas}), for $\,u_1\w,\dots,u_4\w$, and those
in the analog of (\ref{bas}) for $\,v\nh_1\w,\dots,v\nh_4\w$, \emph{differ by 
a permutation which is the same for both signs\/} $\,\pm\hs$.
\end{rem}

\section{Con\-for\-mal changes of product metrics}\label{cc}
\setcounter{equation}{0}
The scalar curvatures and Ein\-stein tensors of two 
con\-for\-mal\-ly related metrics $\,\bg\,$ and $\,g=\bg/\hi^2$ in
dimension $\,n\,$ 
are themselves related by 
\begin{equation}\label{crm}
\begin{array}{rl}
\mathrm{i)}&\sca\,=\,\hi^2\hs\bsc+2(n-1)\hi\hs\bde\hi
-n(n-1)\hs\bg(\bna\nnh\hi,\nnh\bna\nnh\hi)\hh,\\
\mathrm{ii)}&\ein\,=\,\bei\,
+\,\displaystyle{\frac{n-2}\hi}\text{\bf[}\hs\bna\nh d\hi\,
-\,\frac{\bde\hi}n\bg\hh\text{\bf]}\hh,\\
\end{array}
\end{equation}
since -- see, e.g.,\
\cite[p.\,529]{derdzinski-00} -- 
for the Ric\-ci tensors one has
\begin{equation}\label{rct}
\ric\,=\,\brc\,+\,(n-2)\hi^{-\nh1}\hs\bna\nh d\hi\,
+\,[\hi^{-\nh1}\bde\hi\,
-\,(n-1)\hi^{-\nh2}\hs\bg(\bna\nnh\hi,\nnh\bna\nnh\hi)]{\bg}\hh.
\end{equation}
When $\,\bg=\hg+\tg\,$ is a product metric on a product manifold 
$\,\hM\times\tM\nh$,
\begin{equation}\label{rco}
\begin{array}{l}
\mathrm{the\ factor\ distributions\ are\ }\,\ein\hyp\mathrm{orthogonal\
if\ and\ only\ if\ }\,\hi\nh\,\mathrm{\ has\ additive}\hyp\\
\mathrm{ly\ separated\ variables\hskip-2.7pt:\ }\hs\hi
=\hh\ha\hh+\ta\mathrm{,\ where\
}\,\ha\nnh:\nnh\hM\nh\to\bbR\,\mathrm{\ and\ }\,\ta\nnh:\nnh\tM\nh\to\bbR\hh,
\end{array}
\end{equation}
This is clear from (\ref{crm}-ii), as both conditions amount to 
$\,\partial\nh_i\w\partial\nh_a\w\hi=0\,$ in
product coordinates $\,x^i\nh,x^a\nh$. With the notation of
Remark~\ref{denot}, for such
a product metric $\,\bg=\hg+\tg\,$ and $\,g=\bg/\hi^2\nh$, (\ref{crm}) yields
\begin{equation}\label{eee}
\begin{array}{rl}
\mathrm{a)}&\sca\hh=\hh\hi^2(\hsc+\tsc\hs)+2(n-1)\hi(\hY\nnh+\tY)\hh
-n(n-1)(\hQ+\tQ)\hh,\\
\mathrm{b)}&\ein\,=\,\hei\,+\,(n-2)\hi^{-\nh1}\hs\hna\nh d\hi\,
+\,\hat\xi\hs\hg\hskip6pt\mathrm{\
along\ }\,\hM\nh\mathrm{,\ as\ well\ as}\\
\mathrm{c)}&\ein\,=\,\tei\,+\,(n-2)\hi^{-\nh1}\hs\tna\nh d\hi\,
+\,\tilde\xi\hs\tg\hskip6pt\mathrm{\ along\ }\,\tM\nh\mathrm{,\ where}\\
\mathrm{d)}&\hat\xi\,=\,p^{-\nh1}\hsc\hs\,
-\hs\,n^{-\nh1}[\hs\hsc\,+\,\tsc\,+\,(n-2)\hi^{-\nh1}(\hY\nnh+\,\tY)]\hh,\\
\mathrm{e)}&\tilde\xi\,=\,q^{-\nh1}\tsc\hs\,
-\hs\,n^{-\nh1}[\hs\hsc\,+\,\tsc\,+\,(n-2)\hi^{-\nh1}(\hY\nnh+\,\tY)]\hh,\\
\mathrm{f)}&\mathrm{for\ }\,p=\dim\hM\,\mathrm{\ and\ }\,q
=\dim\tM\nh\mathrm{,\ with\ }\,n=p+q.
\end{array}
\end{equation}
We will use the obvious fact that, in (\ref{eee}), if $\,p=q$,
\begin{equation}\label{xix}
n\hh\hat\xi=\hsc-\tsc-(n-2)\hi^{-\nh1}(\hY\nnh\nh+\tY)\hh,\quad
n\hh\tilde\xi=\tsc-\hsc-(n-2)\hi^{-\nh1}(\hY\nnh\nh+\tY)\hh.
\end{equation}
For a con\-for\-mal change $\,g=\bg\hh/\nh\chi^2$ of metrics in dimension
$\,n$, and
any function $\,\hi$, one has -- see, e.g.\
\cite[p.\,528]{derdzinski-00} -- the well-known relations
\begin{equation}\label{ndf}
\begin{array}{rl}
\mathrm{a)}&
\nabla\nh d\hi\,=\,\bna\nh d\hi\,
+\,\chi^{-\nh1}[\hs d\chi\otimes d\hi+d\hi\otimes d\chi
-\bg(\bna\nh\chi,\nnh\bna\nh\hi)\hh\bg]\hh,\\
\mathrm{b)}&\Delta\hn\hi\,=\,\chi^2\hh\bde\hi\,
-\,(n-2)\chi\bg(\bna\nh\chi,\nnh\bna\nh\hi)\hh,\quad
g(\nabla\nnh\hi,\nh\nabla\nnh\hi)\,=\,\chi^2\bg(\bna\nh\hi,\nnh\bna\nh\hi)\hh.
\end{array}
\end{equation}

\section{Con\-for\-mal products and the Weyl tensor}\label{cw}
\setcounter{equation}{0}
Throughout this section we assume that $\,\hM\times\tM\,$ is an 
oriented product manifold, $\,\hi:\hM\times\tM\to\bbR\smallsetminus\{0\}$,
while $\,g=(\hs\hg+\tg\hs)\hs/\hi^2$ on $\,\hM\times\tM\nh$, and either
\begin{equation}\label{dim}
\mathrm{a)}\hskip6pt\dim\hM\hn=\dim\tM\nh=2\hh,\quad\mathrm{or}\quad
\mathrm{b)}\hskip6pt\dim\hM\hn=3\,\mathrm{\ \ and\ }\,\dim\tM\nh=1\hh.
\end{equation}
At any given point $\,x\in\hM\times\tM\,$ we may choose a positive
$\,g$-or\-tho\-nor\-mal basis $\,v\nh_1\w,\dots,v\nh_4\w$ of the tangent
space, so that, depending on which case of (\ref{dim}) occurs, for some even
permutation $\,(i,j,k,l)\,$ of $\,\{1,2,3,4\}$,
\begin{equation}\label{uij}
\begin{array}{l}
\mathrm{a)}\hskip9.8ptv\nh_i\w,v\nh_j\w\,\mathrm{\ are\hs\ tangent\hs\ 
to\hs\ }\,\hM\,\mathrm{\hs\ and\hs\ }\hs\,v\nh_k\w,v\nh_l\w\,\mathrm{\ tangent\hs\
to\hs\  }\,\tM\nh\mathrm{,\hs\mathrm\ \ or}\\
\mathrm{b)}\hskip9.8ptv\nh_l\w\mathrm{\ is\ tangent\ to\ }\,\tM\,
\mathrm{\ and\ } \,v\nh_i\w,v\nh_j\w,v\nh_k\w\mathrm{,\ tangent\ to\
}\,\hM\nh\mathrm{,\nh\ diago}\hyp\\
\mathrm{nalize\nh\ the\nh\ Ein\-stein\nh\ tensor\nh\ of\nh\
}\nh\,\hg\hh\mathrm{\ with\nh\ some\nh\ eigen\-values\nh\
}\,\nnh\kp\nnh_i\w\hn,\hn\kp\nnh_j\w\hn,\hn\kp\nh_k\w.
\end{array}
\end{equation}
In case (\ref{uij}-a), 
as $\,v\nh_i\w\wedge v\nh_j\w\pm v\nh_k\w\wedge v\nh_l\w$ are the two K\"ah\-ler forms 
in Remarks~\ref{kahsf} -- \ref{prsrf},
\begin{equation}\label{wpm}
\mathrm{both\ }\,12\hh W^\pm\nh\mathrm{\ have\ the\ same\ spectrum\ 
}\,(2\hi^2(\hs\hsc+\tsc\hs),\hn-\hi^2(\hs\hsc+\tsc\hs),
\hn-\hi^2(\hs\hsc+\tsc\hs))\hh,
\end{equation}
in $\,\Lambda\nnh^\pm\nnh\hn M\nh$, at $\,x$, 
while, if (\ref{uij}-b) holds, then the only
pos\-si\-bly-non\-zero components of $\,2W\hh$ at $\,x\,$ are those
algebraically related
to $\,2W\nnh_{\nh\!jkjk}\w\nh=2W_{\!ilil}\w\nh=-\hi^2\nh\kp\nnh_i\w$, so that
\begin{equation}\label{wij}
\mathrm{both\ }\,2\hi^{-\nh2}\nh W^\pm\mathrm{\nh\ have\ the\ same\
eigen\-val\-ues\ 
}\,(-\kp\nh_k\w,-\kp\nnh_j\w,-\kp\nnh_i\w)\hh,
\end{equation}
and the ordered spectrum in (\ref{wpm}), as well as that in (\ref{wij}), is 
realized
\begin{equation}\label{spr}
\mathrm{by\nh\ the\nh\ bases\hs\ 
}\,v\nh_i\w\wedge v\nh_j\w\pm v\nh_k\w\wedge v\nh_l\w,\,\,
v\nh_i\w\wedge v\nh_k\w\pm v\nh_l\w\wedge v\nh_j\w,\,\,
v\nh_i\w\wedge v\nh_l\w\pm v\nh_j\w\wedge v\nh_k\w.
\end{equation}
Suppose now that, at the given point $\,x=(y,z)\in\hM\times\tM\nh$, one has
\begin{equation}\label{eth}
\mathrm{either\ (\ref{dim}}\hyp\mathrm{a),\ or\ (\ref{dim}}\hyp\mathrm{b)\ 
and\ }\,\kp\nnh_i\w=\kp\nnh_j\w\hh.    
\end{equation}
In other words, we want to lump together (\ref{dim}-a) with a sub\-case of 
(\ref{dim}-b) having a repeated eigen\-valu\-e in (\ref{uij}-b), and we
restrict the choice of permutations $\,(i,j,k,l)$ by requiring that
$\,\kp\nnh_i\w=\kp\nnh_j\w$. From (\ref{wpm}) or, respectively, (\ref{wij})
we now get
\begin{equation}\label{thc}
\mathrm{the\ condition\ (\ref{bts}}\hyp\mathrm{a)\ with\ 
}\,6\hh\sigma=\hi^2(\hs\hsc+\tsc\hs)\mathrm{,\ or\ }\,2\sigma
=-\hi^2\nh\kp\nh_k\w\hh,
\end{equation}
which makes $\,\sigma\,$ the simple-eigen\-val\-ue function of
$\,W^\pm\nnh$.
\begin{lem}\label{smloc}
If\/ $\,\sigma\ne0\,$ in\/ {\rm(\ref{eth})} -- {\rm(\ref{thc})} and\/ 
$\,\zeta^\pm$ are any length\/ $\,\sqrt{2\,}\,$ eigen\-bi\-vec\-tors of\/
$\,W^\pm\nnh$ at\/ $\,x=(y,z)\in\hM\times\tM\,$ for the simple
eigen\-val\-ue\/ $\,\sigma$, then\/ $\,\zeta^+\nnh\nnh\pm\hs\zeta^-$ both
have rank two, and their images coincide, as an unordered pair, with
\begin{enumerate}
\item[{\rm(i)}]the direct summands\/ 
$\,T\hskip-3pt_y\w\hn\hM\,$ and\/ $\,T\hskip-3pt_z\w\hn\tM\,$ of\/ 
$\,T\hskip-3pt_x\w[\hn\hM\nh\times\nh\tM]$, in the case\/ {\rm(\ref{dim}-a)},
\item[{\rm(ii)}]the eigen\-space of the Ein\-stein tensor of\/ $\,\hg\,$ at\/
$\,y\,$ for the eigen\-val\-ue\/ 
$\,\kp\nnh_i\w$, and its orthogonal complement in\/
$\,T\hskip-3pt_x\w[\hn\hM\nh\times\nh\tM]$, for\/ {\rm(\ref{dim}-b)} with\/
$\,\kp\nnh_i\w=\kp\nnh_j\w$.
\end{enumerate}
\end{lem}
\begin{proof}Up to a sign: each of $\,\zeta^\pm$ is unique, and hence equals 
$\,v\nh_i\w\wedge v\nh_j\w\pm v\nh_k\w\wedge v\nh_l\w$ in (\ref{spr}). Now our
claim follows from (\ref{uij}).
\end{proof}
We will refer to the two sub\-spaces of
$\,T\hskip-3pt_x\w[\hn\hM\nh\times\nh\tM]\,$ forming the unordered pair
in (i) or (ii) above as the \emph{summand planes\/} at $\,x$.

In Lemma~\ref{smloc}, if $\,\zeta^\pm\nnh
=u_i\w\wedge u_j\w\pm u_k\w\wedge u_l\w$, or
$\,\zeta^+\nnh=u_i\w\wedge u_j\w+u_k\w\wedge u_l\w$ and 
$\,\zeta^-\nnh=u_i\w\wedge u_k\w-u_l\w\wedge u_j\w$, for 
a positive $\,g$-or\-tho\-nor\-mal basis $\,u_1\w,\dots,u_4\w$ of
$\,T\hskip-3pt_x\w[\hn\hM\nh\times\nh\tM]$ 
and an even permutation $\,\hs(i,j,k,l)\,$
of $\,\hh\{1,2,3,4\}$, then
\begin{equation}\label{spn}
\begin{array}{l}
\mathrm{one\ summand\ plane\ is\ spanned\ by\
}\,\,u_i\w,u_j\w\mathrm{,\hs\ the\ other\ by\ 
}\,\,u_k\w,u_l\w\,\mathrm{\hs\ or,}\\
\mathrm{respectively,\nh\ one\ by\ 
}\,u_i\w\hn-\hh u_l\w,u_j\w\hn+\hh u_k\w\mathrm{\nh,\ the\ other\
by\ }\,u_i\w\hn+\hh u_l\w,u_j\w\hn-\hh u_k\w,
\end{array}
\end{equation}
as one sees evaluating 
$\,\zeta^+\nnh\nnh\pm\hs\zeta^-\nh$, with
$\,u_i\w\wedge u_j\w+u_k\w\wedge u_l\w
\pm(u_i\w\wedge u_k\w-u_l\w\wedge u_j\w)
=(u_i\w\mp u_l\w)\wedge(u_j\w\pm u_k\w)$.

For the remainder of this section we also \emph{assume that\/ $\,g\,$ is a
weakly Ein\-stein metric, and\/ $\,\sigma\ne0\,$ in\/}
(\ref{eth}) -- (\ref{thc}).

If the Ein\-stein tensor of $\,g\,$ at $\,x\,$ is nonzero,
Remark~\ref{cvrsl} yields the existence of a
positive $\,g$-or\-tho\-nor\-mal basis $\,u_1\w,\dots,u_4\w$
of $\,T\hskip-3pt_x\w[\hn\hM\nh\times\nh\tM]\,$ satisfying 
(\ref{ari}-c) or (\ref{ari}-b), with one of the nine options in (\ref{chb}). 
Now, due to (\ref{spn}),
each of the following nine ordered orthogonal bases of
$\,T\hskip-3pt_x\w[\hh\hM\times\nh\tM]$, 
corresponding to the nine cases of (\ref{chb}), consists of a basis of one 
summand plane followed by a basis of the other: 
\begin{equation}\label{nbs}
\begin{array}{rl}
\mathrm{i)}&(u_1\w,u_2\w,u_3\w,u_4\w)\hh,\quad\mathrm{ii)}\hskip6pt
(u_1\w\nh-u_4\w,u_2\w\nh+u_3\w,u_1\w\nh+u_4\w,u_2\w\nh-u_3\w)\hh,\\
\mathrm{iii)}&(u_1\w\nh+u_3\w,u_2\w\nh+u_4\w,u_1\w\nh-u_3\w,u_2\w\nh-u_4\w)\hh,\\
\mathrm{iv)}&(u_1\w\nh+u_4\w,u_2\w\nh+u_3\w,u_1\w\nh-u_4\w,u_2\w\nh-u_3\w)\hh,\\
\mathrm{v)}&(u_1\w,u_3\w,u_2\w,u_4\w)\hh,\quad\mathrm{vi)}\hskip6pt
(u_1\w\nh-u_2\w,u_3\w\nh+u_4\w,u_1\w\nh+u_2\w,u_3\w\nh-u_4\w)\hh,\\
\mathrm{vii)}&(u_1\w\nh-u_3\w,u_2\w\nh+u_4\w,u_1\w\nh+u_3\w,u_2\w\nh-u_4\w)\hh,\\
\mathrm{viii)}&(u_1\w\nh+u_2\w,u_3\w\nh+u_4\w,u_1\w\nh-u_2\w,u_3\w\nh-u_4\w)\hh,\quad\mathrm{ix)}\hskip6pt
(u_1\w,u_4\w,u_2\w,u_3\w)\hh.
\end{array}
\end{equation}
The third, fourth and eighth bases arise here from the following analog
of the second line in (\ref{spn}): if $\,\zeta^+\nnh
=u_i\w\wedge u_j\w+u_k\w\wedge u_l\w$ and 
$\,\zeta^-\nnh=u_i\w\wedge u_l\w-u_j\w\wedge u_k\w$, then 
one summand plane 
is spanned by $\,u_i\w\hn+\hh u_k\w,u_j\w\hn+\hh u_l\w$, and the other by 
$\,u_i\w\hn-\hh u_k\w,u_j\w\hn-\hh u_l\w$, which is obvious since
$\,u_i\w\wedge u_j\w+u_k\w\wedge u_l\w
\pm(u_i\w\wedge u_l\w-u_j\w\wedge u_k\w)
=(u_i\w\pm u_k\w)\wedge(u_j\w\pm u_l\w)$.
\begin{lem}\label{othzr}
In the cases\/ {\rm(\ref{nbs}-ii)}, {\rm(\ref{nbs}-iii)}, {\rm(\ref{nbs}-iv)}
and\/ {\rm(\ref{nbs}-vii)}, 
the restrictions of the Ein\-stein tensor\/ $\,\ein\,$ of\/ $\,g\,$ to\/ 
both summand planes equal zero, while in the remaining five cases
the summand planes 
are\/ $\,\ein$-or\-thog\-o\-nal. 
When\/ {\rm(\ref{ari}-c)} and\/ \hbox{{\rm(\ref{nbs}-i)}} hold, 
the summand planes are the eigen\-spaces of\/ $\,\ein$.
\end{lem}
\begin{proof}By (\ref{ari}), the en\-do\-mor\-phism $\,E\,$ of
$\,T\hskip-3pt_x\w[\hh\hM\times\nh\tM]\,$ corresponding via $\,g\,$ to
$\,\ein\,$ assigns 
$\,-\nnh\lambda u_1\w,-\nh\mu u_2\w,\mu u_3\w,\lambda u_4\w$ to
$\,u_1\w,u_2\w,u_3\w,u_4\w$, with 
$\,\mu\,$ standing for $\,\lambda\,$ in \hbox{(\ref{ari}-c).} Consequently, 
in the former four cases, $\,E\,$ sends each summand plane into the other. 
In the latter five, $\,E\,$ leaves both summand planes
invariant. For (\ref{nbs}-i), (\ref{nbs}-v) and (\ref{nbs}-ix), this last
claim is immediate, For (\ref{nbs}-vi) and (\ref{nbs}-viii), it follows since
$\,\mu=\lambda$. Namely, 
our assumption that $\,\sigma\ne0\,$ in (\ref{eth}) -- (\ref{thc}) gives
$\,\sca\ne0\,$ in (\ref{chb}), and hence $\,\xi\ne0$, leading to 
(\ref{ari}-c) rather than (\ref{ari}-b).
\end{proof}
\begin{rem}\label{eight}In the eight cases of (\ref{nbs}) other than
(\ref{nbs}-i), the restrictions of $\,\ein\,$ to the summand planes 
are both
$\,g$-trace\-less. For the fifth, sixth, eighth and ninth cases this follows
since, by (\ref{ari}), the corresponding orthogonal basis in
(\ref{nbs}) consists of eigen\-vec\-tors of $\,\ein\,$ in
$\,T\hskip-3pt_x\w[\hh\hM\times\nh\tM]\,$ for the eigen\-values 
$\,(-\nnh\lambda,\lambda,-\nh\mu,\mu)\,$ (the ninth case) or
$\,(-\nnh\lambda,\lambda,-\nnh\lambda,\lambda)\,$ (the other three). 
For the remaining four (second, third, fourth, seventh) our claim is
obvious from Lemma~\ref{othzr}.
\end{rem}

 
\section{Examples}\label{ex}
\setcounter{equation}{0}
The following two theorems provide constructions of weakly Ein\-stein
con\-for\-mal products. Their proofs, and explicit realizations of the
geometric structures used in the constructions, are given in
Sect.\,\ref{pt} -- \ref{gc}.
\begin{thm}\label{exatt}
Given Riemannian surfaces\/ $\,(\varSigma,g)\,$ and\, $\,(\varPi\nh,h)\,$ with
Gauss\-i\-an curvatures\/ $\,K$ and\, $\,c\,$ such that\/ $\,c\,$ is 
constant, and a function\/ 
$\,\hi:\varSigma\to\bbR\smallsetminus\{0\}$, for\/ $\,\Delta\,$ and\/
$\,\nabla$ referring to the metric\/ $\,g$, let on\/ $\,\varSigma\,$ either 
\begin{itemize}
\item[(i)] $\Delta \hi^{-1}=0\,$ and\/
$\,2g(\nabla\nnh\hi,\nnh\nabla\nnh\hi)=(c-K)\hi^2\nh$, or
\item[(ii)] $2\nabla\nh d\hi=(\Delta\hn\hi)g\,$ and\/
$\,(K\nnh\hn+\hs c)\hi^2\nh+\hi\Delta\hn\hi
=2g(\nabla\nnh\hi,\nnh\nabla\nnh\hi)$.
\end{itemize}
Then the metric\/ $\,(g+h)/\hi^2$ on\/ $\,\varSigma\times\varPi\hs$
is weakly Ein\-stein. Also, 
$\,(g+h)/\hi^2$ is 
\begin{itemize}
\item [(iii)] Ein\-stein in the case\/ {\rm(i)} if and only if\/ 
$\,2\nabla\nh d\hi=(\Delta\hn\hi)g$,
\item [(iv)] Ein\-stein in the case\/ {\rm(ii)} if and only if\/ 
$\,g(\nabla\nnh\hi,\nnh\nabla\nnh\hi)=c\hh\hi^2\nh$,
\item[(v)] con\-for\-mal\-ly flat if and only if\/ $\,K=-c$. 
\end{itemize}
More precisely, the equalities in\/ {\rm(iii)} -- {\rm(v)} characterize
points at which the Ein\-stein or, respectively, Weyl tensor of the
metric\/ $\,(g+h)/\hi^2$ vanishes. Wherever both these tensors are nonzero,
the lo\-cal-homo\-thety invariant of Remark\/~{\rm\ref{htinv}} is given by
\begin{equation}\label{lhi}
\begin{array}{l}
\bsc\hh,\,\,(-\nnh\lambda,0,0,\lambda)\hh,\,\,
(-\bsc/\nh12,-\bsc/\nh12,\,\bsc/6)\hh,\,\,
(-\bsc/\nh12,-\bsc/\nh12,\,\bsc/6)\,\mathrm{\ in\ the\ case\ (i),}\\
\bsc\hh,\,(-\nnh\lambda,-\nnh\lambda,\lambda,\lambda)\hh,\,
(-\bsc/\nh12,\,\bsc/\nh24,\,\bsc/\nh24)\hh,\,
(-\bsc/\nh12,\,\bsc/\nh24,\,\bsc/\nh24)\,\mathrm{\ in\ the\ case\ (ii),}\\
\end{array}
\end{equation}
with the scalar curvature\/ $\,\bsc\ne0\,$ of the metric\/ 
$\,(g+h)/\hi^2$ and a parameter\/ $\,\lambda\ne0$.
\end{thm}
The next theorem employs our usual notations: $\,\ric,\ein,\sca\,$ for the
Ric\-ci/Ein\-stein tensors and scalar curvature of the metric $\,g$,
and $\,\nabla\hs$ for its Le\-vi-Ci\-vi\-ta connection as well as 
the $\,g$-grad\-i\-ent, while the barred symbols $\,\bei,\bsc\,$ and
$\,\overline W\hs$ correspond to 
$\,\bg$.
\begin{thm}\label{exaot}
Let there be 
given an open interval\/ $\,I\subseteq\bbR\,$ with the coordinate\/ 
$\,\vt\,$ and the metric\/ $\,d\vt\hh^2\nh$, a Riemannian three-man\-i\-fold\/ 
$\,(M\nh,g)$, a vector field\/ $\,v\,$ on\/ $\,M$ and functions\/
$\,\hi,\kp:M\to\bbR\,$ such that\/ $\,\hi\ne0\,$ everywhere. 
If, moreover,
\begin{equation}\label{spc}
\begin{array}{rl}
\mathrm{i)}&
2\nabla\nh d\hi+\hi\hh\ric\,\mathrm{\ has\ the\ }\,g\hyp\mathrm{spectrum\
}\,(-|v|^2\nh,-|v|^2\nh,0)\hh,\\
\mathrm{ii)}&(2\nabla\nh d\hi\,+\,\hi\hh\ric)(v,\,\cdot\,)\,=\,0\hh,\\
\mathrm{iii)}&[(\sca-6\kp)\hi\,+\,6\hh\Delta\hn\hi]\hh\hi\,
=\,12\hh g(\nabla\nnh\hi,\nh\nabla\nnh\hi)\hh,\\
\mathrm{iv)}&\ein(v,\,\cdot\,)\,=\,\kp\hn g(v,\,\cdot\,)\hh,
\end{array}
\end{equation}
then\/ $\,\bg=(g+d\vt\hh^2)/\hi^2$ is a 
weakly Ein\-stein metric on\/ $\,M\nnh\times\hn I\nh$. Also, 
$\,\bg\,$ is Ein\-stein, or con\-for\-mal\-ly flat, if and only if\/
$\,v\,$ vanishes identically on\/ $\,M\,$ or, respectively,
$\,g\,$ has constant sectional curvature. More
precisely, $\,\bei=0$ or $\,\overline W\nnh\nh=\hs0\,$ at precisely those
points at which\/ $\,v=0\,$ or, respectively, $\,\ein=0$. Wherever\/
$\,v\ne0\,$ and\/ $\,\ein\ne0$, 
the lo\-cal-homo\-thety invariant of Remark\/~{\rm\ref{htinv}} is given by
\[
\bsc\hh,\quad(-\nnh\lambda,-\nnh\lambda,\lambda,\lambda)\hh,\quad
(-\bsc/\nh12,\,\xi-\bsc/\nh12,-\xi+\bsc/6)\hh,\quad
(-\bsc/\nh12,\,\xi-\bsc/\nh12,-\xi+\bsc/6),
\]
$\bsc\ne0\,$ being the scalar curvature of\/
$\,\bg$, with parameters\/ $\,\lambda\ne0\,$ and\/ $\,\xi$, while\/
$\,\xi=\bsc/8$ in the case where $\,\ein\,$ has, at every point, the\/
$\,g$-spec\-trum\/ $\,\{\kp,-\kp/2,-\kp/2\}$.
\end{thm}
\begin{rem}\label{spceq}
The conditions i) -- ii) in (\ref{spc}), imposed on a function $\,\hi\,$
and a vector field $\,v\,$ on a Riemannian three-man\-i\-fold $\,(M\nh,g)$, are
equivalent to
\begin{equation}\label{amt}
\begin{array}{l}
2\nabla\nh d\hi+\hi\hh\ric-(\Delta\hn\hi+\sca\hi/2)g
=\omega\hn\otimes\omega\hh\mathrm{\ for\ the\ }1\hyp\mathrm{form\
}\hh\omega=g(v,\,\cdot\,)\mathrm{,\nnh\ which}\\
\mathrm{also\ implies\ that\ }\,\,\sym\,=\hs2\nabla\nh d\hi\hh
+\hh\hi\hh\ric\,\,\mathrm{\ has\
}\,\mathrm{tr}\nnh_g\w\sym\hh=\hh2(\Delta\hn\hi
+\sca\hi/2)\hh=\hh-\nh2|v|^2\nnh.
\end{array}
\end{equation}
In fact, i) -- ii) obviously yield the second line of (\ref{amt}), and so
does the first line of (\ref{amt}). Now assume i) -- ii). At points $\,x\,$
where $\,v\nh_x\w=0$, (\ref{amt}) follows: the tensor field
$\,\sym=2\nabla\nh d\hi+\hi\hh\ric$, having the 
spectrum $\,(-|v|^2\nh,-|v|^2\nh,0)$, vanishes at $\,x\,$ along with its
$\,g$-trace $\,2(\Delta\hn\hi+\sca\hi/2)=-\nh2|v|^2\nh$. Wherever 
$\,v\ne0$, the left-hand side in (\ref{amt}), with the spectrum
$\,(0,0,|v|^2)$, clearly equals the right-hand side: treated as
$\,(1,1)$ tensors, both send $\,v\,$ to $\,|v|^2\nh v$, and annihilate
vectors orthogonal to $\,v$.
Conversely, (\ref{amt}) easily gives $\,\sym(v,\,\cdot\,)=0\,$ 
and $\,\sym(v'\nh,\,\cdot\,)=-|v|^2\nh g(v'\nh,\,\cdot\,)\,$ if
$\,g(v,v')=0$.
\end{rem}

\section{Proof of Theorem~\ref{exatt}}\label{pt}
\setcounter{equation}{0}
Let us adopt the notation of (\ref{cnv}) and (\ref{crm}), with
$\,(\ha,\ta)=(\hi,0)$. Now our
\begin{equation}\label{ado}
\begin{array}{l}
\varSigma,g,\nabla\nh,\Delta,2K,\varPi\nh,h,2c\hs\,\mathrm{\ 
become\ }\,\hs\hM\nh,\hg,\hna\nh,\hde\hs,\hsc,\tM\nh,\tg,\tsc\mathrm{\hh,\
and\ }\,\hs\tY\nnh=\hh\hs\tQ=\hh0,\\
\mathrm{while\ }\hs g\hs\mathrm{\ is\ now\ the\
con\-for\-mal}\hyp\mathrm{prod\-uct\ metric\ }\,g=\bg/\hi^2\nnh\mathrm{,\nnh\
for\ }\,\bg=\hg+\tg.
\end{array}
\end{equation}
Except for (iii) -- (v), all of our assertion will be obvious
from Remark~\ref{cnvrs} 
if we prove the following two claims, for some positive
$\,g$-or\-tho\-nor\-mal basis
$\,u_1\w,\dots,u_4\w$ of the tangent space at any point:
\begin{enumerate}
\item[{\rm(I)}]Assuming (i) we get (\ref{ari}-b) with
$\,c_2\w=c_3\w=c_4\w=0\,$ and $\,\mu=0$.
\item[{\rm(II)}]Similarly, (ii) leads to (\ref{ari}-c) with 
$\,c_2\w=c_3\w=c_4\w=0\,$ and $\,\xi=\sca/8$.
\end{enumerate}
In both cases, by 
(\ref{rco}), for the Ein\-stein tensor $\,\ein=\ric-\sca\hs g/4$,
\begin{equation}\label{rcc}
\mathrm{the\ factor\ distributions\ }\,T\hn\hM\,\mathrm{\ and\ 
}\,T\hn\tM\,\mathrm{\ are\ }\,\ein\hyp\mathrm{orthogonal.}
\end{equation}
We will choose  
$\,(u_1\w,\dots,u_4\w)=(v\nh_1\w,\dots,v\nh_4\w)\,$ and
$\,(i,j,k,l)\,$ so 
as to have (\ref{uij}-a), (\ref{wpm}) and (\ref{spr}). 
In the case (i), we begin by observing that
\begin{enumerate}
\item[{\rm(a)}]$\hi^2\nh(\tsc-\hsc)=4\hh\hQ$,
\qquad\quad{\rm(b)}\hskip5pt$\hi\hY\nnh\hn=2\hh\hQ$,
\item[{\rm(c)}]$\sca\,=\,\hi^2(\hsc+\tsc\hs)\,\,$ for the scalar curvature 
$\,\sca\,$ of $\,g$,
\item[{\rm(d)}]$\ein\,$ restricted to $\,T\nh\hM\,$ is $\,g$-trace\-less,
and on $\,T\nh\tM\,$ it equals $\,0$.
\end{enumerate}
Namely, (i) yields (a) -- (b), as 
$\,\hi\hY\nnh\nh-2\hh\hQ
=\hi\hs\hde\hi-2\hg(\hna\nnh\hn\hi,\hna\nnh\hn\hi)=-\hi^3\hde\hi^{-1}\nh=0\,$
by (\ref{dqe}-b). 
Now (b) and (\ref{eee}-a), with $\,\tY\nnh\nh=\hs\tQ=0$, imply
(c), while (\ref{rcc}) leads to (d) for $\,T\nh\hM\nh$, since (d) for
$\,T\nh\tM\,$ is obvious from (a) -- (b), (\ref{eee}-c) and (\ref{xix}).

In view of (d) and (\ref{rcc}) we may choose an orientation and a positive
$\,g$-or\-tho\-nor\-mal basis
$\,(u_1\w,\dots,u_4\w)=(v\nh_1\w,\dots,v\nh_4\w)\,$ of
$\,T\hskip-3pt_x\w\hn M\,$ at any given point $\,x$, realizing for
$\,\ein$ a spectrum of
the form $\,(-\nnh\lambda,0,0,\lambda)$, which yields (\ref{uij}-a) 
with $\,(i,j,k,l)=(1,4,2,3)$.

Now (\ref{wpm}), (\ref{spr}) and (c) imply
(\ref{ari}-b) for $\,c_2\w=c_3\w=c_4\w=0$, as required in (I).

Next, assume (ii), which has the following obvious consequences:
\begin{enumerate}
\item[{\rm(e)}]$2\hi\hY\nnh\hn-4\hh\hQ=-\hi^2\nh(\hs\hsc+\tsc\hs)$, 
\item[{\rm(f)}]$\ein\,$ restricted to either factor distribution equals
a function times $\,g$,
\item[{\rm(g)}]$\sca\,=\,-\nh2\hi^2(\hsc+\tsc\hs)\,\,$ for the scalar
curvature $\,\sca\,$ of $\,g$,
\end{enumerate}
(g) being immediate from (e) and (\ref{rwr}-iii) with
$\,\tY\nnh\nh=\hs\tQ=0$. 

Just as (d) did before, (f) now allows us to choose a positive 
$\,g$-or\-tho\-nor\-mal basis $\,u_1\w,\dots,u_4\w$ of 
of any given oriented tangent space consisting of eigen\-vec\-tors of
$\,\ein\,$ with eigen\-val\-ues having, this time, 
the form $\,(-\nnh\lambda,-\nnh\lambda,\lambda,\lambda)$, leading to
(\ref{uij}-a), where $\,(i,j,k,l)=(1,2,3,4)$. Then (\ref{wpm}) and (g) give 
(\ref{ari}-c) with $\,c_2\w=c_3\w=c_4\w=0$ and $\,\xi=\sca/8$, proving
(II).

For the con\-for\-mal-prod\-uct metric $\,g=\bg/\hi^2\nh$,
since $\,\ta=\tY\nnh\nh=\tQ=0$,
\begin{equation}\label{rhs}
\begin{array}{l}
\mathrm{the\hs\ 
right}\hyp\mathrm{hand\hs\ sides\hs\ of\hs\ (\ref{eee}}\hyp\mathrm{b)\hs\
and\hs\ (\ref{eee}}\hyp\mathrm{c)}\\
\mathrm{have\ }\hs g\hyp\mathrm{traces\
differing\ by\ }\,2\hi\hY\nh+\hh\hi^2\nh(\hsc-\tsc).
\end{array}
\end{equation}
The `only if\hs' claim in (iii) is obvious from (\ref{eee}-b).
The `if' part now follows: assuming that $\,2\hna\nh d\hi=\hY\nh\hg$, we get
$\,2\hi\hY\nnh+\hi^2\nh(\hsc-\tsc)=0\,$ in (\ref{rhs}) by combining
(a) and (b). Similarly, in the case of (ii), the Ein\-stein condition -- 
that is, vanishing of $\,2\hi\hY\nnh+\hi^2\nh(\hsc-\tsc)\,$ in
(\ref{rhs}) -- is, by (e), equivalent to having $\,2\hh\hQ=\hi^2\nh\tsc$, 
which yields (iv). Finally, (v) is obvious from Remark~\ref{prsrf}.

\section{Proof of Theorem~\ref{exaot}}\label{px}
\setcounter{equation}{0}
The tensor field $\,\sym=2\nabla\nh d\hi+\hi\hh\ric\,$ has the spectrum 
$\,(-|v|^2\nh,-|v|^2\nh,0)$. In the interior of the subset of
$\,M\nnh\times\hn I\hs$
on which $\,v=0$, one thus has $\,\sym=0$, and (\ref{rct}), applied to
$\,n=4\,$ and $\,g,\bg\,$ replaced with our $\,\bg\,$ and the product metric
$\,g+dt^2\nh$, implies that $\,\brc\,$ is a functional multiple of $\,\bg$,
which makes $\,\bg\,$ an Ein\-stein metric, proving our first
claim in this case.

We may therefore restrict our discussion to points where $\,v\ne0$, further
assuming that they are generic relative to the spectrum of $\,\ein\,$
(see Remark~\ref{genrc}), and $\,M\,$ is oriented. Thus, by (\ref{spc}-iv),
locally, 
some positive $\,\bg$-or\-tho\-nor\-mal frame $\,u_1\w,\dots,u_4\w$
has $\,u_1\w,u_2\w,u_3\w$
tangent to the $\,M\,$ factor and 
di\-ag\-o\-nal\-izing $\,\ein\hh$, with our $\,v\,$
equal to a positive functional multiple of $\,u_3\w$.

Next, $\,u_1\w,\dots,u_4\w$ also 
di\-ag\-o\-nal\-ize $\,\bei\,$ with the ordered $\,\bg$-spec\-trum
$\,(-\nnh\lambda,-\nnh\lambda,\lambda,\lambda)$, for 
$\,\lambda=\hi\hh|v|^2/2$.
Namely, we have (\ref{eee}) with
$\,\tsc,\tY\nnh\nnh,\tQ,\tei,\nnh\tna\nh d\hi\,$
equal to zero, $\hs(n,p,q)\nh=(4,3,1)$, and 
$\,M\nh,I\nh,\bsc,\bei,g,\nabla\nh,Y\nnh\nh,Q,dt^2$ replacing
$\,\hM\nh,\tM\nh,\sca,\ein,\hg,\hna\nh,\hY\nnh\nh,\hQ,\tg$, so that
\begin{equation}\label{nwe}
\begin{array}{rl}
\mathrm{a)}&\bsc\hh\,=\,\hh\hi^2\sca\hs\,+\,\hs6\hi Y\nh-\,12\hh Q\mathrm{,\ where\
}\,Q=g(\nabla\nnh\hi,\nh\nabla\nnh\hi)\,\mathrm{\ and\
}\,Y\nnh\nh=\Delta\hn\hi\hh,\\
\mathrm{b)}&12\hs\bei=12\hs\ein+24\hi^{-\nh1}\hn\nabla\nh d\hi
+(\sca-6\hi^{-\nh1}Y)\hh g\,\mathrm{\ along\ }\,M\nh,\\
\mathrm{c)}&4\hh\bei\,=\,-(\sca+2\hi^{-\nh1}Y)\,dt^2\hskip6pt\mathrm{\ along\
}\,I\nh.
\end{array}
\end{equation}
As $\,\ein=\ric-\sca g/3$, (\ref{nwe}-b) combined with (\ref{amt}) gives
$\,\hi\hs\bei=\sym+|v|^2\nh g/2$, along $\,M\nh$, which -- according to
(\ref{spc}) -- 
is di\-ag\-o\-nal\-ized by $\,u_1\w,u_2\w,u_3\w$ with the ordered
$\,g$-spec\-trum $\,(-|v|^2/2,-|v|^2/2,|v|^2/2)$, that is, the  
$\,\bg$-spec\-trum $\,(-\hi\lambda,-\hi\lambda,\hi\lambda)$. At the same time,
by (\ref{nwe}-c) and (\ref{amt}), along $\,I\hs$ (the span of $\,u_4\w$)
$\,\bei\,$ equals $\,dt^2$ times $\,\hi^{-\nh1}\nh|v|^2/2$, which is
nothing else than $\,\bg\,$ times $\,\lambda$. Consequently,
\begin{equation}\label{esp}
\begin{array}{l}
\bei\,\,\mathrm{\ has\ the\ }\,\bg\hyp\mathrm{spectrum\ 
}\,(-\nnh\lambda,-\nnh\lambda,\lambda,\lambda)\,\mathrm{\ with}\\ 
\lambda\,=\,\hi\hh|v|^2/2\mathrm{,\ realized\ by\ our\
}\,(u_1\w,\dots,u_4\w)\hh.
\end{array}
\end{equation}
On the other hand, with the new notation, in view of (\ref{wij}),
\begin{equation}\label{hts}
2\hi^{-\nh2}\hs\overline W^\pm\mathrm{\ have\ the\ }\,\bg\hyp\mathrm{spectrum\ 
}\,(-\kp_3\w,-\kp_2\w,-\kp_1\w)\mathrm{,\ realized\
as\ in\ (\ref{spr}),}
\end{equation}
for the Weyl tensor $\,\overline W\nnh$ of $\,\bg$,
where $\,\kp\nnh_i\w$ are the eigen\-value functions of
$\,\ein$, associated with $\,u_i\w$, $\,i=1,2,3$. Consequently, our choice of
$\,u_3\w$, combined with (\ref{spc}-iv)
gives $\,\kp=\kp_3\w$. The condition (\ref{spc}-iii),
which now
reads $\,[(\sca-6\kp)\hi+6\hh Y]\hh\hi
=12\hh Q$, turns (\ref{nwe}-a) into the equality
$\,\bsc=6\hi^2\hn\kp$. We thus obtain 
(\ref{ari}-c) for $\,\bei\,$ and $\,\overline W\nnh\nh$, with 
$\,c_2\w=c_3\w=c_4\w=0\,$ and $\,\xi\,$ given by
$\,12\hs\xi=\bsc-6\hi^2\hn\kp_2\w=2\hs\bsc+6\hi^2\hn\kp_1\w$ (both
equalities yield the same $\,\xi$, since
$\,\kp_1\w+\kp_2\w=-\kp$), along with the final clause of the theorem:
if $\,\kp_1\w=\kp_2\w=-\kp/2$, the preceding sentence shows that
$\,\xi=\bsc/8$.

Remark~\ref{cnvrs} now implies that 
$\,\bg\,$ is a weakly Ein\-stein metric, while the claim about the
Ein\-stein and con\-for\-mal\-ly flat cases 
is obvious from (\ref{esp}) and (\ref{hts}).

\section{Euclidean spheres and Lo\-rentz\-i\-an pseu\-do\-spheres}\label{es}
\setcounter{equation}{0}
Throughout this section, 
$\,\langle\,,\rangle\,$ denotes a Euclidean or Lo\-rentz\-i\-an 
($-\,+\,\ldots+$) inner product $\,\langle\,,\rangle\,$ in a 
real vector space $\,\mathcal{V}\hs$ of dimension $\,n+1\ge3$, and 
$\,\varSigma\,$ stands for the sphere $\,S^{-\nnh1}\nh(0)\,$ or one sheet 
of the two-sheet\-ed pseu\-do\-sphere $\,S^{-\nnh1}\nh(0)$,
where $\,S(v)=\langle v,v\rangle-a$, and $\,a\in(0,\infty)\,$ or, respectively,
$\,a\in(-\infty,0)$. The sub\-man\-i\-fold metric $\,g\,$ of $\,\varSigma\,$
thus has constant sectional curvature $\,K\nnh=a^{-\nnh1}\nh$.

A function $\,\hi\,$ on a Riemannian manifold $\,(M\nh,g)\,$ is called 
\emph{con\-cir\-cu\-lar\/} \cite{yano} if 
\begin{equation}\label{cnc}
\begin{array}{l}
\nabla\nh d\hi\,=\hs\,\sigma\hn g\hs\,\mathrm{\ for\ some\ function\
}\,\sigma\mathrm{,\ obviously\ given\hs\ by\ }\,n\sigma\hs=\hs\Delta\hi\hh,\\
\mathrm{for\ }\,n=\dim M\nh\mathrm{,\nh\ or,\nh\ equivalently,\nnh\ 
}\,\nabla\nh\hi\mathrm{\ is\ a\ con\-for\-mal\ vector\ field.}
\end{array}
\end{equation}
\begin{rem}\label{rctev}
In {\rm(\ref{cnc})}, if 
$\,\sigma=F(\hi)\,$ is a function of $\,\hi$, then the gradient 
$\,\nabla\nh\hi$, at points where\/ $\,\nabla\nh\hi\ne0$, is an 
eigen\-vec\-tor of the Ric\-ci tensor for the eigen\-val\-ue  
$\,(1-n)F'\nh(\hi)$. At such points, $\,g\,$ has the Gaussian curvature  
$\hh-F'\nh(\hi)\hs$ when $\,n=2$. (This is obvious from (\ref{bch})
applied to $\,v=\nabla\nh\hi$.)
\end{rem}
\begin{lem}\label{concr}Con\-cir\-cu\-lar functions on any nonempty connected
open subset\/ $\,\,U$ of\/ $\,\varSigma$, up to additive constants,
coincide with restrictions to\/
$\,\,U\hs$ of linear functionals\/ $\,\mathcal{V}\nh\to\bbR$, and\/
$\,\sigma\,$ in\/ {\rm(\ref{cnc})} then equals\/ $\,-a^{-\nnh1}\nh\hi$.
\end{lem}
This is immediate from the Boch\-ner identity (\ref{bch}) applied to 
(\ref{cnc}), as one sees differentiating restrictions of linear functionals 
twice along the $\,g$-ge\-o\-des\-ics, and using the trigonometric 
or hyperbolic descriptions of the latter.

Let $\,u\in\mathcal{V}\nh$. 
For the La\-plac\-i\-an $\,\Delta\,$ and grad\-i\-ent 
$\,\nabla\hs$ associated with the sub\-man\-i\-fold metric $\,g\,$ of 
$\,\varSigma$, and the function 
$\,\chi:\varSigma\to\bbR\,$ obtained by restricting to $\,\varSigma$
the linear functional $\,\langle u,\,\cdot\,\rangle$, one then has
\begin{equation}\label{dch}
\mathrm{a)}\hskip6pt\Delta\chi=-na^{-\nnh1}\nh\chi\hh,\qquad
\mathrm{b)}\hskip6ptg(\nabla\nh\chi,\nnh\nabla\nh\chi)=-a^{-\nnh1}\nh\chi^2\nh
+\langle u,u\rangle\hh.
\end{equation}
In fact, Lemma~\ref{concr} yields (\ref{dch}-a), while (\ref{dch}-b) is also
immediate: at any $\,y\in\varSigma$, the $\,g$-grad\-i\-ent of $\,\chi\,$
is the component of $\,u\,$ tangent to $\,\varSigma$, with the norm squared
$\,\langle u,u\rangle-a^{-\nnh1}\nh\langle u,y\rangle^2\nh$.

On $\,M\nh=\varSigma\times\varSigma$, let $\,\bde\,$ and $\,\bna\,$
be the $\,\bg$-La\-plac\-i\-an and $\,\bg$-grad\-i\-ent, for 
the Riemannian product $\,\bg\,$ of two copies of the metric $\,g\,$
(which is also the sub\-man\-i\-fold metric of $\,M\,$ within
$\,\mathcal{V}\nh\times\mathcal{V}$). Then, for
the function 
$\,\psi:M\to\bbR\,$ arising as the restriction of $\,B\,$ to $\,M\,$ of 
any bi\-lin\-e\-ar form
$\,B:\mathcal{V}\nh\times\mathcal{V}\nh\to\bbR$, 
\begin{equation}\label{dps}
\mathrm{a)}\hskip6pt\bde\hh\psi=-\nh2na^{-\nnh1}\nh\psi\hh,\qquad
\mathrm{b)}\hskip6pt\bg(\bna\nnh\psi,\nnh\bna\nnh\psi)
=-\nh2a^{-\nnh1}\nh\psi^2\nh+F\nh,
\end{equation}
where $\,F:M\to\bbR\,$ is given by
$\,F(y,z)=\langle Ay,Ay\rangle+\langle A\nh^*\nnh z,A\nh^*\nnh z\rangle\,$ 
for the linear 
en\-do\-mor\-phism $\,A\,$ of $\,\mathrm{V}\hs$ such that 
$\,B=\langle A\hh\cdot\,,\,\cdot\,\rangle$. Namely, our claim trivially
follows from (\ref{dch}) 
applied to $\,\chi=\psi(\,\cdot\,,z)$, or $\,\chi=\psi(y,\,\cdot\,)$,
with fixed $\,z$, or $\,y$, and $\,u=A\nh^*\nnh z\,$ or, respectively, 
$\,u=Ay$.
\begin{rem}\label{eulor}Any polynomial function
$\,P:\mathcal{V}\nh\to\bbR\,$ vanishing on $\,\varSigma\,$ is polynomially
divisible by $\,S$. To see this, 
rather than invoking Hil\-bert's \emph{Null\-stel\-len\-satz}, let us fix
a line $\,t\mapsto x+tv\,$ intersecting $\,\varSigma\,$ at two points. Along
the line, $\,P\nh/S\,$ is a polynomial, 
and so $\,d\hh^k\nh[P\nh/S\hh]/\hn dt^k\nh=0\,$ for some $\,k$.
The same holds for all
nearby lines; thus, $\,P\nh/S\,$ (which is a smooth function on
$\,\mathcal{V}\hs$ according to Remark~\ref{divis}) has the
$\,k$\hh th directional
derivative along $\,v\,$ equal to zero at $\,x\,$ for all pairs
$\,(v,x)\,$ from a nonempty open set, and our conclusion is a consequence of 
real-an\-a\-lyt\-ic\-i\-ty.
\end{rem}
\begin{lem}\label{bothc}If\/ $\,(\hM\times\tM\nh,\hg+\tg)\,$ is a Riemannian
product of surfaces of constant scalar 
curvatures\/ $\,\tsc\,$ and\/ $\,\hsc\,$ with\/ $\,\hsc+\tsc\ne0\,$ and, for a
function\/ $\,\hi:\hM\times\tM\to\bbR$,
\begin{enumerate}
\item[{\rm(i)}]$\hna\nh d\hi\,$ and\/ $\,\tna\nh d\hi\,$ are functional 
multiples of\/ $\,\hg\,$ and\/ $\,\tg$,
\item[{\rm(ii)}]$2(\hY\nnh\nh-\hs\tY)\,=\,\hi\hh(\tsc-\hsc)$,
\item[{\rm(iii)}]$(1-\ve)(\hs\hsc+\tsc\hs)\hh\hi^2\nh+6\hi\hh(\hY\nnh+\tY)
=12(\hQ+\tQ)$, where\/ $\,\ve\in\bbR\smallsetminus\{1\}$,
\end{enumerate}
then\/ 
$\,\hi=\ha+\ta\,$ with some\/ $\,\ha\nnh:\nnh\hM\nh\to\bbR\,$ and\/ 
$\,\ta\nnh:\nnh\tM\nh\to\bbR$.
\end{lem}
\begin{proof}By (i) Lemma~\ref{concr},
$\,\hY\nnh\nh=\hde\hi=-\hsc\hh\hi+\eta_1\w$
and $\,\tY\nnh\nh=\tde\hi=-\tsc\hh\hi+\eta_2\w$, with $\,\eta_1\w$ (or, 
$\,\eta_2\w$) constant along $\,\hM$ (or, $\,\tM$), which is clearly also true 
when $\,\hsc\hs\tsc=0$. Now (ii) gives
$\,\hi\hh(\hsc-\tsc)=2(\eta_1\w-\eta_2\w)$. If $\,\hsc\ne\tsc$, our claim 
follows. In the remaining case, 
$\,\hsc=\tsc=2c\,$ and $\,\eta_1\w=\eta_2\w=2\hh c\hh\gm\,$ for some 
real $\,c\ne0\,$ and $\,\gm$.

Locally, we may identify both $\,(\hM\nh,\hg)\,$ and $\,(\tM\nh,\tg)\,$ with
$\,(\varSigma,g)\,$ defined above, for $\,a=c^{-\nnh1}\nh$. Using 
(i) and Lemma~\ref{concr} again, we see that 
$\,\hi(y,z)-\gm\,$ has, for each fixed $\,z\in\varSigma\,$ (or,
$\,y\in\varSigma$) an extension from $\,\varSigma\,$ to a linear functional
in the variable $\,y\in\mathcal{V}\hs$ (or, $\,z\in\mathcal{V}$), and this
unique extension $\,(y,z)\mapsto B(y,z)$, being linear in $\,y$ and
$\,z\,$ separately, is bi\-lin\-e\-ar,
leading to
(\ref{dps}) for $\,\psi=\hi-\gm$, with $\,\bde\hh\psi=\hY\nnh+\tY\hs$ and 
$\,\bg(\bna\nnh\psi,\nnh\bna\nnh\psi)=\hQ+\tQ$. Now (iii) states that
\begin{equation}\label{ole}
(1-\ve)\hh c\hh[B(y,z)+2\gm\hs]B(y,z)-3[\langle Ay,Ay\rangle
+\langle A\nh^*\nnh z,A\nh^*\nnh z\rangle]-6c\gm
\end{equation}
equals $\,0\,$ for all $\,y,z\in\varSigma\subseteq S^{-\nnh1}\nh(0)$, cf.\
the beginning of this section. Remark~\ref{eulor} for 
each fixed $\,z\in\varSigma\,$ (or, $\,y\in\varSigma$) shows that, on 
$\,\mathcal{V}\nh$, (\ref{ole}) simultaneously equals
$\,\langle y,y\rangle-a\,$ times a function of $\,z$, and
$\,\langle z,z\rangle-a\,$ times a function of $\,y$. Equating the latter two
products and separating variables, we get (\ref{ole}) equal to a constant 
$\,p$ times $\,[\langle y,y\rangle-a][\langle z,z\rangle-a]$. Applying
Remark~\ref{bideg} to the components of bi\-de\-gree $\,(2,2)$, we get
$\,(1-\ve)\hh c\hh[B(y,z)]^2\nh=p\langle y,y\rangle\langle z,z\rangle\,$ for
all $\,y,z\in\mathcal{V}\nh$. If $\,p\,$ were nonzero, choosing $\,y,z\,$ with 
$\,B(y,z)=0\ne \langle y,y\rangle\langle z,z\rangle\,$ we would get a
contradiction. 
Thus, $\,p\,$ equals $\,0$, and hence so does $\,B$, making $\,\hi\,$
constant, as $\,\hi(y,z)=\gm+B(y,z)$, which again yields our assertion.
\end{proof}

\section{Surface metrics}\label{sm}
\setcounter{equation}{0}
In this section $\,(\varSigma,g)\,$ is always a Riemannian surface, 
assumed oriented whenever we mention its com\-plex-struc\-ture tensor $\,J$.
Its K\"ah\-ler form $\,g(J\hs\cdot\,,\,\cdot\,)\,$ then equals the
area form of $\,g$. The following three conditions are mutually equivalent:
\begin{itemize}
\item[\rm{(i)}] $g\,$ is, locally, a warp\-ed-prod\-uct metric,
\item[\rm{(ii)}] there exists a no\-where-zero $\,g$-Kil\-ling field $\,w\,$ on
$\,\varSigma$,
\item[\rm{(iii)}] locally, $\,(\varSigma,g)\,$ admits a con\-cir\-cu\-lar
function $\,\alpha\,$ without critical points, cf.\ (\ref{cnc}),
and then we may choose 
$\,w\,$ in (ii) to be $\,Jv$, for $\,v=\nabla\nh\hi$.
\end{itemize}
About (i) -- (ii), see Remark~\ref{wpone}. In the case of (ii) -- (iii), 
the $\,g$-Kil\-ling property of $\,w\,$ (skew-ad\-joint\-ness of 
$\,B=\nabla\nh w:T\nh M\to T\nh M$) is, for dimensional reasons, nothing else
than $\,B=\sigma\nh J\,$ for some function $\,\sigma$, which in turn amounts
to $\,g(A\hs\cdot\,,\,\cdot\,)=\sigma\hn g$ for $\,A=-J\nh B=\nabla\nh v$,
with $\,v=-Jw\,$ being locally a gradient, as $\,A\,$ is self-ad\-joint.

Assuming (i), we get, in suitable local coordinates $\,(x^1\nh,x^2)=(t,y)$,
for some function $\,\khi\,$ of $\,t$, the Gaussian curvature
$\,K\hs$ of $\,g$, the gradient $\,\nabla\nh\ay$, 
Hess\-i\-an $\,\nabla\nh d\ay$ and La\-plac\-i\-an 
$\,\Delta\ay\,$ of any function $\,\ay\,$ depending only on $\,t$, with
$\,(\,\,)^{\boldsymbol{\cdot}}=d/dt$,
\begin{equation}\label{goe}
\begin{array}{rll}
\mathrm{a)}&g=dt^2\nh+e^{2\khi}\hn dy^2\nh,\,\qquad
\mathrm{b)}\hskip6ptK\nnh=-(\ddot\khi+\dot\khi^2),\quad
&\mathrm{c)}\hskip6ptg(\nabla\nnh\ay,\nnh\nabla\nnh\ay)
=\dot\ay^2\nh,\\
\mathrm{d)}&\nabla\nh d\ay\,
=\,\ddot\ay\,dt^2\hs+\,\dot\ay\dot\khi e^{2\khi}\hn dy^2,\quad
&\mathrm{e)}\hskip6pt\Delta\hn\ay\,=\,\ddot\ay\,+\,\dot\ay\dot\khi\hh,\\
\mathrm{f)}&2\nabla\nh d\ay=(\Delta\hn\ay)g\,\mathrm{\ if\ and\
only\ if\ }\,\ddot\ay=\dot\ay\dot\khi\hh,&\\
\mathrm{g)}&K\nh\dot\ay\,=-\nh\dddot\ay\,\mathrm{\ if\
}\,2\nabla\nh d\ay=(\Delta\hn\ay)g\hh.&
\end{array}
\end{equation}
In fact, (\ref{goe}-a) -- (\ref{goe}-e) follow, since one easily verifies that
\[%
\begin{array}{l}
g_{11}\w=\,g^{11}\nh=1,\quad g_{22}\w=\,e^{2\khi}\nh,\quad g^{22}\nh
=\,e^{-2\khi}\nh,\quad g_{12}\w=g^{12}\nh=\hs0,  \\
\vg_{\hskip-2.7pt11}^1\hs=\,\vg_{\hskip-2.7pt11}^{\hs2}\hs
=\,\vg_{\hskip-2.7pt12}^1\hs=\,\vg_{\hskip-2.7pt22}^{\hs2}\hs
=\,0,\qquad\vg_{\hskip-2.7pt12}^{\hs2}
=\dot\khi,\qquad\vg_{\hskip-2.7pt22}^1=-\dot{\khi}e^{2\khi}\nh,\\
R_{121}\w{}^2=-(\ddot\khi+\dot\khi^2),\hskip8.8pt\ay_{,11}\w
=\ddot\ay,\hskip8.8pt\ay_{,12}\w=0,\hskip8.8pt\ay_{,22}\w
=\dot\ay\dot\khi e^{2\khi}\nh.
\end{array}
\]%
Now (\ref{goe}-d) gives (\ref{goe}-f) and, by (\ref{goe}-b), $\,d/dt\,$ 
applied to (\ref{goe}-f) yields (\ref{goe}-g).
\begin{rem}\label{cncir}
Let $\,\alpha\,$ be a con\-cir\-cu\-lar function on $\,(\varSigma,g)$.
Then, locally, at points where $\,d\alpha\,$ is nonzero,
$\,Y\nnh\nh=\Delta\alpha\,$ and $\,Q=g(\nabla\nh\alpha,\nnh\nabla\nh\alpha)\,$ 
are functions of $\,\alpha$. Furthermore, in the open set where $\,dK\ne0$, 
such $\,\alpha\,$ is unique up to af\-fine replacements 
($p\hh\alpha+q\,$ with constants $\,p\ne0\,$ and $\,q$).

In fact, by (iii), $\,d_w\w Y\nnh\nh=d_w\w Q=0\,$
for the Kil\-ling field $\,w=Jv$, where $\,v=\nabla\nh\hi$, while
two con\-cir\-cu\-lar functions not satisfying an af\-fine
relation lead to two linearly independent Kil\-ling fields,
making $\,K\hs$ constant.
\end{rem}
\begin{lem}\label{twops}
If\/ $\,\alpha\,$ is a con\-cir\-cu\-lar function on\/ $\,(\varSigma,g)$,
then, locally, at points where\/ $\,\nabla\nh\alpha\neq0$,
in suitable coordinates\/ $\,x^1,\,x^2$, for\/ 
$\,(\,\,)^{\boldsymbol{\cdot}}={{{\partial}/{\partial x^1}}}$ and the
Gauss\-i\-an curvature\/ $\,K\nh$ of\/ $\,g$, one has\/
$\,\partial\hh\alpha/\partial x^2\nh=0\,$ and 
\begin{equation}\label{dae}
\Delta\alpha=2\ddot\alpha\hh,\quad g(\nabla\nh\alpha,\nabla\nh\alpha)
=\dot{\alpha}^2,\quad K=-\frac{\dddot{\alpha}}{\dot{\alpha}}.
\end{equation}
\end{lem}
This is obvious from (i) -- (iii), (\ref{goe}-c), and
(\ref{goe}-e) -- (\ref{goe}-g) for $\,\hi=\alpha$.
\begin{lem}\label{cgcrv}
Let $\,\Delta\alpha=2p\hh\alpha+2q\,$ and\/ 
$\,g(\nabla\nh\alpha,\nnh\nabla\nh\alpha)=p\hh\alpha^2\nh+2q\alpha+q'$ 
for a function\/ $\,\alpha\,$ on a Riemannian surface\/ 
$\,(M\nh,g)\,$ and constants\/ $\,p,q,q'\nh$. Then, wherever\/ 
$\,\nabla\nh\alpha\,$ is nonzero, $\,\nabla\nh d\alpha=(p\hh\alpha+q)g\,$ and\/
$\,g\,$ has the constant Gauss\-i\-an curvature\/ $\,K\nh=-p$.
\end{lem}
\begin{proof}
The Hessian $\hh\nabla\nh d\alpha\hs$ has the trace $\,2(p\hh\alpha+q)\,$ and,
wherever $\,\nabla\nh\alpha\ne0$, (\ref{dqe}-a)
shows that $\,\nabla\nh\alpha\,$ 
is an eigen\-vec\-tor of $\,\nabla\nh d\alpha\,$ for the eigen\-val\-ue 
$\,p\hh\alpha+q$. Thus $\,\nabla\nh d\alpha=(p\hh\alpha+q)g$, and our claim
follows from Remark \ref{rctev}.
\end{proof}

\section{Geometric realizations with Kil\-ling fields}\label{gk}
\setcounter{equation}{0}
We now show that Theorems~\ref{exatt} and~\ref{exaot} actually yield
examples of proper 
weakly Ein\-stein metrics, using Riemannian surfaces $\,(\varSigma,g)\,$ which 
admit no\-where-ze\-ro Kil\-ling fields. 
This last condition is necessary for (ii) in
Theorem~\ref{exatt} (see (ii) -- (iii) in Sect.\,\ref{sm}) but,
seemingly, not for (i).

In terms of (\ref{goe}), the assumptions of Theorem~\ref{exatt} imposed on
$\,\khi\,$ and $\,\hi\,$ are
\begin{equation}\label{cdi}
\begin{array}{rl}
\mathrm{i)}&(\ddot\hi+\dot\hi\dot\khi)\hi=2\dot\hi^2\nh
=(c+\ddot\khi+\dot\khi^2)\hi^2\mathrm{\nh\,\ due\ to\ (\ref{dqe}-b),\  in\
case\ (i),}\\
\mathrm{ii)}&\ddot\hi=\dot\hi\dot\khi\hh,\quad(c-\ddot\khi-\dot\khi^2)\hi^2\nh
+(\ddot\hi+\dot\hi\dot\khi)\hi=2\dot\hi^2\mathrm{\ \ in\ case\ (ii).}\\
\end{array}
\end{equation}
In (\ref{cdi}-i), or (\ref{cdi}-ii), using the first equality to eliminate
$\,\dot\khi$, we get
\begin{equation}\label{eli}
\begin{array}{rl}
\mathrm{i)}&\dot\khi+\ddot\hi/\dot\hi=2\dot\hi/\hi\,\mathrm{\ wherever\
}\,\hi\dot\hi\ne0\mathrm{,\ case\ (i)}\\
\mathrm{ii)}&\dot\hi=\ddot\hi/\dot\khi\,\mathrm{\ wherever\
}\,\dot\hi\ne0\,\mathrm{\ in\ case\ (ii).}\\
\end{array}
\end{equation}
Consequently, 
in the coordinates $\,(x^1\nh,x^2)=(t,y)$, 
(i) or, respectively, (ii) amounts to
the following third-or\-der ordinary differential equation, imposed on
$\,\hi\,$ alone:
\begin{equation}\label{phq}
\mathrm{i)}\hskip6pt\hi\dot\hi\nh\dddot\hi
=\,2(\hi\ddot\hi-\dot\hi^2)\ddot\hi\,+\,c\hi\dot\hi^2\mathrm{,\,\ or\ \ \ 
ii)}\hskip6pt\hi^2\nh\dddot\hi
=\,(2\hi\ddot\hi-2\dot\hi^2+c\hi^2)\dot{\hi}\hh.
\end{equation}
Solving \eqref{phq} with initial conditions such
that $\,\hi\dot\hi\ne0\,$ (which reflects Remark~\ref{noprd}),
we thus get examples of weakly Ein\-stein metrics
$\,\bg=(g+h)/\hi^2$ on $\,\varSigma\times\varPi\nh$. To evaluate the scalar
curvature $\,\bsc\,$ of $\,\bg\,$ and the spectrum of its Ein\-stein tensor
$\,\bei$, we apply 
(\ref{eee}) -- (\ref{xix}) and (\ref{goe}-b) -- (\ref{goe}-e) to
$\,\hg=g=dt^2\nh+e^{2\khi}\hn dy^2$ and
\begin{equation}\label{npq}
\begin{array}{l}
(n,p,q,\hM\nh,\tM\nh,\tg,\hsc,\tsc,\hY\nnh,\hQ,\tY\nnh,\tQ)\\
\hskip45pt=\,(4,2,2,\varSigma,\varPi\nh,h,2K,2c,\Delta\hi,
g(\nabla\nh\hi,\nnh\nabla\nh\hi),0,0),
\end{array}
\end{equation}
concluding, from (\ref{eli}) and (\ref{cdi}), that, in case (i) of
Theorem~\ref{exatt}
\begin{enumerate}
\item[{\rm(a)}] $(c-K)\hi^2\nh=\hi\Delta\hi=2\dot\hi^2$ and
$\,\nabla\nh d\hi=\ddot\hi\,dt^2\nh
-(\ddot\hi-2\dot\hi^2\nnh/\nh\hi)\hs e^{2\khi}\hn dy^2\nh$,
\item[{\rm(b)}] $\bsc=4(c\hh\hi^2\nh-\dot\hi^2)$, while 
$\,\hat\xi\hi^2\nh=-\nh2\dot\hi^2$ and $\,\tilde\xi=0$,
\item[{\rm(c)}] $\hi^2\hh\bei
=2(\hi\ddot\hi-\dot\hi^2)\hs(dt^2\nh-\hh e^{2\khi}\hn dy^2)$
along $\,\varSigma\,$ and $\,\bei=0\,$ along $\,\varPi$,
\end{enumerate}
and, similarly, in case (ii),
\begin{enumerate}
\item[{\rm(d)}] $(K\nnh\hn+\hs c)\hi^2\nh=2(\dot\hi^2\nh-\hi\ddot\hi)\,$ and
$\,\Delta\hi=2\ddot\hi$, while
$\,\nabla\nh d\hi=\ddot\hi\hs(dt^2\nh+\hh e^{2\khi}\hn dy^2)$,
\item[{\rm(e)}] $\bsc=8(\hi\ddot\hi-\dot\hi^2)$, while 
$\,\hat\xi\hi^2\nh=\dot\hi^2\nh-c\hh\hi^2\nh-2\hi\ddot\hi\,$ 
and $\,\tilde\xi=c\hh\hi^2\nh-\dot\hi^2\nh$,
\item[{\rm(f)}] $\hi^2\hh\bei\nh
=\nh(\dot\hi^2\nnh-\hn c\hi^2)(dt^2\nnh+e^{2\khi}\hn dy^2)$
along $\hs\varSigma\hs$ and $\hs\hi^2\hh\bei\nh=\nh(c\hh\hi^2\nnh
-\nh\dot\hi^2)\hh h\hs$
along $\hs\varPi\nnh$.
\end{enumerate}
In both cases, by (\ref{rco}), 
the factor distributions are $\,\bei$-or\-thog\-o\-nal.

Now, due to (v) in Theorem~\ref{exatt}, (b), (f), and (\ref{goe}-b) or,
respectively, (\ref{goe}-g), the con\-for\-mal flatness of $\,\bg=(g+h)/\hi^2$
is equivalent to 
\begin{equation}\label{cfl}
\dot\hi^2\nh=c\hh\hi^2\mathrm{\ \ in\ case\ (i),}\quad
\hi\ddot\hi=\dot\hi^2\mathrm{\ \ in\ case\ (ii).}
\end{equation}
Theorem~\ref{exatt}(iii)\hs-\hs(iv), along with (c) and (f) 
imply, in turn, that $\,\bg=(g+h)/\hi^2$ is Ein\-stein if and only if
\begin{equation}\label{eip}
\hi\ddot\hi=\dot\hi^2\mathrm{\ \ in\ case\ (i),}\quad
\dot\hi^2\nh=c\hh\hi^2\mathrm{\ \ in\ case\ (ii).}
\end{equation}
The special solutions with (\ref{cfl}) or (\ref{eip}) 
form, within the
three-di\-men\-sion\-al man\-i\-fold of solutions to \eqref{phq} having
$\,\hi\dot\hi\ne0$, submanifolds of positive co\-di\-men\-sions.

The remaining ``generic'' solutions \emph{lead to 
proper weakly Ein\-stein metrics}. 
\begin{rem}\label{hmthi}According to (c), (f) and the line following (f),
the unordered spectrum of $\,\bei\,$ consists of
$\,\pm\nh2(\hi\ddot\hi-\dot\hi^2),\hs0,\hs0\,$ in case (i), and 
$\,\pm(\dot\hi^2\nnh-\hn c\hi^2)$, each repeated twice, in case (ii).
\end{rem}
\begin{rem}\label{treun}The proper weakly Ein\-stein
$\,2+2\,$ con\-for\-mal-prod\-uct metrics $\,\bg\,$
constructed above have the form
$\,(dt^2\nh+e^{2\khi}\hn dy^2\nh+h)/\hi^2\nh$, where $\,h\,$ has the constant
Gauss\-i\-an curvature $\,c\,$ and $\,\chi,\hi\,$ are functions of $\,t$.
Thus they obviously constitute con\-for\-mal $\,3+1\,$ products, with the
factor metrics $\,e^{-\nh2\khi}\nh(dt^2\nh+h)\,$ and $\,dy^2\nh$, the
Riemannian product of which is divided by the square of
$\,\hi\hs e^{-\chi}\nh$.

If $\,c=0$, then -- as in the final clause of Remark~\ref{nnunq} --
writing, locally, $\,h=\hh d\eta^2\nh+\hh d\zeta^2\nh$, we obtain
a further $\,3+1\,$ con\-for\-mal-prod\-uct decomposition of $\,\bg$,
the factor metrics being this time 
$\,dt^2\nh+e^{2\khi}\hn dy^2\nh+\hh d\eta^2$ and $\,d\zeta^2\nh$.
\end{rem}

\section{Geometric realizations via con\-for\-mal changes}\label{gc}
\setcounter{equation}{0}
We again use (\ref{npq}) and replace $\,g\,$ by $\,\hg\,$ 
to simplify ref\-er\-ences to Sect.\,\ref{cc}.

In the case (i) of Theorem~\ref{exatt}, rather than insisting on the existence
of a nontrivial Kil\-ling field, we can also proceed by expressing $\,\hg$,
locally, as $\,\hg=e^{2\sg}\nh  g$, where $\,g=\hh dx^2\nnh+\hh dy^2$ 
is now a flat metric. By (\ref{crm}-i) with $\,\hi=\hh e^{-\sg}$ and
(\ref{ndf}-b),
\begin{equation}\label{hse}
\mathrm{i)}\hskip6pt\hsc\,=\,-\nh2\hh e^{-\nh2\sg}\nh\Delta\sg\hh,\qquad
\mathrm{ii)}\hskip6pt\hde\,=\,\hh e^{-\nh2\sg}\nh\Delta\hh.
\end{equation}
Thus, the condition (i) states that $\psi=\hi^{-1}$ is a
harmonic function without zeros,
while, due to (\ref{ndf}-b) and (\ref{hse}-i),
\begin{equation}\label{lps}
\Delta\sg\hs+\,c\hs e^{2\sg}
= \,2g(\nabla\nh\log\nh|\psi|,\nnh\nabla\nh\log\nh|\psi|)\hh.
\end{equation}
Proper weakly Einstein metrics $\,\bg=(\hg+\tg)/\hi^2$ 
will arise if one chooses $\,\psi\,$ to
be harmonic and then finds $\,\sg\,$ with 
\eqref{lps}, making sure, according to
(iii) -- (v) in Theorem~\ref{exatt}, that
\begin{equation}\label{mks}
\mathrm{the\ vector\ field\ }\,e\hn^{-\nh2\sg}\hh\nabla\nh\psi^{-1}\mathrm{\
is\ not\ con\-for\-mal,\ and\ }\,\Delta\sg\ne\hs c\hs e^{2\sg}\nnh.
\end{equation}
Let us now set $\,\psi(x,y)=e^x\nh\cos y\,$ 
in Cartesian coordinates $\,x,y$, so that $\,\psi\,$ is harmonic. We solve
(\ref{lps}) for $\,\sg\,$ assumed to be a function of $\,y$, which amounts to
the sec\-ond-or\-der ordinary differential equation 
\begin{equation}\label{ode}
\sg''(y)\,+\,c\hs e^{2\sg(y)}\,=\,2\sec^2\hskip-2pty\hh.  
\end{equation}
Note that, in the case where
\begin{equation}\label{log}
\sg'(y)=\tan y\,\mathrm{\ or,\ equivalently,\
}\,y\mapsto\sg(y)+\log|\hskip-1.4pt\cos y|\,\mathrm{\ is\ constant,}
\end{equation}
equation (\ref{ode}) fails to hold if $\,c\le0$, but does hold when
$\,c>0$, provided that the constant in (\ref{log}) is 
$\,-\hn\log\nh\sqrt{c\,}$. For $\,\zeta(y)=\sg'(y)-\tan y$, (\ref{ode}) yields
\begin{equation}\label{zpr}
\begin{array}{rl}
\mathrm{i)}&\zeta'\nh(y)=\sec^2\hskip-2pty-c\hs e^{2\sg(y)}\nh,\quad
\mathrm{ii)}\hskip6pt\zeta''\nh(y)=2\hs\zeta'\nh(y)\tan y
-2c\hs e^{2\sg(y)}\nh\zeta(y)\hh,\\
\mathrm{iii)}&\mathrm{so\ that\ }\,\zeta\,\mathrm{\ is\ nonconstant\ 
as\ long\ as\ we\ exclude\ the\ case\ (\ref{log}),}
\end{array}
\end{equation}
since constancy of $\,\zeta\,$ and (\ref{zpr}-i) would give $\,c>0\,$ and
(\ref{log}) with the constant $\,-\hn\log\nh\sqrt{c\,}$.
Solutions $\,\sg\,$ of (\ref{ode}), except those with (\ref{log}), lead to 
proper weakly Ein\-stein metrics: namely, we have (\ref{mks}). In fact,
(\ref{ode}) with $\,\sg''(y)=\Delta\sg=\hs c\hs e^{2\sg}$ would give 
(\ref{log}), 
while if $\,w=e\hn^{-\nh2\sg}\hh\nabla\nh\psi^{-1}$ were a
con\-for\-mal
vector field, the component equality $\,w\hn_{1,2}\w+w\hn_{2,1}\w=0\,$ 
for its covariant derivative, in the coordinates $\,x,y$, would
read $\,2\hh[\sg'(y)-\tan y]\hs e^{-x-2\sg(y)}\nh\sec y=0$, again implying
(\ref{log}).

The surface metrics $\,g\,$ obtained here still admit nontrivial Kil\-ling
fields, provided by the coordinate vector field $\,\partial/\partial x$,
since the con\-for\-mal factor, multiplied by the standard flat metric so
as to yield $\,g$, does not depend on $\,x$. However, in contrast with the
examples described in Sect.\,\ref{gk}, $\,\hi=\psi^{-\nnh1}$ and
he con\-for\-mal factor are now functionally independent.

As $\,\hi=\hh e^{-x}\nh\sec y$, using subscripts for partial
derivatives we get
\begin{enumerate}
\item[{\rm(a)}] 
$(\hi\nh_x\w,\hi\nh_y\w,\hi\nh_{xx}\w,\hi\nh_{xy}\w,
\hi\nh_{yy}\w)=\hi\hh(-\nnh1,\tan y,1,-\tan y,1+2\tan^2\hskip-2pty)$.
\end{enumerate}
With the notation of
Remark~\ref{denot}, (a) and (\ref{hse}-ii) give 
$\,\hde\hh\hi=2\hh e^{-x-2\sg(y)}\nh\sec^3\hskip-2pty$, 
so that, by (\ref{ndf}-a) with $\,\chi=\hh e^{-\sg}\nh$, (\ref{eee}-a) and
(\ref{xix}),
\begin{enumerate}
\item[{\rm(b)}] $\hna\hn d\hi=\nabla\nh d\hi-d\hh\sg\otimes\hh d\hi
-d\hi\otimes\hh d\hh\sg+g(\nabla\nh\sg,\nabla\nnh\hi)\hh g$,
\item[{\rm(c)}] $\bsc
=4(c-\hh e^{-\nh2\sg(y)}\nh\sec^2\hskip-2pty)\hs
e^{-\nh2x}\nh\sec^2\hskip-2pty$,
\item[{\rm(d)}] $\hat\xi=-\nh2\hh e^{-\nh2\sg(y)}\nh\sec^2\hskip-2pty\,$
and $\,\tilde\xi=0$.
\end{enumerate}
Due to 
(a) and (b), $\,\hna\hn d\hi$, in the coordinates $\,x,y$, equals
$\,\hi\,$ times the symmetric matrix with the diagonal
$\,(1+\sg'(y)\tan y,1+2\tan^2\hskip-2pty-\sg'(y)\tan y)\,$ and the
off-di\-ag\-o\-nal entry $\,\sg'(y)-\tan y$. In view of (d), (\ref{eee}-b) and
(\ref{eee}-c), the analogous data for
$\,\bei/2\,$ along $\,\hM\,$ are: diagonal
$\,(\sg'(y)\tan y-\tan^2\hskip-2pty,
\tan^2\hskip-2pty-\sg'(y)\tan y)$, off-di\-ag\-o\-nal $\,\sg'(y)-\tan y$,
while $\,\bei=0\,$ along $\,\tM\nh$. As the matrix of $\,\bei/2\,$ along $\,\hM\,$
is trace\-less, with the determinant 
$\,-[\sg'(y)-\tan y]^2\nh\sec^2\hskip-2pty$, the $\,g$-spec\-trum of
$\,\bei/2\,$ along $\,\hM$ consists of
$\,\pm[\sg'(y)-\tan y]\sec y$.
\begin{rem}\label{sglbd}According to the preceding two sentences and (c),
the lo\-cal-homo\-thety invariant (\ref{lhi}-i), for the metrics $\,\bg\,$
constructed above, has $\,\pm\lambda\,$ equal to 
$\,2[\sg'(y)-\tan y]\hh e^{-2x-2\sg(y)}\nh\sec^3\hskip-2pty\,$ and 
$\,\bsc=4(c-\hh e^{-\nh2\sg(y)}\nh\sec^2\hskip-2pty)\hs
e^{-\nh2x}\nh\sec^2\hskip-2pty$, so that, by
(\ref{zpr}-i), $\,\bsc=-4\hh\zeta'\hn e^{-\nh2\sg}\hn\hi^2$
and $\,\pm\lambda=2\hh\zeta\hh e^{-\nh2\sg}\hn\hi^2\nh\sec y$. Thus,
due to (\ref{zpr}-iii), $\,\bsc\hh\lambda\ne0$. Next, $\,\bsc\,$ \emph{and\/ 
$\,\lambda\,$ are functionally independent}, which follows since the
ratio $\,\bsc/\nnh\lambda=\mp\hh2\hh\zeta^{-\nnh1}\nh\zeta'\nnh\cos y$, a
function of $\,y$, is nonconstant (allowing us to express $\,y$, and then also
$\,x$, in terms of $\,\bsc\,$ and $\,\lambda$). Namely, if 
$\,2p=\zeta^{-\nnh1}\nh\zeta'\nnh\cos y\,$ were constant, with $\,p\ne0\,$ 
by (\ref{zpr}-iii), 
differentiating the equality $\,\zeta'\nh=2p\hs\zeta\nh\sec y\,$ we would
obtain 
$\,\zeta''\nh=4p^2\zeta\nnh\sec^2\hskip-2pty+2p\hs\zeta\nh\sec y\tan y\,$
while, at the
same time, from (\ref{zpr}-ii),
$\,\zeta''\nh=4p\hs\zeta\nh\sec y\tan y-2c\hs e^{2\sg}\hn\zeta$. Equating the
two expressions for $\,\zeta''$ we get
$\,c\hs e^{2\sg}\nh=p\sec y\tan y-2p^2\nh\sec^2\hskip-2pty$. Thus, 
(\ref{zpr}-i) reads $\,\zeta'\nh
=(2p^2\nh+1)\sec^2\hskip-2pty-p\sec y\tan y$, and so $\,\zeta\,$ equals 
$\,(2p^2\nh+1)\tan y-p\hh\sec y\,$ plus a constant. As
$\,\zeta'\nh=2p\hs\zeta\nh\sec y$, it now follows that
$\,4p^2\nh+1-p(4p^2\nh+3)\sin y\,$ is equal to a constant times $\,\cos y$,
which is the required contradiction.
\end{rem}

\section{Clas\-si\-fi\-ca\-tion-type theorems}\label{cl}
\begin{thm}\label{cpsfm}
Any weakly Ein\-stein metric in di\-men\-sion four, con\-for\-mal to a product 
of surface metrics, arises, locally at points where\/
$\,\ein\ne0\,$ and\/ $\,W\nnh\nh\ne\hs0$, from one of the constructions
described in Theorem\/~{\rm\ref{exatt}}.
\end{thm}
To provide some context for the next theorem, let us recall two facts, one
stated as 
(\ref{wpm}) -- (\ref{wij}), the other consisting of Remark~\ref{cvrsl} 
and (\ref{bts}-a).  They refer to what happens 
at every point $\,x\,$ of an oriented con\-for\-mal-prod\-uct Riemannian
four-man\-ifold with the metric $\,\bg$, scalar curvature $\,\bsc$, and 
Weyl tensor $\,\overline W\nnh$. First,
\begin{equation}\label{sms}
\mathrm{both\ }\,\overline W^\pm\nnh\mathrm{\ have\ the\ same\ unordered\
spectrum\ at\ }\,x\hh.
\end{equation}
Secondly, \emph{if\/ $\,\bg\,$ is weakly Ein\-stein, while 
$\,\overline W^+\hskip-4pt$ is nonzero and has a repeated eigen\-valu\-e at\/
$\,x$, then, at\/} $\,x$,
\begin{equation}\label{uns}
\mathrm{the\ unique\ simple\ eigen\-val\-ue\ of\
}\,\overline W^+\nnh\nh\mathrm{\ equals\nh\ }-\nnh\nh\bsc/3\mathrm{,\ or\nh\ 
}-\nnh\nh\bsc/\nh12\mathrm{,\ or\ }\,\hs\bsc/6\hh.
\end{equation}
\begin{thm}\label{cptho}Let\/ $\,\bg\,$ 
be a weakly Ein\-stein\/ $\,3\hh+\nh1\,$ con\-for\-mal-prod\-uct metric on\/
$\,M\nnh\times\hn I\nh$, 
where\/ $\,I\subseteq\bbR\,$ is an open interval. If, for a fixed local
orientation and a point\/ $\,x\in M\nnh\times\hn I\nh$, either
\begin{enumerate}
\item[{\rm(i)}]$\overline W^+\nnh\hn$ has three distinct eigen\-valu\-es at\/
$\,x$, or
\item[{\rm(ii)}]$\overline W^+\hskip-4pt\ne\hs0\,$ at\/ $\,x\,$ and\/
$\,\overline W^+\nnh$ has a repeated eigen\-valu\-e at all points of some
neighborhood of\/ $\,x$, with its sim\-ple-eigen\-val\-ue function equal to\/
$\,-\bsc/\nh12$,
\end{enumerate}
then\/ $\,\bg\,$ arises near\/ $\,x\,$ as in Theorem\/~{\rm\ref{exaot}}, from
some data\/ $\,g,v,\hi,\kp$. In both cases,
\begin{equation}\label{evs}
\mathrm{at\ every\ point,\ }\,-\nnh\bsc/\nh12\,\mathrm{\ is\ an\
eigenvalue\ of\ }\,\overline W^\pm\nnh\nh.
\end{equation}
\end{thm}
Under the hypothesis preceding (i) in Theorem~\ref{cptho}, at points 
with $\,\overline W^+\hskip-5pt\ne0$ that are  
generic relative to the spectrum of $\,\overline W^+\nnh\nnh$, in the
sense of Remark~\ref{genrc}, one must -- due to (\ref{uns}) -- have either
(i), or a version of (ii) in which the simple eigen\-val\-ue
$\,-\nnh\nh\bsc/\nh12\,$ is replaced, possibly, 
by $\,-\nnh\nh\bsc/3\,$ or $\,\bsc/6$.

Using a different (and more standard) notion of genericity, one may thus
say that Theorem~\ref{cptho} covers the generic situation, in (i), as well as
``one-third'' -- represented by (ii) -- of the nongeneric case.

\section{Three lemmas needed for Theorem~\ref{cpsfm}}\label{tl}
\setcounter{equation}{0}
Whenever $\,(\hM\times\tM\nh,\hg+\tg)\,$ is
a Riemannian product of oriented surfaces 
and $\,g=(\hs\hg+\tg\hs)\hs/\hi^2$ on $\,\hM\times\tM\nh$,
where $\,\hi:\hM\times\tM\to\bbR\smallsetminus\{0\}$, (\ref{eee}) --
(\ref{xix}) for $\,(n,p,q)=(4,2,2)\,$ give, for the Ein\-stein tensor 
$\,\ein=\ric-\sca\hs g/4$,
\begin{equation}\label{rwr}
\begin{array}{rl}
\mathrm{i)}&4\ein=8\hi^{-\nh1}\hs\hna\nh d\hi+
\hi\hs[\hi(\hsc-\tsc)-2(\hY\nnh+\tY)]\hs g\hskip10pt\mathrm{\ along\ }\,\hM,\\
\mathrm{ii)}&4\ein=8\hi^{-\nh1}\hs\tna\nh d\hi+
\hi\hs[\hi(\tsc-\hsc)-2(\hY\nnh+\tY)]\hs g\hskip10pt\mathrm{\ along\ }\,\tM,\\
\mathrm{iii)}&\sca\hh=\hh\hi^2(\hsc+\tsc\hs)+6\hi(\hY\nnh+\tY)\hh
-12(\hQ+\tQ)\mathrm{,\hskip12ptwith\ (\ref{cnv}).}\\
\end{array}
\end{equation}
We will repeatedly assume that, for $\,\hM\nh,\tM\nh,\hg,\tg\,$ 
and $\,\hi\,$ as above,
\begin{equation}\label{asm}
\begin{array}{l}
g=(\hs\hg+\tg\hs)\hs/\hi^2\nh\mathrm{\hn\ is\hn\ a\hn\ weakly\hn\
Ein\-stein\hn\ metric\hn\ on\ the\hn\ oriented}\\
\mathrm{four}\hyp\mathrm{manifold\
}\,\,M\hh=\hs\hM\times\tM\nh\mathrm{,\ where\ }\,\dim\hM\hn=\hs\dim\tM\nh=2\hh.
\end{array}
\end{equation}
As before, $\,\sca\,$ is the scalar curvature, and 
$\,\ric,\ein,W\nnh$ the Ric\-ci, Einstein and Weyl 
tensors of the metric $\,g$, the obvious modified versions of these symbols 
denote their analogs for $\,\hg\,$ and $\,\tg$, while 
$\,\hY\nnh,\tY\nnh,\hQ,\tQ\,$ equal
$\,\hde\hs\hi,\,\tde\hs\hi,\,\hg(\hna\nh\hi,\nnh\hna\nh\hi)\,$ and
$\,\tg(\tna\nh\hi,\nnh\tna\nh\hi)$.
\begin{lem}\label{smeps}Under the assumption\/ {\rm(\ref{asm})}, in every
connected component\/ $\,\,U$ of the open set in\/ $\,M\,$ where\/
$\,\ein\ne0\,$ and\/ $\,W\nnh\nh\ne\hs0$, for some constant\/ 
$\,\ve\in\{-\nh2,-\frac12,1\}$,
\begin{enumerate}
\item[{\rm(a)}]$\sca\,=\,\ve\hn\hi^2(\hs\hsc+\tsc\hs)\,$ and\/
$\,\hsc+\tsc\ne0\,$ everywhere in\/ $\,U\nh$,
\item[{\rm(b)}]$(1-\ve)\hh\hi^2(\hs\hsc+\tsc\hs)+6\hi\hh(\hY\nnh+\tY)
=12(\hQ+\tQ)$,
\item[{\rm(c)}]$2(\hY\nnh\nh-\tY)\,=\,\hi\hh(\tsc-\hsc)\,$
if\/ $\,\ve\in\{-\frac12,1\}$,
\item[{\rm(d)}]$\hna\nh d\hi\,$ and\/ $\,\tna\nh d\hi\,$ are functional 
multiples of\/ $\,\hg\,$ and\/ $\,\tg\,\hs$ if\/ $\,\ve\in\{-\nh2,-\frac12\}$,
\item[{\rm(e)}]$\ve\in\{-\nh2,1\}\,$ if and only if\/
$\,T\nh\hM\,$ and\/ $\,T\nh\tM\,$ are\/ $\,\ein$-or\-thog\-o\-nal,
\item[{\rm(f)}]$\ve=-\frac12\,$ if and only if\/ $\,\ein\,$ restricted to\/ 
both\/ $\,T\nh\hM\,$ and\/ $\,T\nh\tM\,$ yields\/ $\,0$.
\end{enumerate}
\end{lem}
Let $\,\hd\hi,\,\td\hi\,$ be the ``partial 
differentials'' of $\,\hi$, and $\,\hd\td\hi$ a $\,(0,2)\,$
tensor, with the components in product coordinates $\,x^i\nh,x^a$ 
given by
$\,\partial\nh_i\w\hi\,$ and $\,\partial\nh_a\w\hi\,$ for
$\,\hd\hi\,$ and $\,\td\hi\,$, as well as 
$\,\mathrm{d}_{ia}\w=\mathrm{d}_{ai}\w=
\partial\nh_i\w\partial\nh_a\w\hi\,$ and
$\,\mathrm{d}_{ij}\w=\mathrm{d}_{ab}\w=0\,$ for $\,\mathrm{d}=\hd\td\hi$.

We will eventually show, in Lemma~\ref{third}, that Lemma~\ref{smeps}(f) only
holds vacuously, since $\,\ve\ne-\frac12$. As an intermediate step, we first
establish some conclusions, valid when $\,\ve=-\frac12$, so as
to use them later in deriving a contradiction.
\begin{lem}\label{prdwz}
For\/ $\,M\nh,g,U\nh,\hi,\hh\hsc\hh,\hs\tsc\,$ and\/ $\,\ve\,$ as in
Lemma\/~{\rm\ref{smeps}}, if\/ $\,\ve\in\{-\nh2,1\}$, then\/ 
$\,\hd\hi\hh\otimes\td\hi=\hs0\hs\,$ identically on\/ $\,U\nh$, while 
if\/ $\,\ve=-\frac12$, then\/ $\,d\hs\hsc\hh\otimes d\hs\tsc=0\,$ and\/ 
$\,\hd\td\hi\ne0\,$ everywhere in\/ $\,U\nh$.
\end{lem}
\begin{lem}\label{third}
Under the hypotheses of Lemma\/~{\rm\ref{smeps}}, $\,\ve\ne-\frac12$.
\end{lem}

\section{Proofs of the first two lemmas}\label{pf}
\setcounter{equation}{0}
\begin{proof}[Proof of Lemma~\ref{smeps}]
By (\ref{wpm}), at any point of $\,\,U\nh$, we have (\ref{bts}) and
-- consequently, in view of Remark~\ref{cvrsl}  -- one of the nine cases in 
(\ref{chb}) where, at the end of each line, we also provide the resulting 
ordered spectra of $\,24\hh W^+$ and $\,24\hh W^-\nnh\nh$. Equating each
simple eigen\-val\-ue in (\ref{chb}) with the one in (\ref{wpm}), we obtain
the assertion (a) where, with the same order as in (\ref{chb}) and (\ref{nbs}),
\begin{equation}\label{srt}
\ve\,\,\mathrm{\ equals,\ respectively,\ 
}\,-\hskip-4pt2,\textstyle{-\frac12,-\frac12,-\frac12,\,1,\,1,-\frac12},\,
1,\,1.
\end{equation}
Now (\ref{srt}) and Lemma~\ref{othzr} yield (e) -- (f), the `if' parts 
obvious for logical reasons as $\,\ein\ne0$. Next, (d) for
$\,\ve\in\{-\frac12,-\nh2\}\,$ 
follows from (f) and (\ref{rwr}) or, respectively, from (\ref{srt}),
(\ref{nbs}), (\ref{rwr}) and (\ref{ari}-c). Finally, we get 
(c) from Remark~\ref{eight}, evaluating either $\,g$-trace 
in (\ref{rwr}), while (a) and (\ref{rwr}-iii)
imply (b).
\end{proof}
We prove Lemma~\ref{prdwz} by cases: (i) with $\,\ve\in\{-\nh2,1\}$, and 
(ii) with $\,\ve=-\frac12$.
\begin{proof}[Proof of Lemma~\ref{prdwz}{\rm(i)}]
We derive a contradiction from the assumption that the subset $\,\,U'$
of $\,\,U\,$ on which $\,\hd\hi\,$ and $\,\td\hi\,$ are both nonzero
is nonempty.

By Lemma~\ref{smeps}(e) and (\ref{rco}), $\,\hi:U\to\bbR\smallsetminus\{0\}\,$ 
has additively separated variables: 
$\,\hi=\ha+\ta\,$ with $\,\td\hh\ha=\hd\hh\ta=0$.

If $\,\ve=1$, (b) and (c) in Lemma~\ref{smeps} easily imply, via a 
sep\-a\-ra\-tion-of-var\-i\-ables argument, that, for some
constants $\,p_1\w,p_2\w,q_1\w,\dots,q_5\w$, locally in $\,\,U'\nh$,
\[
\begin{array}{l}
\hY\nnh=2p_1\w\ha+2q_1\w\hh,\quad\tY\nnh=-2p_1\w\ta+2q_2\w\hh,\quad
\hsc=p_2\w\ha+q_3\w\hh,\quad\tsc=p_2\w\ta+q_4\w\hh,\\
\hQ=p_1\w\ha^2\nh+(q_1\w+q_2\w)\ha+q_5\w,
\qquad
\tQ=-p_1\w\ta^2\nh+(q_1\w+q_2\w)\ta-q_5\w.
\end{array}
\]
Plugging this into the equality of Lemma~\ref{smeps}(c) we see that, as 
$\,d\ha,d\ta\,$ are nonzero, 
$\,p_2\w=4p_1+q_3\w-q_4\w=4(q_1\w-q_2\w)=0$, and the above displayed formula
yields
\[
\begin{array}{ll}
\hY\nnh=2p_1\w\ha+2q_1\w\hh,&\quad\hQ=p_1\w\ha^2\nh+2q_1\w\ha+q_5\w
\\
\tY\nnh=-2p_1\w\ta+2q_1\w\hh,&\quad\tQ=-p_1\w\ta^2\nh+2q_1\w\ta-q_5\w.
\end{array}
\]
Lemma~\ref{cgcrv} applied first to 
$\,(g,\alpha,p,q,q')=(\hg,\ha,p_1\w,q_1\w,q_5\w)$, then to
$\,(g,\alpha,p,q,q')=(\tg,\ta,\nh-p_1\w,q_1\w,\nh-q_5\w)$, thus gives 
$\,(\hsc,\tsc)=(-\nh2p_1\w,2p_1\w)$, and so, by (\ref{wpm}),
$\,W\nnh\nh=\hs0$, contrary to the definition of $\,\,U\nh$, which
proves our claim in the case $\,\ve=1$.

Suppose now that $\,\ve=-\nh2$. 
Locally in $\,\,U'\nh$, applying 
Lemma \ref{twops} to both pairs $\,(\hg,\ha)\,$ and $\,(\tg,\ta)$, which
is allowed by Lemma~\ref{smeps}(d), we get
coordinates $\,x^1\nh,x^2$ for $\,\hM\,$ and $\,x^3\nh,x^4$ for
$\,\tM\,$ such that, by Lemma~\ref{smeps}(b) and \eqref{dae},
\begin{equation}\label{apb}
(\alpha+\beta)^2(\dot\beta\nh\dddot\alpha+\dot\alpha\nh\dddot\beta)\,=
\,2\hh\dot\alpha\dot\beta\hh[(\alpha+\beta)(\ddot\alpha+\ddot\beta)
-\dot\alpha^2\nh-\dot\beta^2].
\end{equation}
Our notation here is simplified: $\,\alpha\,$ stands for $\,\ha\,$ and
$\,\beta\,$ for $\,\ta$, while $\,(\,\,)\dot{\,}$ denotes
$\,\partial/\partial x^1$ when applied to functions of $\,x^1\nh$, and
$\,\partial/\partial x^3$ for functions of $\,x^3\nh$. 

Note that, due to the definition of $\,\,U'\nh$, as 
$\,\hi=\alpha+\beta:U\to\bbR\smallsetminus\{0\}$, 
\begin{equation}\label{nzr}
\dot\alpha\dot\beta(\alpha+\beta)\,\ne\,0\,\mathrm{\ everywhere\ in\ 
}\,\,U'\nh.
\end{equation}
We can consequently rewrite \eqref{apb} as 
\begin{equation*}
  \frac{\dddot{\alpha}}{2\dot{\alpha}}{{-}}\frac{\ddot{\alpha}}{\alpha+\beta}{{+}}\frac{\dot{\alpha}^2}{(\alpha+\beta)^2}={{\frac{\ddot{\beta}}{\alpha+\beta}-\frac{\dot{\beta}^2}{(\alpha+\beta)^2}}}+{\frac{\kp}{2}},\quad
\mathrm{where\ \ }{{\kp}}=-\frac{\dddot{\beta}}{{\dot{\beta}}}, 
\end{equation*}
and apply $\,\dot{\beta}^{-1}(\alpha+\beta)^2\partial/\partial x^3$ twice in
a row, first obtaining 
\begin{equation*}
\ddot{\alpha}-\frac{2\dot{\alpha}^2}{\alpha+\beta}={{\frac{{{\dot\kp}}}{2\dot{\beta}}(\alpha+\beta)^2}}-{{\kp}}(\alpha+\beta)+\frac{2\dot{\beta}^2}{\alpha+\beta}-3\ddot{\beta},
\end{equation*}
and then
\begin{equation*}
2\dot\alpha^2
=\,\frac{{({{\dot{\kp}}}/\dot{\beta})}\dot{\,}}
{2\dot\beta}(\alpha+\beta)^4\nh
+2\kp(\alpha+\beta)^2\nh+4\ddot\beta(\alpha+\beta)-2\dot{\beta}^2.
\end{equation*}
Now, applying $\,2(\alpha+\beta)^{-2}{{{\partial}/{\partial x^3}}}$, we see
that 
\begin{equation*}
[(\dot\kp/\dot\beta)/\dot\beta]\nnh^{^{\boldsymbol{\cdot}}}(\alpha
+\beta)^2\nh+4(\dot\kp/\dot\beta)\nh^{\boldsymbol{\cdot}}(\alpha+\beta)
+4\dot\kp=0.
\end{equation*}
It follows that $\,\kp\,$ is constant, or else $\,\alpha+\beta$, being a
root of a nontrivial quadratic equation with coefficients that are functions 
of $\,x^3\nh$, would itself be a function of $\,x^3\nh$, even though
$\,d\ha\hn\ne0$. Thus, in \eqref{apb}, $\,2\hskip-1.7pt\dddot\beta=-\sca_1\w\dot\beta\,$
for some constant $\,\sca_1\w$.

Since \eqref{apb} involves $\,\alpha\,$ and
$\,\beta\,$ symmetrically, we also have
$\,2\hskip-1.1pt\dddot\alpha=-\sca_0\w\dot\alpha\,$ for some
$\,\sca_0\w\in\bbR$, while, by \eqref{dae},
\begin{equation}\label{sop}
\sca_0\w\,\mathrm{\ and\ }\,\,\sca_1\w\,\mathrm{\ are\ the\ scalar\
curvatures\ of\ }\,\hs\hg\hs\,\mathrm{\ and\ }\,\hs\tg. 
\end{equation}
With $\,2\hskip-1.7pt\dddot\alpha,2\hskip-1.7pt\dddot\beta\,$ replaced by
$\,-\sca_0\w\dot\alpha\,$ 
and $\,-\sca_1\w\dot\beta$, \eqref{apb} reads
\begin{equation*}
(\sca_0\w+\sca_1\w)(\alpha+\beta)^2\hs=\,4[\dot\alpha^2\nh+\dot\beta^2\nh
-(\alpha+\beta)(\ddot\alpha+\ddot\beta)],
\end{equation*}
as (\ref{nzr}) allows us to divide by $\,\dot\alpha\dot\beta$. In other
words,
\begin{equation*}
\ddot\alpha-\frac{\dot\alpha^2}{\alpha+\beta}\,=
\,-\frac{\sca_0\w+\sca_1\w}4\hs(\alpha+\beta)+\frac{\dot\beta^2}{\alpha+\beta}
-\ddot\beta.
\end{equation*}
Applying $\,\dot\beta^{-\nh1}(\alpha+\beta)^2\partial/\partial x^3$ to 
the last equality, we obtain 
\begin{equation*}
\dot\alpha^2\hs=\,\frac{\sca_1\w-\sca_0\w}4\hs(\alpha+\beta)^2\nh
+2\ddot\beta(\alpha+\beta)-\dot\beta^2
\end{equation*}
and, further applying $\,4\hs\partial/\partial x^3\nh$, we get
$\,0=-2(\sca_0\w+\sca_1\w)(\alpha+\beta)\dot\beta$. Hence
$\,W\nnh\nh=\hs0$ according to (\ref{nzr}) -- \eqref{sop} and (\ref{wpm}),
which is the required  contradiction.
\end{proof}
\begin{proof}[Proof of Lemma~\ref{prdwz}{\rm(ii)}]Now $\,\ve=-\frac12$.
First, $\,\hd\td\hi\ne0\,$ everywhere in $\,\,U$ since, by
Lemma~\ref{smeps}(f) and (\ref{crm}-ii), $\,\ein\,$ would vanish at points
with $\,\hd\td\hi=0$.

To prove the other claim, 
suppose that, on the contrary, the subset $\,\,U'$ of $\,\,U$ on which 
$\,d\hs\hsc\,$ and $\,d\hs\tsc\,$ are both nonzero is nonempty.

Thus, $\,\hi:U'\nh\to\bbR\smallsetminus\{0\}\,$ has the form
$\,\hi=\ha\ta+\gm\,$ with $\,\td\hh\ha=\hd\hh\ta=0\,$ and a constant $\,\gm$.
In fact, we obtain $\,\ha\,$ by restricting $\,\hi\,$ to
$\,[\hM\times\{z\}]\cap\hs U'$ with a fixed $\,z\in\tM\nh$, and as
$\,z\,$ varies, the ``af\-fine'' form  $\,\hi=\ha\ta+\widetilde\gm\,$ follows 
from Lemma~\ref{smeps}(d) and 
Remark~\ref{cncir}, where $\,\widetilde\gm\,$ may still vary with
$\,z$, but $\,\hd\widetilde\gm=0$. As our assumptions involve $\,\hM\,$ and
$\,\tM\,$ symmetrically, we similarly have 
$\,\hi=\hb\tb+\widehat\delta\,$ with 
$\,\td\hh\hb=\hd\hh\tb=\td\widehat\delta=0$.
Thus, applying $\,\partial\nh_i\w\partial\nh_a\w$ to both expressions for
$\,\hi$, in product coordinates $\,x^i\nh,x^a\nh$, and noting that
$\,\hd\td\hi\ne0$, we get $\,d\hb=p\,d\ha\,$ and
$\,d\tb=p^{-\nh1}\hn d\ta\,$ with a constant $\,p\ne0$. Writing
$\,\hb=p\ha+q_1\w$ and $\,\tb=p^{-\nh1}\ta+q_2\w$, where
$\,q_1\w,q_2\w\in\bbR$, we now get 
$\,\ha\ta+\widetilde\gm=\hi=\ha\ta+pq_2\w\ha+p^{-\nh1}\nnh q_1\w\ta
+\widehat\delta$, and so
$\,\gm=\widetilde\gm-p^{-\nh1}\nnh q_1\w\ta=pq_2\w\ha+\widehat\delta\,$ is
constant, and our claim follows if we replace $\,\ha\,$ by
$\,\ha+p^{-\nh1}\nnh q_1\w$.

Locally in $\,\,U'\nh$, Lemma \ref{twops} applied to both $\,(\hg,\ha)\,$ and
$\,(\tg,\ta)\,$ yields 
coordinates $\,x^1\nh,x^2$ for $\,\hM\,$ and $\,x^3\nh,x^4$ for
$\,\tM\,$ such that, by Lemma~\ref{smeps}(b) and \eqref{dae},
\begin{equation}\label{ddd}
\frac{\dddot\alpha}{\dot\alpha}-\frac{4\ddot\alpha\beta}{\alpha\beta+\gm}
+\frac{4\dot\alpha^2\nh\beta^2}{(\alpha\beta+\gm)^2}\,
=\,\kp+\frac{4\alpha\ddot\beta}{\alpha\beta+\gm}
-\frac{4\alpha^2\nh\dot\beta^2}{(\alpha\beta+\gm)^2},\mathrm{\ \ where\ \
}\kp=-\frac{\dddot\beta}{\dot\beta}, 
\end{equation}
Our notation here is simplified: $\,\alpha\,$ stands for $\,\ha\,$ and
$\,\beta\,$ for $\,\ta$, while $\,(\,\,)\dot{\,}$ denotes
$\,\partial/\partial x^1$ when applied to functions of $\,x^1\nh$, and
$\,\partial/\partial x^3$ for functions of $\,x^3\nh$. The  
divisions in (\ref{ddd}) are allowed: 
as $\,\hi=\alpha\beta+\gm:U\to\bbR\smallsetminus\{0\}\,$ and
$\,\hd\td\hi\ne0$.
\begin{equation}\label{non}
\dot\alpha\dot\beta(\alpha\beta+\gm)\,\ne\,0\,\mathrm{\ everywhere\ in\ 
}\,\,U'\nh.
\end{equation}
We now apply $\,\dot{\beta}^{-\nh1}(\alpha\beta+\gm)^2\partial/\partial x^3$ to
(\ref{ddd}) twice in a row, first obtaining 
\begin{equation*}
-4\gm\ddot{\alpha}+\frac{8\gm\dot{\alpha}^2\nh\beta}{\alpha\beta+\gm}
=\frac{\dot\kp}{\dot\beta}(\alpha\beta+\gm)^2\nh
-4\kp\alpha(\alpha\beta+\gm)-12\alpha^2\nh\ddot\beta
+\frac{8\alpha^3\nh\dot\beta^2}{\alpha\beta+\gm}
\end{equation*}
and then
\begin{equation*}
8\gm^2\nh\dot\alpha^2
=\,\frac{(\dot\kp/\dot\beta)\dot{\,}}{\dot\beta}(\alpha\beta+\gm)^4\nh
-\frac{2\dot\kp}{\dot\beta}\alpha(\alpha\beta+\gm)^3\nh
+8\kp\alpha^2\nh(\alpha\beta+\gm)^2\nh
+16\alpha^3\nh\ddot\beta(\alpha\beta+\gm)
-8\alpha^4\nh\dot\beta^2.
\end{equation*}
Next, applying
$\,\alpha^{-\nh2}\nh(\alpha\beta+\gm)^{-\nh2}
\partial/\partial x^3\nh$, we see that 
\begin{equation*}
[(\dot\kp/\dot\beta)/\dot\beta]\nnh^{^{\boldsymbol{\cdot}}}
(\beta+\gm/\alpha)^2\nh+2(\dot\kp/\dot\beta)\nh^{\boldsymbol{\cdot}}
(\beta+\gm/\alpha)+2\dot\kp\,=\,0\hh.
\end{equation*}
It follows that $\,\gm=0$. Otherwise, $\,\kp\,$ would be constant, as 
$\,\beta+\gm/\alpha$, being a root of a nontrivial quadratic equation with
coefficients that are functions of $\,x^3\nh$, would itself be a function of
$\,x^3\nh$, even though $\,d\ha\hn\ne0$. However, by \eqref{dae} and 
(\ref{ddd}), $\,2\kp=\tsc$, while $\,d\hs\tsc=0\,$ on $\,\,U'\nh$.

With $\,\gm=0$, (\ref{ddd}) becomes
\begin{equation}\label{ddz}
\frac{\dddot\alpha}{\dot\alpha}-\frac{4\ddot\alpha}{\alpha}
+\frac{4\dot\alpha^2}{\alpha^2}\,
=\,-\frac{\dddot\beta}{\dot\beta}+\frac{4\ddot\beta}{\beta}
-\frac{4\dot\beta^2}{\beta^2}.
\end{equation}
At the same time, Lemma~\ref{smeps}(c) and \eqref{dae} give
\begin{equation}\label{tst}
\frac{\dddot\alpha}{\dot\alpha}-\frac{2\ddot\alpha}{\alpha}\,
=\,\frac{\dddot\beta}{\dot\beta}-\frac{2\ddot\beta}{\beta}.
\end{equation}
An obvious sep\-a\-ra\-tion-of-var\-i\-ables argument shows that both sides in
(\ref{ddz}) and (\ref{tst}) are constant, so that
\begin{equation}\label{bsc}
\mathrm{i)}\hskip6pt
\frac{\dddot\alpha}{\dot\alpha}-\frac{4\ddot\alpha}{\alpha}
+\frac{4\dot\alpha^2}{\alpha^2}\,
=\,b_0\w\hh,\qquad
\mathrm{ii)}\hskip6pt
\frac{\dddot\alpha}{\dot\alpha}-\frac{2\ddot\alpha}{\alpha}\,
=\,b_1\w
\end{equation}
for some $\,b_0\w,b_1\w\in\bbR$.
Equation (\ref{bsc}-ii) states that
$\,3b_2\w=\alpha^{-2}(\ddot\alpha+b_1\w\alpha)\,$ is constant, and
multiplying the equality $\,\ddot\alpha+b_1\w\alpha=3b_2\w\alpha^2$ by
$\,2\dot\alpha\,$ we get
$\,\dot\alpha^2\nh+b_1\w\alpha^2\nh=2b_2\w\alpha^3\nh+b_3\w$ with some
constant $\,b_3\w$. 

Subtracting (\ref{bsc}-i) from (\ref{bsc}-ii), and then replacing
$\,\ddot\alpha\,$ by $\,-b_1\w\alpha+3b_2\w\alpha^2$ (see the last paragraph),
we obtain $\,4\dot\alpha^2\nh=(b_0\w-3b_1)\alpha^2\nh+6b_2\w\alpha^3\nh$. 
while, as we just saw,
$\,4\dot\alpha^2\nh=-4b_1\w\alpha^2\nh+8b_2\w\alpha^3\nh+4b_3\w$. As
$\,\alpha\,$ is nonconstant by (\ref{non}),
$\,b_1\w=-b_0\w$ and $\,b_2\w=b_3\w=0$.
Thus, $\,\ddot\alpha=b_0\w\alpha$ and
$\,\dot\alpha^2\nh=b_0\w\alpha^2\nh$. However, (\ref{ddz}) and (\ref{tst}) for
$\,\beta$ yield the analog of (\ref{bsc}) with
$\,(-b_0\w,b_1)\,$ instead of $\,(b_0\w,b_1)$, and so our conclusion 
that $\,b_1\w=-b_0\w$, which also applies to the new pair, now amounts
to $\,b_0\w=b_1\w=0$. The resulting constancy of $\,\alpha\,$ and $\,\beta\,$
contradicts (\ref{non}).
\end{proof}

\section{Proofs of Lemma~\ref{third} and Theorem~\ref{cpsfm}}\label{pm}
\setcounter{equation}{0}
Assuming that $\,\ve=-\frac12\,$ in Lemma~\ref{smeps}, we will
derive a contradiction.

Since the hypotheses of Lemma~\ref{smeps} involve the two factor manifolds
symmetrically, Lemma~\ref{prdwz} allows us to restrict our discussion to a
prod\-uct-type connected open
set $\,\,U'\nh\subseteq U$ on which either $\,\tsc\,$ is constant and 
$\,d\hh\hsc\ne0\,$ everywhere, or both $\,\hsc\,$ and $\,\tsc\,$ are constant.

In the former case, $\,\hi:U'\nh\to\bbR\smallsetminus\{0\}\,$ equals 
$\,\alpha\beta+\gm\,$ with
$\,\td\hh\alpha=\hd\hh\beta=\hd\hh\gm=0$. 
In fact, we obtain $\,\alpha\,$ by restricting $\,\hi\,$ to 
$\,[\hM\times\{z\}]\cap\hs U'$ with a fixed $\,z\in\tM\nh$, and as 
$\,z\,$ varies, Lemma~\ref{smeps}(d) and Remark~\ref{cncir} imply
the ``af\-fine'' form $\,\hi=\alpha\beta+\gm$. Since 
$\,\hd\td\hi\ne0\,$ everywhere (see Lemma~\ref{prdwz}),
\begin{equation}\label{nnc}
d\alpha\,\mathrm{\ and\ }\,d\beta\,\mathrm{\ are\ both\ nonzero\ everywhere\
in\ }\,\,U'\nh.
\end{equation}
Due to (\ref{nnc}) for $\,\alpha\,$ and Lemma~\ref{smeps}(d),  
the $\,\tg$-Hess\-i\-ans
of $\,\beta\,$ and $\,\gm\,$ are both functional multiples of the
metric $\,\tg$, and so, in view of Lemma~\ref{concr} and (\ref{dch}-a),
\begin{equation}\label{tdb}
\tde\beta=-\tsc\beta+b_1\w\mathrm{\ and\ }\,\tde\gm=-\tsc\gm+b_2\w\mathrm{\
for\ some\ constants\ }\,b_1\w,b_2\w\hh.
\end{equation}
Thus, $\,(\alpha\beta+\gm)(\tsc-\hsc)
=2(\beta\hde\hh\alpha-\alpha\tde\beta-\tde\gm)
=2[\beta\hde\hh\alpha+\tsc(\alpha\beta+\gm)-b_1\w\hn\alpha-b_2\w]\,$ by
Lemma~\ref{smeps}(c), that is, 
\[
[\alpha(\hsc+\tsc)+2\hh\hde\hh\alpha]\hs\beta\,+\,(\hsc+\tsc)\gm\,
=\,2(b_1\w\hn\alpha+b_2\w)\hh.
\]
Hence, as $\,\hsc+\tsc\ne0\,$ in Lemma~\ref{smeps}(a), along 
$\,[\hM\times\{z\}]\cap\hs U'$ for a fixed $\,z\in\tM\nh$, we get 
$\,\gm=b_3\w\beta+b_4\w$, where 
$\,b_3\w,b_4\w\in\bbR$. Replacing $\,\alpha\,$ by $\,\alpha+b_3\w$ we may now
assume that $\,\gm\,$ is constant. For this new $\,\alpha$, the above
displayed formula and (\ref{nnc}) yield
\begin{equation}\label{asp}
2\hh\hde\hh\alpha\,=\,-\alpha(\hsc+\tsc)\hh.
\end{equation}
Also, by Lemma~\ref{concr} and (\ref{dch}-b),
$\,2\tQ=-\tsc\hh(\alpha\beta+\gm+b_5\w)^2\nh+b_6\w$, 
where $\,b_5\w,b_6\w\in\bbR$. As $\,\hi=\alpha\beta+\gm$,
Lemma~\ref{smeps}(b) with $\,\ve=-\frac12\,$
and $\,\hY\nnh+\tY\nnh=\beta\hde\hh\alpha+\alpha\tde\beta\,$ gives
\[
\begin{array}{l}
(\hsc+\tsc)(\alpha\beta+\gm)^2\nh
+4(\alpha\beta+\gm)[\beta\hde\hh\alpha+b_1\w\hn\alpha+\tsc\gm
+2\hh\tsc\hh b_5\w]\\
\hskip180pt
=\,4[2\beta^2\hg(\hna\nnh\alpha,\nnh\hna\nnh\alpha)-\tsc\hh b_5^2\nh+b_6\w]
\end{array}
\]
due to (\ref{tdb}) and (\ref{asp}). 
Thus, $\,\beta\,$ is a root of a quadratic equation with coefficients
that are constant along $\,T\nh\tM\nh$, so that the leading coefficient must
equal zero:
\begin{equation}\label{lcz}
\alpha^2\nh(\hsc+\tsc)+4\alpha\hh\hde\hh\alpha
-8\hg(\hna\nnh\alpha,\nnh\hna\nnh\alpha)\,=\,0\hh,
\end{equation}
or else $\,\beta$, already constant along $\,T\nh\hM\nh$, would
be constant, contradicting (\ref{nnc}).

Locally in $\,\,U'\nh$, Lemma~\ref{smeps}(d) makes Lemma \ref{twops}
applicable to $\,\hg\,$ and $\,\alpha$, leading to
coordinates $\,x^1\nh,x^2$ for $\,\hM\,$ with (\ref{dae}), where 
$\,(\,\,)\dot{\,}$ denotes $\,\partial/\partial x^1\nh$. Setting
$\,c=\tsc/2$, we now rewrite (\ref{asp}) and (\ref{lcz}) as
\begin{equation}\label{add}
\mathrm{i)}\hskip6pt\alpha\nnh\dddot\alpha
=\dot\alpha(2\ddot\alpha+c\hh\alpha)\hh,\qquad
\mathrm{ii)}\hskip6pt\alpha^2\nnh\nh\dddot\alpha
=\dot\alpha(4\alpha\ddot\alpha-4\dot\alpha^2\nh+c\hh\alpha^2)\hh.
\end{equation}
Subtracting (\ref{add}-i) times $\,\alpha\,$ from (\ref{add}-ii) we get
$\,\alpha\ddot\alpha=2\dot\alpha^2$ which, differentiated, yields 
$\,\alpha\nnh\dddot\alpha=3\dot\alpha\ddot\alpha$, turning (\ref{add}-i), due
to (\ref{nnc}), into
$\,\ddot\alpha=c\hh\alpha$. Thus, $\,\dddot\alpha=c\hh\dot\alpha$, and
(\ref{dae}) gives $\,\hsc=2K\nnh=-2c$, contrary to the assumption that
$\,\hsc\,$ is nonconstant. 

This leaves us with the remaining case: $\,\ve=-\frac12\,$ and
$\,\hsc,\tsc\,$ are both constant. Combining parts (b) -- (d) of 
Lemma~\ref{smeps} with Lemma~\ref{bothc}, we get
$\,\hd\td\hi=0\,$ everywhere, which contradicts Lemma~\ref{prdwz},
thus completing the proof of Lemma~\ref{third}.
\begin{proof}[Proof of Theorem~\ref{cpsfm}]Due to Lemma~\ref{third}, in
Lemma~\ref{smeps}, 
$\,\hi=\ha+\ta$ with $\,\ha\nnh:\nnh\hM\nh\to\bbR\,$ and 
$\,\ta\nnh:\nnh\tM\nh\to\bbR$, cf.\ Lemma~\ref{smeps}(e) and (\ref{rco}), 
and either
\begin{equation}\label{chy}
\ve=1\hh,\quad
\hi(\hY\nnh+\tY)=2(\hQ+\tQ)\hh,\quad
2(\hY\nnh-\tY)=(\tsc-\hsc)\hi\hh,
\end{equation}
or $\,\ve=-\nh2$, and then, with $\,\ha,\ta\,$ chosen as in (\ref{rco}),
\begin{equation}\label{mul}
\begin{array}{l}
(\ha+\ta)^2\,(\hs\hsc+\tsc\hs)+2(\ha+\ta)(\hY\nnh+\tY)=4(\hQ+\tQ)\hh,\\
\hna\nh d\ha,\,\,\tna\nh d\ta\,\mathrm{\ are\ functional \
multiples\ of\ }\,\hg\,\mathrm{\ and\ }\,\tg.
\end{array}
\end{equation}
By Lemma \ref{prdwz}, locally, one of 
$\,\ha,\ta\,$ is constant and, switching them if necessary, we may assume
constancy of $\,\ta$, so that $\,\hi\,$ is constant in the
$\tM\,$ direction. Either of \eqref{chy} -- 
\eqref{mul} now implies, via separation of variables, that
$\,\tsc\,$ is constant, and (i) or (ii) in Theorem~\ref{exatt} follows, as
a consequence of (\ref{dqe}-b).
\end{proof}

\section{Proof of Theorem~\ref{cptho}}\label{cg}
\setcounter{equation}{0}
We will denote by $\,M\nnh\times\hn I\hs$ various sufficiently small
prod\-uct-type neighborhoods of $\,x\,$ in the original product manifold,
where $\,I\subseteq \bbR\,$ is an open interval with the coordinate 
$\,t\,$ and the metric $\,dt^2\nh$. We are thus given an oriented Riemannian
three-man\-i\-fold $\,(M\nh,g)\,$ and a function
$\,\hi:M\nnh\times\hn I\to\bbR\smallsetminus\{0\}$, such that 
$\,\bg=(g+dt^2)/\hi^2$ is a weakly Ein\-stein metric on
$\,M\nnh\times\hn I\nh$. 
Let us begin by showing that, under either of the assumptions (i) and (ii), 
$\,\hi\,$ has additively separated variables:
\begin{equation}\label{asv}
\hi\,=\,\alpha\,+\,\beta\,\mathrm{\ \ with\ \ }\,\alpha:M\to\bbR\,\mathrm{\ and\ }\,\beta:I\to\bbR\hh,
\end{equation}
while, for the
Ein\-stein tensor $\,\bei=\brc-\bsc\bg/4$ of $\,\bg\,$ 
and its Weyl tensor $\,\overline W\nnh\nh$, near $\,x$,
\begin{equation}\label{ewt}
\begin{array}{rl}
\mathrm{a)}&\mathrm{the\ \hs}M\mathrm{\ and\ }\nh\,I\mathrm{\hs\ factor\
distributions\ are\ }\hs\bei\hyp\mathrm{or\-thog\-o\-nal,}\\
\mathrm{b)}&\overline W\mathrm{\ \ and\ \ }\,\hs\bei\hs\mathrm{\ \ satisfy\
(\ref{ari}}\hyp\mathrm{c)\ 
with\ }\,\,c_2\w\hs=\,c_3\w\hs=\,c_4\w\hs=\,0\hh.
\end{array}
\end{equation}
More precisely, we will use (\ref{ure}) -- (\ref{eao}) below to 
establish (\ref{asv}) -- (\ref{ewt}) in each
connected component 
of the open dense set of points generic relative to the spectra of
$\,\ein\,$ and $\,\overline W^+\nnh\nnh$, in the sense of Remark~\ref{genrc}.
This will imply that (\ref{asv}) and (\ref{ewt}-a) hold everywhere, as they
are equalities. (The former reads $\,\hd\td\hi=0$, cf.\ Sect.~\ref{tl}.)

Now, if $\,u_1\w,\dots,u_4\w$ and $\,v\nh_1\w,\dots,v\nh_4\w$ are 
positive $\,\bg$-ortho\-nor\-mal frames such that
\begin{equation}\label{ure}
\begin{array}{l}
u_1\w,\dots,u_4\w\,\mathrm{\ realizes\ one\ of\ the\ conclusions\ 
\hh(\ref{ari})\hs\ in}\\
\mathrm{Theorem~\ref{weact}\ and,\nnh\ in\ particular,\nnh\ it\ 
di\-ag\-o\-nal\-izes\ }\,\hs\bei\hh,
\end{array}
\end{equation}
while the corresponding assumption in (\ref{sth}) is also satisfied, and 
\begin{equation}\label{vre}
\begin{array}{l}
v\nh_1\w,v\nh_2\w,v\nh_3\w\mathrm{,\ tangent\ to\ the\ }\hs M\hs\mathrm{\
factor,\nh\ di\-ag\-o\-nal\-ize\ the\ Ein\-stein}\\
\mathrm{tensor\ }\,\ein=\ric-\sca g/3\,\mathrm{\ of\ }\,g\,\mathrm{\ with\
some\ eigen\-valu\-e\ functions\ }\,\kp\nnh_i\w\hh,
\end{array}
\end{equation}
then, due to (\ref{wij}), for
the Weyl tensor $\,\overline W\nnh$ of $\,\bg$, and both
$\,\overline W^\pm\nnh$,
\begin{equation}\label{wsp}
2\hi^{-\nh2}\overline W^\pm\mathrm{\ have\ the\ }\,\bg\hyp\mathrm{spectrum\ 
}\,(-\kp_3\w,-\kp_2\w,-\kp_1\w)\mathrm{,\ realized\ as\ in\ (\ref{spr}).}
\end{equation}
We will derive (\ref{asv}) -- (\ref{ewt}) from the following observation.
\begin{equation}\label{eao}
\begin{array}{l}
\mathrm{In\hs\ case\hs\ (i),\hs\
}\,\,u_1\w,\dots,u_4\w\,\mathrm{\hs\ and\hs\
}\,v\nh_1\w,\dots,v\nh_4\w\,\mathrm{\hs\ arise\hs\ from\hs\ each\hs\ other}\\
\mathrm{via\ permutations,\nh\nnh\ possibly\nh\ combined\ with\ some\
sign\nh\ changes.}\\
\mathrm{In\hs\ case\hs\ (ii),\ we\ may\ assume\ that\
}\,(u_1\w,\dots,u_4\w)\hs=\hs(v\nh_1\w,\dots,v\nh_4\w)\hh.
\end{array}
\end{equation}
Namely, (\ref{wsp}) and the first part of Remark~\ref{dfeig} prove the first
part of (\ref{eao}). Assuming (i), we get $\,c_2\w=c_3\w=c_4\w=0\,$ in
(\ref{ari}), from (\ref{wsp}) and the final clause of Remark~\ref{dfeig},
which gives (\ref{ewt}-b) -- and hence (\ref{evs}) -- as (\ref{ari}-a) or
(\ref{ari}-b) would lead to 
repeated eigen\-valu\-es of $\,\overline W^\pm\nnh\nnh$. In view of 
(\ref{ure}) and (\ref{eao}), the frame $\,v\nh_1\w,\dots,v\nh_4\w$ now 
di\-ag\-o\-nal\-izes $\,\bei$. Consequently,  
(i) implies (\ref{ewt}-a), and then (\ref{rco}) yields (\ref{asv}). 
We have thus established (\ref{asv}) -- (\ref{ewt}) in the case (i). 

Next, let (ii) be satisfied. According to (\ref{sms}) and Remark~\ref{cvrsl},
our frame $\,u_1\w,\dots,u_4\w$ on $\,M\nnh\times\hn I\hs$ realizes -- due to
the assumption about $\,-\bsc/\nh12\,$ in (ii) -- both \hbox{(\ref{chb}-i)}
and (\ref{nbs}-i). More precisely, Remark~\ref{cvrsl} excludes the option 
(\ref{ari}-a) in Theorem~\ref{weact} while, in (\ref{chb}-i), 
$\,8\hh\xi=\bsc\ne0=c_2\w=c_3\w=c_4\w$, and so one has 
(\ref{ari}-c), that is, \hbox{(\ref{ewt}-b),} rather than (\ref{ari}-b).
\emph{The two spectra in\/ {\rm(\ref{ari}-c)} thus are\/}  
$\,(-\nnh\lambda,-\nnh\lambda,\lambda,\lambda)$ \emph{and\/} 
$\,(-\bsc/\nh12,\bsc/\nh24,\bsc/\nh24)$. Rearranging
$\,v\nh_1\w,v\nh_2\w,v\nh_3\w$ so
as to have $\,\kp_1\w=\kp_2\w$, we conclude, from (\ref{nbs}-i) 
and the final clause of Lemma~\ref{othzr}, that 
$\,\mathrm{span}\,(v\nh_1\w,v\nh_2\w)\,$ equals one of
$\,\mathrm{span}\,(u_1\w,u_2\w)\,$ and 
$\,\mathrm{span}\,(u_3\w,u_4\w)$, \emph{which are the
eigen\-dis\-tri\-bu\-tions of\/ $\,\bei\,$ for the eigen\-val\-ue functions\/ 
$\,-\lambda\,$ and\/} $\,\lambda$. The two italicized statements 
remain unaffected, up to a sign change in $\,\lambda$, when the pairs
$\,(u_1\w,u_2\w)\,$ and $\,(u_3\w,u_4\w)\,$ are switched and/or
independently rotated, which leads to the second part of (\ref{eao}).
By (\ref{ure}), this
gives (\ref{ewt}-a) and, again, we get (\ref{asv}) from (\ref{rco}).

We have thus shown that (\ref{asv}) -- (\ref{ewt}) hold in both cases (i)
and (ii).

One also has (\ref{eee}) with
$\,(n,p,q,\tsc,\tY\nnh\nh,\tQ,\tei,\tna\nh d\hi)
=(4,3,1,0,\ddot\beta,\dot\beta^2\nnh,0,\ddot\beta\,dt^2)$, and 
$\,\hM\nh,\tM\nh,\sca,\ein,\hg,\hna\nh,\hY\nnh\nh,\hQ,\tg\,$ 
replaced by $\,M\nh,I\nh,\bsc,\bei,g,\nabla\nh,Y\nnh\nh,Q,dt^2\nh$,
which obviously reads
\begin{equation}\label{enw}
\begin{array}{rl}
\mathrm{a)}&\bsc=\hi^2\sca+\hn6\hi(Y\nnh\nh+\ddot\beta)
-\hn12(Q+\hn\dot\beta^2)\mathrm{\ for\ }\,(Y\nnh\nh,Q)
=(\Delta\alpha,\,g(\nabla\nh\alpha,\nnh\nh\nabla\nh\alpha)),\\
\mathrm{b)}&12\hs\bei=12\hs\ein+24\hi^{-\nh1}\hn\nabla\nh d\alpha
+[\hs\sca-6\hi^{-\nh1}(Y\nh+\ddot\beta)]\hh g\,\mathrm{\ along\ }\,M\nh,\\
\mathrm{c)}&4\hh\bei\,
=\,-[\hs\sca+2\hi^{-\nh1}(Y\nnh-3\ddot\beta)]\,dt^2\hskip6pt\mathrm{\ along\
}\,I\nh\mathrm{,\ where\ }\,(\,\,)^{\boldsymbol{\cdot}}=d/dt.
\end{array}
\end{equation}
By (\ref{ewt}-b) and (\ref{ure}), $\,\bei\,$ has, the ordered spectrum 
$\,(-\nnh\nh\lambda,-\nnh\nh\lambda,\lambda,\lambda)\,$ 
or $\,(\lambda,\lambda,-\nnh\nh\lambda,-\nnh\nh\lambda)$, realized by
the frame $\,u_1\w,\dots,u_4\w$. Either of 
(i) -- (ii) implies -- cf.\ (\ref{eao}) -- that 
$\,v\nh_i\w$ in (\ref{vre}) are eigen\-vec\-tors of $\,\bei$, and we are free
to assume that $\,v\nh_1\w$ and $\,v\nh_2\w$ correspond to the same
$\,\bg$-eigen\-value function $\,\mp\lambda\,$ which, in the case (ii), or
(i), is 
obvious from (\ref{eao}) or, respectively, (\ref{eao}) coupled with 
(\ref{mod}) applied to $\,v\nh_i\w$, $\,i=1,2,3$. 
Consequently, by (\ref{enw}), $\,v\nh_i\w$, $\,i=1,2,3$, are
$\,g$-eigen\-vec\-tors of 
$\,\nabla\nh d\alpha\,$ with some eigen\-val\-ue functions $\,\delta_i\w$,
and (\ref{enw}) yields four separate expressions for $\,\lambda/\hi$, namely
\begin{equation}\label{fse}
\begin{array}{rl}
\mathrm{a)}&\mp2\lambda/\hi=(2\kp\nnh_i\w+\sca/6)\hi\,-Y\nnh-\ddot\beta
+4\delta_i\w\mathrm{,\ for\ }\,i=1,2\hh,\\
\mathrm{b)}&\pm2\lambda/\hi=(2\kp_3\w+\sca/6)\hi\,-Y\nnh-\ddot\beta
+4\delta_3\w\hh,\\
\mathrm{c)}&\pm2\lambda/\hi=-Y\nnh+3\ddot\beta-\sca\hh\hi/2\hh.
\end{array}
\end{equation}
Furthermore, for some simple-eigen\-val\-ue function $\,\kp\,$ of $\,\ein$,
\begin{equation}\label{sei}
\begin{array}{rl}
\mathrm{a)}&\bsc=6\hh\kp\hn\hi^2\,\mathrm{\ with\ 
}\,\,i\in\{1,2,3\}\,\,\mathrm{\ such\ that\ }\,\,\kp\hs=\hs\kp\nnh_i\w,\\
\mathrm{b)}&\mathrm{while,\ in\ both\ cases\ (i)\ and\ (ii),\
}\,\beta\,\mathrm{\ is\ constant.}
\end{array}
\end{equation}
Here is how (\ref{sei}-a) follows. Assuming (i), we choose
$\,\kp=\kp\nnh_i\w$ in (\ref{wsp}) so as to realize (\ref{evs}), which
we already proved in the third line after (\ref{eao}). If (ii) is satisfied, 
we let $\,\kp\,=\kp_3\w$, so that $\,\kp_1\w=\kp_2\w=-\kp\hn/2$, and (ii)
combined with (\ref{wsp}) yields (\ref{sei}-a).

In the case (i), $\,\kp_1\w\ne\kp_2\w$ by (\ref{wsp}), and to obtain 
(\ref{sei}-b), that is, the equality $\,\partial\hh\hi/\partial\hh t=0$, 
it suffices to subtract the $\,i=1\,$ and $\,i=2\,$ versions of (\ref{fse}-a).

Next, assuming (ii), and letting $\,M\nnh\times\hn I\hs$ be a small
neighborhood of a point at which $\,\dot\beta\ne0$, we will arrive at a
contradiction. As before, $\,\kp_3\w=\kp\,$ and
$\,\kp_1\w=\kp_2\w=-\kp\hn/2$. Adding (\ref{fse}-a) for 
$\,i=1\,$ to (\ref{fse}-b), we get
$\,(\kp+\sca/3)(\alpha+\beta)=2(Y\nnh+\ddot\beta)
-4(\delta_1\w+\delta_3\w)$, and $\,d/dt\,$ applied to this, via 
separation of variables, gives 
$\,\dddot\beta=p\dot\beta$, where $\,2p=\kp+\sca/3\,$ is
constant. Thus, 
$\,\ddot\beta=p\beta+q\,$ and $\,\dot\beta^2\nh=p\beta^2\nh+2q\beta+q_1\w$ 
for some $\,q,q_1\w\in\bbR$. Now (\ref{sei}-a), (\ref{asv}) and
(\ref{enw}-a) yield
\[
(\sca-6\kp)(\alpha+\beta)^2\nh+6(\alpha+\beta)(Y+p\beta+q)-12(Q+p\beta^2\nh
+2q\beta+q_1\w)=0\hh,
\]
which is a quadratic equation with coefficients 
that are constant along $\,I\nh$, imposed on $\,\beta$. The
leading coefficient must therefore vanish: $\,0=\sca-6\kp+6p-12p=-9\kp$, 
as $\,2p=\kp+\sca/3$. With (\ref{wsp}) and 
$\,(\kp_1\w,\kp_2\w,\kp_3\w)=(-\kp/2,-\kp/2,\kp)$, 
this gives $\,\overline W\nnh\nh=\hs0$, providing the required
contradiction and, consequently, proving (\ref{sei}-b).

Using (\ref{sei}-b) and (\ref{asv}), we may set $\,\hi=\alpha$, with
$\,\beta=0$.

For $\,\kp\,$ and $\,i\,$ chosen in (\ref{sei}-a), and $\,v\,$ equal to 
any positive functional multiple of $\,v\nh_i\w$, we now clearly have
(\ref{spc}-iv), as well as (\ref{spc}-iii), the latter from (\ref{sei}-a)
and (\ref{enw}-a) with $\,(\alpha,\beta)=(\hi,0)$.

According to the lines following (\ref{enw}),
$\,v\nh_1\w$ and $\,v\nh_2\w$ correspond to the same 
$\,\bg$-eigen\-value function $\,\mp\lambda$, and so does
one of the pairs 
$\,(u_1\w,u_2\w)\,$ and $\,(u_3\w,u_4\w)$. Thus, as a 
consequence of
(\ref{ewt}-b) and (\ref{ari}-c),
$\,v\nh_1\w\wedge v\nh_2\w\pm v\nh_3\w\wedge v\nh_4\w$ are
eigen\-bi\-vec\-tors of $\,\overline W^+\nnh\nnh$ for the eigen\-val\-ue
$\,-\bsc/\nh12$. Now (\ref{wsp}) gives (\ref{sei}-a) for $\,i=3$. Hence
$\,\kp=\kp_3\w$, and $\,v\,$ equals a positive function times $\,v\nh_3\w$.
With $\,\sym=2\nabla\nh d\hi+\hi\hh\ric\,$ and $\,(\alpha,\beta)=(\hi,0)$,
(\ref{enw}-b) reads 
$\,\hi\hs\bei=\sym-(\mathrm{tr}_g\w\sym)g/4$, 
and $\,\mathrm{tr}_g\w\sym=2\hs Y\nnh+\sca\hi$. Let $\,\psi=\pm\lambda/\hi$. 
Now $\,v\nh_1\w,v\nh_2\w$ and $\,v\,$ di\-ag\-o\-nal\-ize $\,\bei\,$ and
$\,\sym$, with the ordered $\,g$-spec\-trum of $\,\hi\hs\bei\,$ equal to
$\,(-\psi,-\psi,\psi)$, while (\ref{fse}-c) with $\,\beta=0\,$ yields
$\,\mathrm{tr}_g\w\sym=2\hs Y\nnh+\sca\hi=-\nh4\psi$. The resulting
equality $\,\hi\hs\bei=\sym+\psi g$, with
$\,\hi\hs\bei(v,\,\cdot\,)=\psi g(v,\,\cdot\,)$ and
$\,\hi\hs\bei(v\nh_i\w,\,\cdot\,)=-\psi g(v\nh_i\w,\,\cdot\,)\,$ for
$\,i=1,2$, gives $\,\sym(v,\,\cdot\,)=0$, proving (\ref{spc}-ii), and
$\,\sym(v\nh_i\w,\,\cdot\,)=-2\psi g(v\nh_i\w,\,\cdot\,)$. Changing the
sign of $\,\hi\,$ (and hence of $\,\sym$), if necessary, so as to replace
the last expression with $\,-2|\psi|g(v\nh_i\w,\,\cdot\,)$, and then, suitably 
normalizing $\,v$, we obtain (\ref{spc}-i).

This completes the proof of Theorem~\ref{cptho}.

\section{Properties of the new examples}\label{pn}
\setcounter{equation}{0}
In Sect.\,\ref{gk} and \ref{gc} we described examples of proper weakly 
Einstein con\-for\-mal products arising from Theorems~\ref{exatt} and
\ref{exaot}. We now show that, with just one exception, those examples are
new, that is, different from the ones known before. 

The previously known examples of proper weakly Einstein manifolds formed two 
narrow classes, consisting of the EPS space (Section~\ref{ep}) and certain 
K\"ah\-ler surfaces \cite[Sect.\,12]{derdzinski-euh-kim-park}.
The lo\-cal-homo\-thety invariant of Remark~\ref{htinv} is
given by
\[
\begin{array}{l}
\sca\hh,\,\,\,(\sca/2,0,0,-\sca/2)\hh,\,\,\,
(-\sca/\nh12,-\sca/\nh12,\sca/6)\hh,\,\,\,
(-\sca/\nh12,-\sca/\nh12,\sca/6)\,\,\,\mathrm{for\ the\ former,}\\
\sca\hh,\quad(-\nnh\lambda,-\nnh\lambda,\lambda,\lambda)\hh,\quad
(\sca/6,-\sca/\nh12,-\sca/\nh12)\hh,\quad
(-\sca/3,\sca/6,\sca/6)\quad\mathrm{for\ the\ latter,}
\end{array}
\]
at every point, with the scalar curvature 
$\,\sca\ne0\,$ and a parameter $\,\lambda\ne0$, as pointed out in
\cite[Sect.\,10 and formulae (1.6)--(1.7)]{derdzinski-park-shin}.
Thus, none of the K\"ah\-ler-sur\-face examples in
\cite{derdzinski-euh-kim-park} is, even locally, a con\-for\-mal product, 
since, by (\ref{wpm}) and (\ref{wij}), in a con\-for\-mal product 
both $\,W^\pm\nnh$ must have the same spectrum.

The EPS space is, however, a warped product, and hence a con\-for\-mal 
product; see (\ref{fet}) and (\ref{wrp}). It is realized by the
construction of Sect.\,\ref{gk}, case (i), with $\,c=0\,$ and any $\,\hi\,$
such that $\,\dot\hi=2p\,$ is a nonzero constant. In fact, (\ref{phq})
follows, without (\ref{cfl}) or (\ref{eip}), and (\ref{eli}-i) gives
$\,\chi=q+2\log|\hi|$, where $\,q\in\bbR$, so that, by
(\ref{goe}-a), $\,\bg=(g+h)/\hi^2\nh
=p^2\hn[\hh d\vt^2\nh+\hh e^{-\nh\evt}\hn d\xi^2\nh
+\hh e\hs^\evt(d\eta^2\hs+\,d\zeta^2)]$, as required in (\ref{eps}), for 
$\,(\vt,\xi)=(-\nh2\log|\hi|,\hs e^qy/p)\,$ and $\,\eta,\zeta\,$ 
with $\,h=p^2\nh(d\eta^2\hs+\,d\zeta^2)$. Also,
\begin{equation}\label{onl}
\mathrm{this\ is\ the\ only\ way\ to\ obtain\ the\ EPS\ space\ in\
Sect.\,\ref{gk},}
\end{equation}
as (\ref{lhi}) and the formula displayed above
force us to use case (i), while (\ref{ttu}-i) gives $\,c=0$,
and (b) in Sect.\,\ref{gk}, with $\,c=0\,$ and constant $\,\bsc$, implies
constancy of $\,\dot\hi$. 

By the \emph{local co\-ho\-mo\-ge\-ne\-i\-ty\/} of a Riemannian manifold
$\,(M\nh,g)\,$ we mean the minimum co\-di\-men\-sion in
$\,T\hskip-3pt_x\w\hn M\nh$, over all $\,x\in M\nh$, of the sub\-space
of $\,T\hskip-3pt_x\w\hn M\,$ consisting of the values at $\,x\,$ of all
Kil\-ling fields defined on neighborhoods of $\,x$. 

The local co\-ho\-mo\-ge\-ne\-i\-ty of the EPS space is obviously zero. 
For the K\"ah\-ler-sur\-face examples in \cite{derdzinski-euh-kim-park}
it equals one: see \cite[Remark 12.1]{derdzinski-euh-kim-park}. The 
proper weakly Einstein metrics described in Sect.\,\ref{gc} are of \emph{local 
co\-ho\-mo\-ge\-ne\-i\-ty two}, due to the func\-tion\-al-in\-de\-pend\-ence
conclusion of Remark~\ref{sglbd}, and Remark~\ref{isowp} (the latter applicable
here since they are warped products with the fibre provided by a surface of
constant Gauss\-i\-an curvature). Since the EPS space is known
\cite{arias-marco-kowalski} to be, up to lo\-cal homo\-thety, the only
locally homogeneous proper weakly Einstein manifold, (\ref{onl}) combined
with Remark~\ref{isowp} shows that all the remaining proper weakly Einstein
metrics of Sect.\,\ref{gk}, other than those of case (i) with 
$\,c=0\,$ and constant $\,\dot\hi$, have local co\-ho\-mo\-ge\-ne\-i\-ty one. 
They are thus different from the examples of Sect.\,\ref{gc}.

The lo\-cal-homo\-thety invariants (\ref{lhi}) of the metrics in
Sect.\,\ref{gk} realize, at various points, all pairs $\,(\bsc,\lambda)\,$ 
with $\,\bsc\hh\lambda\ne0$, as one sees using (b), (e) in Sect.\,\ref{gk},
Remark~\ref{hmthi}, and a suitable choice of initial data for (\ref{phq}).
The same applies to Sect.\,\ref{gc}: see Remark~\ref{sglbd} and
(\ref{ode}).

\section{Harmonic curvature}\label{hc}
\setcounter{equation}{0}
One says that a Riemannian manifold or metric has \emph{harmonic curvature\/}
when the divergence of its curvature tensor vanishes identically or,
equivalently, its Ric\-ci tensor satisfies the Co\-daz\-zi equation. Obvious
examples are provided by 
\begin{equation}\label{obv}
\begin{array}{l}
\mathrm{Einstein\ metrics,\nnh\ conformally\ flat\ metrics\ having\ constant\
scalar\ }\\
\mathrm{curvature,\hs\ and\hs\ products\hs\ of\hs\  manifolds\hs\ with\hs\
harmonic\hs\ curvature.}    
\end{array}
\end{equation}
See \cite[Sect.\,16.33]{besse}. In dimension four, some further such metrics
have the form 
\begin{equation}\label{kpc}
\begin{array}{l}
(K\nnh\hn+\hs c)^{-2}(g+h)\mathrm{,\ for\ the\ product\
}\hs\,g+h\hs\,\mathrm{\ of\ surface\ metrics\ }\hs\,g,h\\
\mathrm{with\ Gaussian\ curvatures\ }\,K\mathrm{\ and\
}\,c\mathrm{,\ where\ both\ }\,c\hs\mathrm{\ and\ the\ func}\hyp\\
\mathrm{tion\ }\,\,\gamma\hs
=\hs(K\nnh\hn+\hs c)^3\hs+\,3(K\nnh\hn+\hs c)\Delta K\,
-\,6g(\nabla\!K,\nabla\!K)\,\,\mathrm{\ are\ constant,} 
\end{array}
\end{equation}
while $\,K\nnh\hn+\hs c\ne0\,$ everywhere on the product four-man\-i\-fold
\cite[Lemma 3]{derdzinski-88}. 

As shown by DeTurck and Goldschmidt \cite{DG}, harmonic curvature implies 
\begin{equation}\label{rea}
\mathrm{real}\hyp\mathrm{analyticity\ of \ the \ metric\ in \ suitable \ local
\ coordinates.}
\end{equation}
\begin{rem}\label{unque}
According to \cite[Proposition 3(i)]{derdzinski-83}, a conformal change
leading from a product $\,\bg\,$ of two surface metric, at points where the
scalar curvature
$\,\bsc\,$ of $\,\bg\,$ is nonzero, 
to a metric with harmonic curvature, if it exists, is -- up to a constant
factor -- unique, and consists in division by $\,\bsc^2\nh$.
\end{rem}
\begin{lem}\label{lem:11.1}
No proper weakly Einstein metric arises from\/ {\rm(\ref{kpc})}.
\end{lem}
\begin{proof}
If it did, Theorem \ref{cpsfm} would give (i) or (ii) in
Theorem~\ref{exatt} with $\,\hi$ equal, locally, to a constant multiple of
$\,K\nnh\hn+\hs c\,$ (Remark~\ref{unque}). From (\ref{dqe}-b) it would now
follow that, for the constant $\gamma$ in (\ref{kpc}),
$\,\gamma-(K\nnh\hn+\hs c)^3$ equals either $\,0\,$ or $\,-3(K\nnh\hn+\hs c)^3$,
implying constancy of $\,K$, contrary to Remark~\ref{noprd}.
\end{proof}
\begin{thm}\label{nwehc}
There exists no proper weakly Einstein four-di\-men\-sion\-al Riemannian
manifold with harmonic curvature.
\end{thm}
\begin{proof}
Assuming, on the contrary, the existence of such an oriented manifold, we
will show, by
considering two separate cases, that its metric must then have the form
(\ref{obv}) or 
(\ref{kpc}), which will in turn contradict Remarks \ref{noprd} --
\ref{sezro} or, respectively, Lemma {\blue{\ref{lem:11.1}}}, and hence
complete the proof. 

In the first case, the Ricci tensor $\,\ric\,$ has four distinct eigenvalues at
some point, and so -- by (\ref{rea}) -- at all points of some dense open set
$\,\,U\nh$. For a fixed local orientation in $\,\,U\nh$, 
as shown in \cite[the lines following formula (6)]{derdzinski-88}, if an
or\-tho\-nor\-mal frame $\,u_1\w,\dots,u_4\w$ diagonalizes $\,\ric$, then
$\,W^+$ 
and $\,W^-$ are diagonalized, with the some ordered spectra for both, by the
corresponding bivectors \eqref{bas}. Remark~\ref{dfeig} now leads to
(\ref{ari}-a) or (\ref{ari}-b) in Theorem~\ref{weact}, with
$\,c_2\w=c_3\w=c_4\w=0$. Due to Remark~\ref{sezro}, the case (\ref{ari}-a),
implying con\-for\-mal flatness, is excluded. In (\ref{ari}-b), 
$\,W^{\pm}$ has a repeated eigenvalue, and
\cite[Lemma 5.4]{derdzinski-25} then shows that the metric is of type
(\ref{obv}) or (\ref{kpc}).
On the other hand, if $\,\ric$ has fewer than four eigen\-valu\-es 
at each point, \cite[Theorem 2.2(c)]{derdzinski-25} yields (\ref{obv}) or
(\ref{kpc}).
\end{proof}




\end{document}